\documentclass[reqno,12pt]{amsart}
\usepackage{color}
\usepackage{amssymb,amsmath,amsthm,amstext,amsfonts}
\usepackage{hyperref}
\usepackage{graphicx}
\usepackage{float}
\usepackage{psfrag}
\usepackage{url}
\usepackage{mathrsfs}
\usepackage{tabularx}
\usepackage{booktabs}
\usepackage[labelfont=bf,format=plain,justification=centering,singlelinecheck=false]{caption}
\usepackage{hyperref}
\hypersetup{
	colorlinks=true,
	linkcolor=blue,    
	urlcolor=blue,
	citecolor=cyan,
}
\usepackage[
backend=biber,
giveninits=true,
maxnames=5,
sorting=nyt,
url=false,
doi=false,
isbn=false
]{biblatex}

\renewbibmacro{in:}{}

\DeclareFieldFormat[article]{volume}{\mkbibbold{#1}}
\DeclareFieldFormat[article]{number}{\mkbibparens{#1}}

\renewbibmacro*{volume+number+eid}{%
	\printfield{volume}%
	\printfield{number}%
	\setunit{\addcomma\space}%
	\printfield{eid}}
\usepackage[a4paper,text={160true mm,230true mm},centering,top=35true mm,head=5true mm,headsep=2.5true mm,foot=8.5true mm]{geometry}
\usepackage{lipsum}
\newcommand\blfootnote[1]{%
  \begingroup
  \renewcommand\thefootnote{}\footnote{#1}%
  \addtocounter{footnote}{-1}%
  \endgroup
}

\makeatletter \@addtoreset{equation}{section} \makeatother

\renewcommand\thetable{\thesection.\@arabic\c@table}

\theoremstyle{plain}
\newtheorem{maintheorem}{Theorem}

\newtheorem{theorem}{Theorem}[section]
\newtheorem{proposition}{Proposition}[section]
\newtheorem{lemma}{Lemma}[section]
\newtheorem{corollary}{Corollary}[section]
\newtheorem{definition}{Definition}[section]
\newtheorem{remark}{Remark}[section]

\newcommand{\te} {\theta}       \newcommand{\Te}{\Theta}

\newcommand{\om}{\omega}        \newcommand{\Om}{\Omega}
\newcommand{\ox}{\omega,x}      \newcommand{\fo}{f_{\omega}}
\newcommand{\oy}{\omega,y}

\newcommand{\Z}{\mathbb{Z}}
\newcommand{\N}{\mathbb{N}}
\newcommand{\R}{\mathbb{R}}
\newcommand{\bP}{\mathbb{P}}

\newcommand{\graph}{\operatorname{graph}}
\newcommand{\Lip}{\operatorname{Lip}}
\newcommand{\loc}{\operatorname{loc}}

\newcommand{\wW}{\widetilde{W}}
\newcommand{\wf}{\widetilde{f}}
\newcommand{\wB}{\widetilde{B}}
\newcommand{\Lnorm}[2]{|#1|'_{#2}}

\newcommand{\cP}{\mathcal{P}}

\newcommand{\cB}{\mathcal{B}}

\newcommand{\cG}{\mathcal{G}}

\newcommand{\cQ}{\mathcal{Q}}
\newcommand{\cF}{\mathcal{F}}
\newcommand{\cT}{\mathcal{T}}
\newcommand{\cW}{\mathcal{W}}

\newcommand{\cS}{\mathcal{S}}

\allowdisplaybreaks[2]
\begin{document}

\begin{sloppypar}
	\title{Characterization of SRB Measures for Random Dynamical Systems in a Banach space}
	
	\author{Chiyi Luo}
	\address{ School of Mathematical Sciences, Center for Dynamical Systems and Differential Equations, Soochow University\\
		Suzhou 215006, Jiangsu, P.R. China}
	\email{luochiyi98@gmail.com}
	
	\author{Yun Zhao}
	\address{ School of Mathematical Sciences, Center for Dynamical Systems and Differential Equations, Soochow University\\
		Suzhou 215006, Jiangsu, P.R. China}
	\email{zhaoyun@suda.edu.cn}
    
    \thanks{This work is partially supported by The National Key Research and Development Program of China (2022YFA1005802), NSFC (12271386) and Qinglan project of Jiangsu Province.}
	
	\begin{abstract}
	This paper considers $C^2$ random dynamical systems in a Banach space, and proves that under some mild conditions, SRB measures are characterized by invariant measures satisfying Pesin's entropy formula, in which entropy is equal to the sum of positive Lyapunov exponents of the system. This can be regarded as a random version of the main result in Blumenthal and Young's paper \cite{Young17}.
\end{abstract}

	\keywords{random dynamical systems, SRB measure, entropy formula, dynamical systems on Banach spaces}
	
	\blfootnote{ 2010 {\it Mathematics Subject classification}: 37A35, 37H15, 37L55}
	\maketitle
    
	\section{Introduction and setting}\label{SEC:1}
       \subsection{Introduction}
	   In the ergodic theory of smooth dynamical systems, SRB measures are the invariant measures most compatible with volume when the volume is not preserved  by a system, see \cite{Young02} for a detailed description. 
	   Roughly speaking, they are invariant measures with smooth density along the expanding (or unstable) directions of the dynamical system. 
	   A principal result concerning SRB measures is the relationship between entropy, volume growth and Lyapunov exponents. 
	   The results show that, for a $C^2$ diffeomorphism on a finite-dimensional compact Riemannian manifold, SRB measures are characterized by invariant measures satisfying Pesin's entropy formula.
	   This formula characterizes the entropy to be the volume growth on unstable manifolds, or equivalently, the sum of positive Lyapunov exponents, see \cite{Ledrappier84, Ledrappier82, Young85} for more details.

       This kind of results has been generalized from two aspects: one is to consider relevant problems in random dynamical systems, the other is to consider similar phenomena in infinite dimensional dynamical systems. 
       It is natural and non-trivial to study related problems by combing these two aspects, i.e., random infinite dimensional dynamical systems.
	   
	   The random dynamical systems consider the composition of sequences of random maps. 
	   This is a more general set-up than the iteration of a certain map, and covers the most important examples of dynamical systems, such as products of random maps, stochastic ordinary and partial differential equations. 
	   We would like to  refer the reader to \cite{Arnold98} and \cite{Liu95} for more details about random dynamical systems. 
	   In the case of two-sided stationary random dynamical system of cocycles of $C^2$ diffeomorphisms on a smooth compact Riemannian manifold over a Polish system, Bahnm\"{u}ller and Liu \cite{Liu98} proved that an invariant measure has the SRB property if and only if it satisfies (random) Pesin's entropy formula, which generalizes  Ledrappier and  Young's result \cite{Young85} for deterministic $C^2$ diffeomorphisms. In the independent and identically distributed (i.i.d.) setting, Liu and  Qian \cite[Chapter VI]{Liu95} gave a detailed proof of the corresponding result, which had first been stated by Ledrappier and Young \cite{LeY88} with a brief proof. See also \cite[Chapter VI]{Liu95} for the proof of the corresponding result in a one-sided stationary setting.  In \cite{liu99} , Liu established Pesin's entropy formula in the setting of both i.i.d. and stationary compositions of random $C^2$ endomorphisms of a compact manifold (not necessarily invertible and maybe with singularities), provided that the sample measures
       are almost all SRB in the random diffeomorphisms case by the absolute continuity property of the unstable manifolds.

       Although it is difficult to establish ergodic theorems in smooth infinite dimensional dynamical systems since the lack of inner product and compactness, there are still many important progresses made in this direction. For example, the multiplicative ergodic theorem \cite{Oseledec68}, which plays a fundamental role in the smooth ergodic theory and nonuniform hyperbolic theory (for more details, see \cite{Pesin07}), has been extended to the infinite dimensional case. Ruelle \cite{Ruelle82} first proved the multiplicative ergodic theorem for compact linear operators on a separable Hilbert space. Ma\~{n}\'{e} \cite{M83} overcame the lack of  inner product structure and extended the multiplicative ergodic theorem to continuous compact operator cocycles on Banach spaces. Later, Thieullen \cite{Thieullen87} proved the multiplicative ergodic theorem for continuous bounded linear operator cocycles on Banach spaces. Lian and Lu \cite{Lian10} established the multiplicative ergodic theorem for strong measurable operator cocycles on a separable Banach space. Froyland \cite{Froyland13} extended it to a continuous semi-invertible cocycle on Banach spaces, and Tokman \textit{et al.} \cite{Quas12} built the multiplicative ergodic theorem for a strong measurable cocycles on a separable Banach space.
	   In addition to the multiplicative ergodic theorem, the nonuniform hyperbolic theory in infinite dimensional dynamical systems has also been studied. Lian and Lu \cite{Lian10} described the dynamics of linear random dynamical systems on a separable Banach space, and obtained a splitting of the phase space related to the Lyapunov exponents and the essential exponent. Moreover, they obtained stable and unstable local random manifolds in this case. 
       Recently, Varzaneh and Riedel \cite{Varzaneh21} obtained corresponding results for cocycles acting on measurable fields of Banach spaces, i.e., the operator is allowed to be non-invertible.  See \cite{Lucas17, Blumenthal16, Young16, Doan09, Lian20, Lian11} and references therein for more results about smooth ergodic theory in infinite dimensional dynamical systems.
	  
       From then on, similar to the finite dimensional case, it is a natural problem to give a description of SRB measures in infinite dimensional dynamical systems.  Li and Shu \cite{Li13}  proved the equivalence of SRB measures and invariant measures satisfying Pesin's entropy formula in Hilbert space for random dynamical systems, and they only \cite{Li12} proved Ruelle's inequality in the case of infinite dimensional Banach space, because there is no clear definition of volume and determination in a Banach setting, which yield  significant challenges. 
	   Recently, Blumenthal and Young \cite{Young17} overcame this difficulty and proved the equivalence of SRB measure and invariant measures satisfying Pesin's entropy formula for a certain mapping on Banach spaces. 
	   For more information on SRB measures for infinite dimensional dynamical systems, see \cite{Lian16, Lian17}.
	  
	   This paper focuses on the relationship between SRB measures and invariant measures satisfying Pesin's entropy formula in random dynamical systems on a Banach space. 
	   There are some new questions should be addressed in our case. 
	   Firstly, because attractors are random and we only have weak measurability conditions, we need to revisit the measurability of certain mappings and sets. 
	   Secondly, since the random maps are not assume to be diffeomorphisms, many problems cannot be dealt with using backward iterations. 
	   This leads to regularity issues such as the continuity of the unstable  subspaces $E^u$ are not uniform under the random mappings.
	  
	   The paper is organized as follows: 
	   Section \ref{SEC:1} presents the settings and the statements of the main results and introduces some notations. 
	 Section \ref{SEC:3} introduces the technical tools called Lyapunov norms and demonstrates the continuity of certain Oseledets splittings. 
	   Section \ref{SEC:4} discusses some properties of non-uniform hyperbolic theory, while Section \ref{SEC:5} presents the proof of the main results.
	   
	   \subsection{Settings and statements of the main results}\label{setting}
    Throughout this paper, let $(\Omega,\cF,\bP)$ be a  probability space, where $\Omega$ is a Polish space, $\bP$ is a Borel probability measure and $\cF$ is the completion of the Borel $\sigma$-algebra of $\Om$ with respect to $\bP$; 
    and let $\theta$ be an invertible ergodic transformation on $\Omega$ with respect to the probability measure $\bP$. 
	   
	   Let $(\cB,|\cdot|)$ be a separable Banach space with a Borel $\sigma$-algebra $\mathscr{B}$.
	   Denote by $B(\cB)$ the space of all bounded linear operators on $\cB$,  and let $|\cdot|$ denote the operator norm on $B(\cB)$ to avoid using too much notations.  
	   The \textit{strong operator topology} $SOT(\cB)$ on $B(\cB)$ is the topology generated by the sub-base 
	   $\{V_{x,y,\epsilon}=\{T\in B(\cB):|T(x)-y|<\epsilon\}:x,y\in\cB, \epsilon>0\}$.
	   A sequence of bounded linear operators $\{T_n\}_{n>0}$ converges to $T\in B(\cB)$ under the strong operator topology if and only if $T_n(x)\rightarrow T(x)$ as $n\rightarrow \infty$ for every $x\in \cB$.
	   The \textit{strong $\sigma$-algebra} on $B(\cB)$ is the Borel $\sigma$-algebra of the strong operator topology $SOT(\cB)$.
	   For a function $L:\Omega\rightarrow B(\cB)$, it is called \textit{strongly measurable}, if $\omega \mapsto L(\om)v$ is measurable for every $v\in \cB$. 
	   By \cite[Lemma A.4]{Quas12}, $L$ is strongly measurable if and only if it is measurable with respect to the strong $\sigma$-algebra on $B(\cB)$.
	   
	   For each $x\in \cB$,  write $\cB_x$ for the tangent space at the point $x$, which can be regarded as $\cB+\{x\}$.
	   
	   A \textit{random dynamical system} can be represented as a measure-preserving skew product map. 
	   Let $(\Omega \times \cB, \cF\otimes \mathscr{B})$ be the product measurable space. 
	   A \textit{random transformation} over $\te$ is a measurable transformation of the form
	   \begin{align*}
	   	\Phi:\Omega \times \cB &\rightarrow \Omega \times \cB \\
	   	(\omega,x)&\mapsto (\theta(\omega),f_{\omega}(x)).
	   \end{align*}
	   where $f_{\omega}$ is a map from $\cB$ to itself.
	   Note that $\Phi$ is measurable implies that $f_{\omega}:\cB \rightarrow \cB$ is measurable for each $\om\in \Om$ and the mapping $\omega \mapsto f_{\omega}(x)$ is measurable for each $x\in\cB$.
	   
	   Let $\mu$ be a $\Phi$-invariant Borel probability measure projection to $\bP$. 
	   	   For every $n\in \N$ and $\om \in \Om$, define the nonlinear cocycle $\{f^{n}_{\omega}\}$ by
	   \begin{equation*}
	   	f^{n}_{\omega}=
	   	\begin{cases}
	   		f_{\theta^{n-1}(\omega)}\circ \cdots \circ f_{\omega} & n>0;\\
	   		Id &n=0.
	   	\end{cases}
	   \end{equation*}
	 The compositions of sequences of maps $\{f^{n}_{\omega}:\omega\in \Omega,n\in\N\}$ is called the random 
  dynamical system corresponding to $(\Phi,\mu)$.
	   
	   We call a set $\Lambda \subset \Omega \times \cB$ is \textit{random compact} if the $\omega$-section $\Lambda_{\omega}:=\{x\in \cB:(\omega,x)\in \Lambda\}$ is compact for every $\om\in \Om$ and the mapping $\omega\mapsto d(x,\Lambda_{\omega})$ is measurable for every $x\in\cB$, where $d(x,\Lambda_{\omega}):=\inf_{y\in \Lambda_{\omega} }|x-y|$ is the distance between a point and a compact set. 
	   Note that a random compact subset of $\Om\times \cB$ is  measurable (see Proposition \ref{Prop:random compact}).
	   
	   \textbf{Basic assumptions} 
	   Throughout the paper, the following properties are assumed:
	   \begin{enumerate}
	   	\item[(H1)]
	   	\begin{enumerate}
	   		\item[(i)]  for each $\omega\in\Omega$, the map $f_{\omega}$ is $C^2$ Fr\'{e}chet differentiable and injective; 
	   		\item[(ii)]  for each $\omega\in\Omega$, the derivative of $f_{\omega}$ at $x\in \cB$, denote by $Df_{\omega}(x)$, is also injective;
	   		\item[(iii)] the mapping $\omega \mapsto f_{\omega}(x)$ is measurable for every $x\in \cB$.
	   	\end{enumerate} 
	   	\item[(H2)] 
	   	\begin{enumerate}
	   		\item[(i)]   the subset $A\subset \Omega\times\cB$ is random compact and satisfies $\Phi(A)=A$; 
	   		\item[(ii)]  $\mu$ is supported on $A$;
	   		\item[(iii)] there exists  $r_0>0$ such that 
	   		$$\int \log^{+}|(f_{\omega}|_{B(A_{\omega},r_0)})|_{C^2}d\bP(\om)<\infty,$$
	   		where $B(A_{\omega},r_0)=\{y\in \cB: d(y,A_{\omega})< r_0\}$.
	   	\end{enumerate} 
	   	\item[(H3)] 
	   	we assume for $\mu$-almost every $(\omega,x)$, 
	   	$$l_{\alpha}(\omega,x):=\lim_{n\rightarrow \infty}\frac{1}{n} \log |Df^n_{\omega}(x)|_{\alpha}<0.$$
	   	Here $|Df_{\omega}(x)|_{\alpha}$ is the Kuratowski measure of the set $Df_{\omega}(x)(B)$, where $B$ is the closed unit ball in $\cB$.
	   \end{enumerate} 

	    Some  \textbf{remarks} on the basic assumptions are given in order. 
	   	Condition (H2) implies that $f_{\omega}(A_{\omega})=A_{\theta(\omega)}$ for every $\om\in \Om$. 
	   	By (H1), one can define $(f_{\omega}|_{A_{\om}})^{-1}$ on $A_{\theta(\omega)}$, which is also continuous. 
	   	Therefore, for every $n>0$, we define the map $f_{\omega}^{-n}$ on $A_{\omega}$ as follows:
	   	$$f_{\omega}^{-n}:=(f_{\theta^{-1}\omega}\circ \cdots \circ f_{\theta^{-n}\omega}|A_{\theta^{-n}\omega})^{-1},$$
	   	and we denote $\Phi^{-n}(\omega,x):=(\theta^{-n}\omega,f_{\omega}^{-n}(x) )$ for every $(\omega,x)\in A$. 
	   	Since $\Phi$ is measurable, by Lusin's theorem, there exists an increasing sequence $\{A_n\}_{ n\in\N}$ of compact subsets of $A$ 
            such that $\mu(\cup_{n\geq 0}A_n)=1$ and $\Phi|_{A_n}$ is continuous for each $n$. 
	   	Let $B_n:=\Phi(A_n), n\geq 0$, which also is a compact subset of $A$, and satisfies $\mu(\cup_{n\geq 0}B_n)=1$. 
	   	Note that $\Phi|_{A_n}$ is a homeomorphism from $A_n$ to $B_n$ since $\Phi|_{A_n}$ is injective and continuous.
	   	This implies that $(\Phi|_{A_n})^{-1}$ is continuous on $B_n$ for each $n\geq 0$.
	   	Hence, by applying Lusin's theorem again, we conclude that $(\Phi|_{A})^{-1}$ is also measurable on $A$.
	   	
	   	The integrand in  (iii) of (H2) is measurable by (iii) of (H1)  and (i) of (H2). 
	   	To demonstrate this, let $\{x_n\}_{n>0}$ be a countable dense subset of $\cB$, we have
	   	$$\sup_{x\in B(A_{\omega},r_0)}| \fo(x)|=\sup_{n>0}|f_{\omega}(x_n)|\cdot \chi_{B(A_{\omega},r_0)}(x_n),$$
	   	where $\chi_{E}$ is the  characteristic function of the set $E$. 
	   	Since $A$ is random compact, 
            the set $\{\omega: d(x_n,A_{\omega})<r_0\}$ is measurable for each $n$. 
	   	This implies that $\omega \mapsto \chi_{B(A_{\omega},r_0)}(x_n)$ is measurable. 
	   	By assumption  (iii) of (H1), $\om\mapsto |f_{\omega}(x_n)|$ is also measurable for each $n$. 
	   	Hence, the $C^0$ norm is measurable. 
	   	For the $C^1$ norm, fix $x\in \cB$ and $u\in \cB_x$. By Fr\'{e}chet differentiability, we have 
	   	$$Df_{\omega}(x)u=\lim_{k\rightarrow \infty} \dfrac{\fo(x+t_ku)-{\fo(x)}}{t_k},$$
	   	where $\{t_k\}_{k>0}$ is a sequence of positive numbers converges to zero. 
	   	Then, by  (iii) of (H1), the map $\omega \mapsto Df_{\omega}(x)u$ is measurable for each $x\in \cB$ and each $u\in \cB$.
	   	Since $\omega\mapsto Df_{\omega}(x_n)u$ is measurable for any $u\in \cB_{x_n}$ and $\cB$ is separable, we can show that $\omega 
            \mapsto |Df_{\omega}(x_n)|$ is measurable for every $n$. Consequently, the $C^1$ norm is measurable.
	   	The measurability of the $C^2$ norm can be demonstrated in the same fashion.
	   	
	   	The limit in condition (H3) exists by the sub-additive ergodic theorem. Moreover, it guarantees that there are at most finitely many non-negative Lyapunov exponents of $(\Phi,\mu)$ and, each of finite multiplicity.
	   
	   \textbf{Two other relevant conditions}
	   \begin{enumerate}
	   	\item[(H4)]	$(\Phi,\mu)$ has no zero Lyapunov exponents.
	   	\item[(H5)] the set $A$ in (H2) is \textit{visible}, means that there exists a compact subset $E\subset \cB$ with finite box-counting dimension such that $\{\omega:A_{\omega}\subset E\}$ has positive $\bP$-measure.
	   \end{enumerate}
	   We state below a provisional definition of SRB measure, and the formal one will be given in Section \ref{SEC:5}.
	  See Section \ref{SEC:Induced volumes} for the definition of induced volume.
    
	   \begin{definition}\label{Def:SRB}
	   	We say $\mu$ is a random SRB measure if 
	   	(i)  it has a positive exponent for $\mu$-almost every $(\omega,x)$, and 
	   	(ii) the conditional measures of $\mu$ on the unstable manifolds are absolutely continuous with respect to the induced volume.
	   \end{definition}
	   
	   By the multiplicative ergodic theorem (see Section \ref{SEC:3}), for $\mu$-almost every points we let $\lambda_1(\omega,x)>\lambda_2(\omega,x)>\cdots$ with multiplicities $m_1(\omega,x),m_2(\omega,x),\cdots$ be the distinct Lyapunov exponents of $(\Phi,\mu)$. 
	   Let $a^{+}=\max\{a,0\}$, and let $h(\mu|\bP)$ be the conditional entropy of $(\Phi,\mu)$, the detailed definition  will be given in Section \ref{SEC:2.1}.
	   
	   The main results of this paper are the following two theorems: one gives the condition for which Pesin's entropy formula hold in random infinite dimensional dynamical systems; the other one gives the characterization of a random SRB measure under suitable conditions. 
	  
    \begin{maintheorem}\label{Thm:A}
	   Suppose $(\Phi,\mu)$ satisfies (H1)-(H3) above, and assume $\mu$ is a random SRB measure, then we have
	   \begin{equation}\label{eq:entropy formula}
	   		h(\mu|\bP)=\int \sum_{i} m_i(\omega,x)\lambda^{+}_{i}(\omega,x)d\mu.
	   \end{equation}
	   \end{maintheorem}
   
	   \begin{maintheorem}\label{Thm:B}
	   Suppose $(\Phi,\mu)$ satisfies (H1)-(H5). 
	   If $\lambda_1>0$ for $\mu$-almost every point, and the entropy formula \eqref{eq:entropy formula} holds, then $\mu$ is a random SRB measure.
	   \end{maintheorem}
	     
	   \textbf{Guide to notation} The following notations are used throughout the paper.
	   \begin{itemize}
	   	\item $\Om$ is a Polish space, $\cF$ is the completion of the Borel $\sigma$-algebra of $\Om$ with respect to the probability measure $\bP$.
	   	\item $\cB$ is a separable Banach space with norm $|\cdot|$, and $\mathscr{B}$ is the Borel $\sigma$-algebra of $\cB$. 
	   	\item $\cF\otimes \mathscr{B}$ is the product $\sigma$-algebra of the product space $\Om \times \cB$.
	   	\item $E_{\om}=\{x:(\ox)\in E\}$ is the $\om$-section of the subset $E$ of $\Om \times \cB$.
	   	\item $B(x,r)$ is the closed ball of radius $r>0$ with the center $x\in\cB$.
	   	\item $B(\cB)$ is the space of all bounded linear operators on $\cB$, we also use $|\cdot|$ to denote the operator norm, and denote by $SOT(\cB)$ the strong operator topology on $\cB$. 
	   	\item $\cG(\cB)$ is the Grassmannian of all closed subspaces of $\cB$, $\cG_{k}(\cB)$ and $\cG^{k}(\cB)$ are the subsets of all  $k$-dimensional subspaces and $k$-codimensional closed subspaces respectively.
	   	\item $d_{H}$ is the Hausdorff distance on $\cG(\cB)$ and $G(\cdot,\cdot)$ is the gap function, defined in Section \ref{SEC:Gaps}.
	   	\item for $E\in \cG_{k}(\cB)$ and $T\in B(\cB)$, $m_{E}$ is the induced volume and $\det(T|E)$ is the determinate function, which are defined in Section \ref{SEC:Induced volumes}.
	   	\item $\nu_{W}$ is the induced volume on a $C^1$ embedded (or injectively immersed) finite-dimensional submanifold $W\subset \cB$, defined in the end of Section \ref{SEC:Induced volumes}. 
	   	\item $\cB_x$ is the tangent space at point $x\in \cB$, and the exponential map $\exp_x:\cB_x\rightarrow \cB$ is the affine linear map sending $v\in \cB_x$ to $v+x\in \cB$.
	   	\item $\Gamma$ is a $\Phi$-invariant and full $\mu$-measure subset of $\Om\times \cB$ as in Theorem \ref{MET}.
	   	\item for $(\ox)\in \Gamma$,  $\cB_x=E^{u}(\ox)\oplus E^{cs}(\ox)$ is the decomposition of $\cB_x$ in the beginning of Section \ref{SEC:Lyapunov norms}, and $\pi^{\tau}_{(\ox)}$ denotes the projection from $\cB_{x}$ onto $E^{\tau}(\ox)$ via the splitting for $\tau=u,cs$.
	   	\item for $(\ox)\in \Gamma$, $\Lnorm{\cdot}{(\ox)}$ is the Lyapunov norm  defined in Section \ref{SEC:Lyapunov norms}.
	   	\item $\Lnorm{\cdot}{}$ is the operator norm of bounded linear operators with respect to the Lyapunov norms. 
	   	      For example, let  $L:(\cB_{x}, \Lnorm{\cdot}{(\ox)}) \rightarrow (\cB_{f_{\om}(x)}, \Lnorm{\cdot}{\Phi(\ox)})$ and $T:E^{u}(\ox) \rightarrow E^{cs}(\ox)$ be two bounded linear operators, then $\Lnorm{L}{}=\sup\{\Lnorm{Lv}{\Phi(\ox)}:v\in \cB_{x}, \Lnorm{v}{(\ox)}=1\}$ and $\Lnorm{T}{}=\sup\{\Lnorm{Tu}{(\ox)}:u\in E^{u}(\ox), \Lnorm{u}{(\ox)}=1\}$.
	   	\item  $\lambda$, $\varepsilon_0$ are constants as in Section \ref{SEC:Lyapunov norms}, and  $\delta_1$ is a constant as in \eqref{eq:delta1}.
	   	\item for $(\ox)\in \Gamma$, $\wf_{(\ox)}=\exp^{-1}_{f_{\omega}x}\circ f_{\om}\circ \exp_{x}$ is  the connecting map from $\cB_x$ to $\cB_{f_\omega x}$.
	   	\item the function $\ell:\Gamma\rightarrow [1,\infty)$ and the constant $\varepsilon_1$ are the same as in \eqref{eq:functionl}.
	   	\item the subset $\Gamma_{\ell}$ is defined in the beginning of Section \ref{SEC:CCS}.
	   	\item for $r>0$, $\tau=u,cs$ and $(\ox)\in \Gamma$, $\widetilde{B}_{(\omega,x)}(r)=\{v\in \cB_x:\Lnorm{v}{(\ox)}\leq r\}$,
	   	      $\widetilde{B}^{\tau}_{(\omega,x)}(r)=\{v\in E^{\tau}(\ox):\Lnorm{v}{(\ox)}\leq r\}$ and $B^{\tau}_{(\omega,x)}(r)=\{v\in E^{\tau}(\ox):|v|\leq r\}$.
	   	\item for $(\ox)\in \Gamma$, the function $g_{(\ox)}$ is  as in Theorem \ref{Thm:Unstable}, the (local) unstable manifold $W^{u}(\ox)$ is defined in   Section \ref{SEC:4.1}, and we denote by $\nu_{(\ox)}$ the induced volume on $W^{u}(\ox)$.
	   	\item for $\om\in \Om$, $n\in \N$ and $y\in \cB$, $f^{-n}_{\om}(y)$ denotes the unique point satisfies that   $f^{n}_{\te^{-n}\om}(f^{-n}_{\om}(y))=y$, if such a point exists. 
	   	\item for $(\ox)\in \Gamma_{\ell}$, $\cS(\ox)$ is the unstable stack of local unstable manifolds as in \eqref{eq:stack}. 
	   	      The random unstable stack $\cS$ is defined in Section \ref{SEC:Random US}.
	   \end{itemize}
	\section{Preliminaries}\label{SEC:Pre}
\subsection{Conditional entropy}\label{SEC:2.1}
      In this subsection, we recall the definition of conditional entropy, see \cite[Chapter 0]{Liu95} for more details.
      Let $(\Phi,\mu)$ be given as in Section \ref{setting}. 
      The partition $\eta_0:=\{\{\omega\}\times \cB\}_{\omega\in \Omega}$ is a measurable partition of $\Om \times \cB$ since $\Om$ is a Polish space. 
      
      For every measurable partition $\eta$ of $\Omega \times \cB$ with $H_{\mu}(\eta|\eta_0)<\infty$, where $H_\mu(\eta_1|\eta_2)$ is the conditional entropy of $\eta_1$ given by $\eta_2$. 
     The conditional entropy of $\eta$ is given as follows: 
	   \begin{align*}	h(\mu|\bP,\eta):=\lim_{n\rightarrow 
          \infty} \frac{1}{n} H_{\mu}\Big(\bigvee_{i=0}^{n-1}\Phi^{-i}\eta |\eta_0\Big).
	   \end{align*}    
      The conditional entropy of $(\Phi,\mu)$ is given by
	   $$h(\mu|\bP):=\sup\Big\{h(\mu|\bP,\eta):\eta \ \text{is a 
      measurable partition of}\  \Omega \times \cB \ \text{with}\  \ H_{\mu}(\eta|\eta_0)<\infty\Big\}.$$
          
      The definition of $h(\mu|\bP,\eta)$ can be generalized for any measurable partition of $\Omega \times \cB$ as follows. 
      For a measurable partition $\eta$ of $\Omega \times \cB$, let 
          $$h(\mu|\bP,\Phi,\eta):=H_{\mu}\Big( \eta |\bigvee_{i=1}^{\infty}\Phi^{-i}\eta\vee \eta_0\Big).$$
      If $H_{\mu}(\eta|\eta_0)<\infty$, then $h(\mu|\bP,\Phi,\eta)=h(\mu|\bP,\eta)$. Let 
          $$h(\mu|\bP,\Phi):=\sup\Big\{h(\mu|\bP,\Phi,\eta): \eta \ \text{is a measurable partition of} \  \Omega \times \cB\Big\}.$$
      By Theorem 5.1 in \cite{Liu95}, one has $h(\mu|\bP,\Phi)=h(\mu|\bP)$.
      Since $(\Phi|_A)^{-1}$ is measurable and $\mu$ is supported on $A$,
      by Theorem 4.3 in \cite{Liu95}, we have that
          \begin{equation*}
              h(\mu|\bP,\Phi)=h(\mu|\bP)=h(\mu|\bP,(\Phi|_A)^{-1}).
          \end{equation*}

\subsection{Measurability and random compactness}
This subsection collects some results about measurability and random compactness that will be used in this paper.

Let $(\Om,\cF)$ be a measurable space, where $\cF$ is a $\sigma$-algebra of subsets of $\Om$. 
For a Polish space $X$ with a complete metric $d$, denote by $\mathscr{B}(X)$ the Borel $\sigma$-algebra of $X$,
 and let $\cF \otimes \mathscr{B}(X)$ be the product $\sigma$-algebra of $\Om \times X$.

The following lemma gives a criterion for the measurability of a given map. See \cite[Lemma 1.1]{Crauel02} for the detailed proofs. 
\begin{lemma}\label{Lem:MPolish}
	Let $Y$ be a metric space.  Suppose that $(\Om,\cF)$ is a measurable space, $X$ is a separable metric space and
	$g: \Om \times X \rightarrow Y$ satisfies
	\begin{enumerate}
		\item[(a)] $\om \mapsto g(\om,x)$ is measurable for each $x\in X$;
		\item[(b)] $x \mapsto g(\om,x)$ is continuous for each $\om\in \Om$.
	\end{enumerate} 
    Then, $g$ is a measurable map from $\Om \times X$ to $Y$.
\end{lemma}

Let $\bP$ be a probability measure on $\Om$. A probability space $(\Omega,\cF,\bP)$ is \textit{complete} if any subset of a measurable set with zero measure is measurable (and has zero measure). See \cite[Theorem III.23]{Castaing77}  for the proof of the following result. 

\begin{proposition}\label{Prop:projec P}
	Let $(\Omega,\cF,\bP)$ be a complete probability space and $X$  a Polish space. 
	Let $\cF \otimes \mathscr{B}(X)$ be the product $\sigma$-algebra of $\Om \times X$.
	Denote by $\pi_{\Om}$ the canonical projection from $\Om \times X$ onto $\Om$. Then $\pi_{\Om}(E)\in \cF$ for every $E\in \cF \otimes \mathscr{B}(X)$.
\end{proposition}

For a closed subset $Z\subset X$ and for $x\in X$, put $d(x,Z):=\inf_{y\in Z} d(x,y)$.
Recall that a subset $\Lambda \subset \Om \times X$ is \textit{random compact} if 
$\Lambda_{\omega}:=\{x\in \cB:(\omega,x)\in \Lambda\}$ 
is compact for every $\om\in \Om$, 
and the mapping $\om \mapsto d(x,\Lambda_{\omega})$ is measurable for every $x\in X$. 

The following proposition \cite[Theorem III.30]{Castaing77} (also see \cite[Proposition 2.4]{Crauel02}) 
gives characterizations of random compact subsets. 
\begin{proposition}\label{Prop:random compact}
	Let $(\Omega,\cF,\bP)$ be a probability space and $X$  a Polish space, and let $K$  be a subset of $\Omega \times X$ with compact $\om$-section $K_{\om}$ for every $\om\in \Om$.
	Consider the following conditions
	\begin{enumerate}
		\item[(i)]   $\{\om:K_{\om}\cap U\neq \emptyset\}\in \cF$ for every open set $U\subset X$;
		\item[(ii)]  $\om \mapsto d(x,K_{\omega})$ is measurable for every $x\in X$;
		\item[(iii)] $K$ is a measurable subset of $\Omega \times X$, i.e.,  $K\in \cF \otimes \mathscr{B}(X)$.
	\end{enumerate} 
    Then, (i) and (ii) are equivalent, and either of them implies (iii). Furthermore, (i), (ii) and (iii) are equivalent if $(\Omega,\cF,\bP)$ is a complete probability space. 
\end{proposition}

The following proposition \cite[Proposition 2.15]{Crauel02} gives a version of tightness for random compact sets.
\begin{proposition}\label{Prop:Trandom compact}
	Let $(\Omega,\cF,\bP)$ be a probability space and $X$  a Polish space. 
	Suppose that $K \subset \Omega\times X$ is a random compact subset. 
	Then, for any $\epsilon>0$ there exists a (non-random) compact subset $K_{\epsilon}\subset X$ such that 
	$\bP\{\omega:K_{\omega}\subset K_{\epsilon}\}>1-\epsilon$. 
\end{proposition}
 
\subsection{Gaps and distance between closed linear subspaces}\label{SEC:Gaps}
In this subsection, we recall the definitions of gap and distance in a Banach space, see  \cite[Chapter 4.2]{Kato95}  or \cite[Section 3]{Young17} for more details.

Let ($\cB,|\cdot|$) be a separable Banach space, and let $\cG(\cB)$ be the Grassmannian of all closed subspaces of the Banach space $\cB$. 
We consider the metric topology on $\cG(\cB)$ defined by the Hausdorff distance $d_H$ between unit spheres: 
for non-trivial subspaces $E,F\in \cG(\cB)$, let
\begin{align*}
	d_{H}(E,F)=\max\Big\{\sup_{v\in S_E} d(v,S_F),\sup_{u\in S_F} d(u,S_E)\Big\},
\end{align*}
where $S_*=\{v\in *:|v|=1\}$ and $d(v,S_*)=\inf_{u\in S_*} |v-u|$, $*=E,\, F$.
For every $k\in\N$, we denote by $\cG_{k}(\cB)$ the collection of $k$-dimensional subspaces of $\cB$, and denote by $\cG^{k}(\cB)$ the collection of $k$-codimensional closed subspaces of $\cB$.

\begin{lemma}[\cite{Quas12}  Lemma B.11]\label{Lem:SPGk}
	Suppose that ($\cB,|\cdot|$) is a separable Banach space, then ($\cG_{k}(\cB),d_H$) is a separable metric space.
\end{lemma}

 For every two non-trivial subspaces $E,F\in \cG(\cB)$,  the \textit{gap} $G(E,F)$ is defined as follows: 
\begin{align*}
G(E,F):=\sup_{v\in S_{E}} d(v,F)
\end{align*}
where $d(v,F)=\inf_{u\in F}|v-u|$. 
Note that $G(\cdot,\cdot)$ is not symmetric, 
the relations between $G(E,F)$ and $G(F,E)$ given by  the following result.

\begin{lemma}[\cite{Blumenthal16} Lemma 2.6]\label{Lem:gap}
	Let $k>0$, $E,F\in \cG_{k}(\cB)$ or  $E,F\in \cG^{k}(\cB)$. If $G(E,F)<1/k$, then we have that 
	$$G(F,E)\leq \dfrac{kG(E,F)}{1-kG(E,F)}.$$
\end{lemma}

For non-trivial subspaces $E,F\in \cG(\cB)$, let further that 
\begin{align*}
\hat{G}(E,F)=\max\{G(E,F),G(F,E)\}.
\end{align*} 

On Hilbert spaces, $\hat{G}$ is a metric and coincides with the operator norm of the difference between orthogonal projections. 
$\hat{G}$ is not a metric on Banach spaces,   since it does not satisfy the triangle inequality in general \cite[Chapter 4.2]{Kato95}.
However, $\hat{G}$ and the distance $d_{H}$ satisfy the following properties (see \cite[Chapter 4]{Kato95} or \cite[Lemma 2.4]{Young16}):
\begin{equation}\label{eq:gap}
\hat{G}(E,F) \leq d_{H}(E,F) \leq  2\hat{G}(E,F).
\end{equation}

Consider a splitting $\cB=E\oplus F$ for $E,F\in \cG(\cB)$, denote by $\pi_{E//F}$ and $\pi_{F//E}$ the projections onto $E$ and $F$ via the splitting respectively.
Then $\pi_{E//F}$ and $\pi_{F//E}$ are bounded linear operators by the closed graph theorem.
\begin{lemma}\label{Lem:ProDh}
	Let $\cB=E\oplus F$ for $E,F\in \cG(\cB)$. 
	If a sequence of closed subspaces $\{E_n\}_{n\in \N}\subset \cG(\cB)$ satisfy that $d_{H}(E_n,E)\rightarrow 0$ as $n\rightarrow \infty$. 
	Then, $|\pi_{E//F}|_{E_n}|\rightarrow 1$ and $|\pi_{F//E}|_{E_n}|\rightarrow 0$ as $n\rightarrow \infty$.
\end{lemma}
\begin{proof}
	We first prove that $|(\pi_{F//E})|_{E_n}|\rightarrow 0$ as $n\rightarrow \infty$. 
	Fix a unit vector $v\in E_n$. 
	By the definition of $d_H$, for every $c>1$ arbitrarily close to 1, there exists a unit vector $v'\in E$ such that $|v-v'|\leq cd_H(E,E_n)$. 
	Then,
	$$|\pi_{F//E} v|=|\pi_{F//E} (v-v')|\leq |\pi_{F//E}|\cdot |v-v'|\leq c|\pi_{F//E}|\cdot d_H(E,E_n).$$
	Therefore, we get 
	$$|\pi_{F//E}|_{E_n}|\leq  |\pi_{F//E}|\cdot d_H(E,E_n).$$
    This yields that $|\pi_{F//E}|_{E_n}|\rightarrow 0$ as $n\rightarrow \infty$.
	
	Note that $Id|_{E_n}=\pi_{E//F}|_{E_n}+\pi_{F//E}|_{E_n}$. This implies that
	$$1-|\pi_{F//E}|_{E_n}|\leq |\pi_{E//F}|_{E_n}| \leq 1+|\pi_{F//E}|_{E_n}|.$$
	Hence, we have that $|\pi_{E//F}|_{E_n}|\rightarrow 1$ as $n\rightarrow \infty$. 
	This completes the proof.
\end{proof}

\begin{lemma}\label{Lem:CdhA}
   	Let $k>0$ and $\cB=E\oplus F$ for $E\in \cG_k(\cB), F\in \cG^{k}(\cB)$. 
   	Assume that $T$ is a linear operator from $E$ into $F$, and $\{T_n:E\rightarrow F\}_{n\geq 0}$ is a sequence of linear operators. 
   	Denote by $E_T=(Id+T)E$ and $E_n=(Id+T_n)E$. 
   	If $d_{H}(E_T,E_n)\rightarrow 0$ as $n\rightarrow \infty$, then $|T_n-T|\rightarrow 0$ as $n\rightarrow \infty$.
\end{lemma}
\begin{proof}
	For each $n\in \N$, denote $L_n:E_T \rightarrow F$ by $L_n:= (T_n-T)\circ \pi_{E//F}$.
	Then, $E_n=(Id+L_n)E_T$ for every $n\in \N$. 
	Note that $\cB=E_T \oplus F$.
    Fix a unit vector $v \in E_T$, let $u=v+L_n v\in E_n$. 
	For every $c>1$ arbitrarily close to 1, there exists a unit vector $u'\in E_T$ such that $|u'-(u/|u|)| \leq cd_H(E_T,E_n)$. 
    Then,
	$$|L_n v|= |u| \cdot |\pi_{F//E_T}(u'-(u/|u|))|\leq (1+|L_n|) \cdot |\pi_{F//E_T}| \cdot cd_H(E_T,E_n).$$
	Therefore, we have
	$$ \dfrac{|L_n|}{|\pi_{F//E_T}|\cdot(1+|L_n|)}\leq d_H(E_T,E_n).$$
	Since $d_{H}(E_T,E_n)\rightarrow 0$, one has $|L_n|\rightarrow 0$ as $n\rightarrow \infty$.
	
	Note that $\pi_{E//F}|_{E_{T}}$ is a linear isomorphism, the minimum norm $m(\pi_{E//F}|_{E_{T}}):=\min \{|\pi_{E//F} v|:v\in E_T, |v|=1\}$ is strictly larger than $0$. 
	Then,
    $$|T_n-T|\leq |L_n|\cdot (m(\pi_{E//F}|_{E_{T}}))^{-1}.$$
	Hence, one has $|T_n-T|\rightarrow 0$ as $n\rightarrow \infty$.
\end{proof}

\subsection{Induced volumes}\label{SEC:Induced volumes}
We will give the detailed definition of induced volume in Definition \ref{Def:SRB}, which was introduced by Blumenthal and  Young  \cite[Section 3]{Young17}.

We first introduce the notion of induced volumes on finite-dimensional subspaces.
Let $E\subset \cB$ be a finite-dimensional subspace,  denote by $m_E$ the unique Haar measure on $E$ for which
$$m_E\{u\in E:|u|\leq 1\}=w_k$$ 
where $k=\dim E$ and $w_k$ is the volume of the Euclidean unit ball in $\R^{k}$. 
Note that $m_E$ depends on the choice of the norm $|\cdot|$ on $\cB$.
A measure of the form $m_E$  is called an \textit{induced volume} on $E$ from $(\cB,|\cdot|)$.

The notion of \textit{determinant} is defined by induced volume naturally.
For two Banach spaces $(\cB_1,|\cdot|_1)$, $(\cB_2,|\cdot|_2)$ and a bounded linear operator $A:\cB_1\rightarrow \cB_2$.
Let $E\subset \cB_1$ be a finite-dimensional subspace of $\cB_1$, the determinant of $A|_E$ is defined as follows:
\begin{equation*}
	\det(A|_{E}):=
	\begin{cases}
		\dfrac{m_{A(E)}(A(B_E))}{m_E(B_E)} &\dim A(E)=\dim E,\\
		0                         &\text{otherwise},
	\end{cases}	    	
\end{equation*}
where $B_E=\{v\in E: |v|_1\leq 1\}$. 
Note that $\det(A|_{E})$ depends on the choices of the norms $|\cdot|_1$, $|\cdot|_2$ on $\cB_1$ and $\cB_2$ respectively. 
Some basic properties about determinant are given in the following.

\begin{proposition}[\cite{Young17} Section 3]\label{Prop:Det}
	Let $E,F,G$ be $k$-dimensional subspaces of $\cB$, and let $A,B:\cB\rightarrow\cB$ be bounded linear maps for which $A(E)\subset F$, $B(F)\subset G$. 
	Then:
	\begin{enumerate}
		\item[(1)] $m_F(A(S))=\det(A|_E)\cdot m_E(S)$ for every Borel set $S\subset E$;
		\item[(2)] $\det(BA|_E)=\det(B|_F)\cdot\det(A|_E)$;
		\item[(3)] for each unit vector $v\in E$, there exists a unit basis $\{v_i\}^{k}_{i=1}$ of $E$ (independent of $A$) such that $v_1=v$ and
		           $$\det(A|_E)\leq k^{\frac{k}{2}} \prod_{i=1}^{k}|Av_i|;$$
		\item[(4)] if $A:E\rightarrow F$ is invertible, and there are splittings $E=E_1\oplus E_2, F=F_1\oplus F_2$  $AE_1=F_1$, $AE_2=F_2$, 
		           then there exists a constant $C_k$ (depends only on $k$) such that
		           $$\dfrac{1}{(|\pi_{F_1//F_2}|^q C_k)} \leq \dfrac{\det(A|_E)}{\det(A|_{E_1})\det(A|_{E_2})}\leq |\pi_{E_1//E_2}|^q C_k,$$
		           where $\pi_{E_1//E_2}$ is the projection from $E_1\oplus E_2$ onto $E_1$ and $q=\dim E_1$.
	\end{enumerate}
\end{proposition}

\begin{proposition}[\cite{Young17} Proposition 3.15]\label{Prop:CofDet}
	Let $(\cB,|\cdot|)$ be a Banach space.
	For any $k\geq 1$ and any $M>1$, there exist $L,\delta$ such that, 
	for every two injective bounded linear operators $A_1,A_2:\cB \rightarrow \cB$ and two $k$-dimensional subspaces $E_1,E_2$ of $\cB$ satisfying that 
	\begin{align*}
		\max_{j=1,2}\{|A_j|, |(A_j|_{E_j})^{-1}|\}\leq M\quad \text{and} \quad \max\{ |A_1-A_2|&, d_H(E_1,E_2)\}\leq \delta,
	\end{align*}
	then we have that
	$$\Big|\log \dfrac{\det(A_1|_{E_1})}{\det(A_2|_{E_2})}\Big|\leq L(|A_1-A_2|+d_H(E_1,E_2)).$$
\end{proposition}

\begin{remark}\label{Rem:SOTC}
	Fix a finite-dimensional subspace $E\subset \cB$. 
 Using the same arguments as in the proof of Proposition 3.5 in \cite{Young17}, one can show that the function $T\mapsto \det(T|_E)$ is continuous on the space of all injective bounded linear operators under the strong operator topology.
\end{remark}

Let $U\subset \R^{k}$ be an open set, $\phi:U\rightarrow \cB$ be a $C^{1}$ Fr\'{e}chet embedding and $W:=\phi(U)$. 
For each Borel set $V\subset W$,  define the \textit{induced volume} $\nu_{W}$ by
$$\nu_{W}(V)=\int_{\phi^{-1}(V)}\det(D\phi(y))dy$$
where $\det(D\phi(y))$ is the determinant of the linear map $D\phi(y)$ from $\R^{k}$ with Euclidean volume onto the tangent space $T_{\phi(y)} W$ with the induced volume $m_{T_{\phi(y)} W}$. 
One can easily show that $\nu_{W}$ does not depend on the choices of $(U,\phi)$ by the change of variables formula for Lebesgue measures, and so it is well defined. 

The previous idea of defining volumes can be easily extended to injectively immersed finite dimensional submanifolds, such as unstable manifolds. 
Assume that $W,{W}'$ are two embedded $k$-dimensional submanifolds and $f:\cB\rightarrow \cB$ is a $C^1$ map that maps $W$ diffeomorphically onto ${W}'$. Then, for each $y \in {W}'$ we have
$$\dfrac{df_{\ast}(\nu_{W})}{d\nu_{{W}'}}(y)=\dfrac{1}{\det(Df(f^{-1}(y))|T_{f^{-1}(y)}W)}$$
where $f_{\ast}(\nu_{W})$ is the pushforward of the measure $\nu_{W}$ by $f$. 
		\section{Multiplicative ergodic theorem}\label{SEC:3}
	Let $T: X\rightarrow X$ be an invertible measure preserving transformation of a probability space $(X,\mu)$. 
	For a given measurable function $\psi:X \rightarrow [1,+\infty)$, we call $\psi$ is $T$-\textit{temperate} (with respect to $\mu$), if 
	$$\lim_{n\rightarrow \pm \infty} \frac{1}{n} \log\psi(T^n(x))=0
          \quad \text{for} \ \mu\text{-}a.e. \ x\in X.$$
	
	The following lemmas are well-known, see Lemmas 8 and 9 in \cite{Walters93} for the first result, and Lemma 5.4 in \cite{Young17} for the second result.
	
	\begin{lemma} \label{Lem:Tem}
		Let $T: X\to X $ be an invertible measure preserving transformation of a probability space $(X,\mu)$, and let $\psi:X\rightarrow \R$ be a measurable function. Assume that either $(\psi\circ T-\psi)^{+}$ or $(\psi\circ T-\psi)^{-}$ is integrable, then $\psi\circ T-\psi$ is integrable, and
		$$\lim_{n \rightarrow \pm \infty} \frac{1}{n} \psi\circ T^n(x)=0 \quad \mu\text{-}a.e. \ x\in X.$$
	\end{lemma}
	\begin{lemma}\label{Lem:slowly}
		Let $T: X \to X$ be an invertible measure preserving transformation of a probability space $(X,\mu)$, and let $\psi:X \rightarrow \R$ be a $T$-temperate function.
		Then, for every $\varepsilon>0$, there exists a  measurable function ${\psi}':X \rightarrow [1,\infty)$ such that
		${\psi}'(x) \geq \psi(x)$ and
		${\psi}'(T^{\pm}x) \leq e^{\varepsilon}{\psi}'(x)$ for $\mu$-almost every  $x\in X$.
	\end{lemma}

	\textit{In the rest of this paper, we assume that $(\Phi,\mu)$ satisfies (H1)-(H3). For simplicity the discussion, we assuming $\mu$ is $\Phi$-ergodic from here through Section \ref{SEC:NON-ergodic}.}
	
	\subsection{Oseledets splitting} 
    Notice that $(\ox) \mapsto Df_{\om}(x)$ is strongly measurable.
    Specifically, for a fixed $v\in \cB$, 
    the mapping $x \mapsto Df_{\om}(x) v$ is continuous for each $\om\in \Om$.
    Moreover, the remark in the basic assumptions (see Section \ref{setting}) demonstrates that the mapping $\omega \mapsto Df_{\omega}(x)v$ is measurable for each $x\in \mathcal{B}$. 
	By applying Lemma \ref{Lem:MPolish}, it follows that $(\ox) \mapsto Df_{\omega}(x)v$ is measurable.
	
	Since $\mu$ is $\Phi$-ergodic, the function $l_{\alpha}: \Om\times \cB\rightarrow \R$ defined in (H3) is constant for $\mu$-almost everywhere.
    Write $l_{\alpha}(\ox)=l_{\alpha}$ for $\mu$-almost every $(\ox)$, and choose $\lambda_{\alpha} \in \R$ such that $0>\lambda_{\alpha}>l_{\alpha}$. 
	Then, the following theorem guarantees the existence of Lyapunov exponents of $(\Phi,\mu)$.
	\begin{theorem}[Z. Lian and K. Lu \cite{Lian10}]\label{MET}
		Let $(\Phi,\mu)$ and $\lambda_{\alpha}$ be as in above. 
		Then there exists a $\Phi$-invariant subset $\Gamma$ of $\Omega\times \cB$ and at most finitely many real numbers
		$$\lambda_1>\lambda_2>\cdots>\lambda_{r}$$ 
		with $\lambda_{r}>\lambda_{\alpha}$ such that for every $(\omega,x)\in \Gamma$, there is a unique splitting of the tangent space $\cB_x$
		$$\cB_x=E_1(\omega,x) \bigoplus \cdots E_{r}(\omega,x)  \bigoplus F(\omega,x)$$ 
		with the following properties:
		\begin{enumerate}
			\item[(a)] for $i=1,\cdots,r$, $\dim E_{i}(\omega,x)=m_i<\infty$ is finite and 
            $Df_{\omega}(x)E_{i}(\omega,x)=E_{i}(\theta(\omega),f_{\omega}(x))$, 
			for any $v\in E_{i}(\omega,x) \setminus \{0\}$, we have
			$$\lim_{n\rightarrow \pm \infty} \frac{1}{n}\log |Df_{\omega}^{n}(x)v|=\lambda_{i},$$
			where $Df_{\omega}^{-n}(x)v=(Df_{\theta^{-n}\omega}^{n}(x)|_{E_{i}(\theta^{-n}\omega,f_{\omega}^{-n}(x))})^{-1}v$ is well defined by injectivity;
			\item[(b)] $F(\omega,x)$ is closed and finite co-dimensional, satisfies that
			$Df_{\omega}(x) F(\omega,x)$ $\subset F(\theta(\omega),f_{\omega}(x))$ and 
			$$\limsup_{n\rightarrow \infty} \frac{1}{n}\log |Df_{\omega}^{n}(x)|_{F(\omega,x)}|\leq \lambda_{\alpha};$$
			\item[(c)] the mappings $(\omega,x)\mapsto E_{i}(\omega,x)$, $i=1,\cdots r$, and $(\omega,x)\mapsto F(\omega,x)$ are measurable from $\Om\times \cB$ into $\cG(\cB)$;
			\item[(d)] denote by $\pi_{i}(\omega,x)$, $\pi_{F}(\omega,x)$ the projection of $\cB_x $ onto $E_{i}(\omega,x)$ and $F(\ox)$ via the splitting, 
			then for every $\pi=\pi_{i}$, any $i$ and $\pi=\pi_{F}$ we have that $(\ox) \mapsto \pi(\ox)$ is strongly measurable and for $\mu$-almost every $(\ox)$
			$$\lim_{n\rightarrow \pm\infty} \frac{1}{n}\log |\pi(\Phi^{n}(\omega,x))|=0.$$
		\end{enumerate}
	\end{theorem}
	
	\begin{lemma}\label{Lem:M3}
		For every $E=E_i$, any $i$, and $E=F$ the following maps are measurable:
		\begin{enumerate}
			\item[(i)] 
			$(\omega,x)\mapsto |Df_{\omega}(x)|_{E(\omega,x)}|$,
			\item[(ii)]
			$(\omega,x)\mapsto m(Df_{\omega}(x)|_{E(\omega,x)})$ where $m(T|_{V})$ is the minimum norm,
			\item[(iii)]
			$(\omega,x)\mapsto \det(Df_{\omega}(x)|_{E(\omega,x)})$ for $E=E_i$.
		\end{enumerate}
	\end{lemma}
	\begin{proof}
		Let $E=E_i$ and $m=\dim E$. 
		Consider the mappings $(T,V)\mapsto |T|_{V}|$, $(T,V)\mapsto m(T_{V})$ defined on $B(\cB)\times \cG_{m}(\cB)$, as well as the mapping $(T,V)\mapsto \det(T|V)$ defined on $B_{inj}(\cB) \times \cG_{m}(\cB)$, where $B_{inj}(\cB)$ denotes the subset of injective linear bounded operators in $B(\cB)$.
		
		Before discussing measurability, we first need to specify the $\sigma$-algebra.
        We consider the $\sigma$-algebra on $B(\cB)\times \cG_{m}(\cB)$ to be the product of the strong $\sigma$-algebra on $B(\cB)$ and the Borel $\sigma$-algbra on $\cG_{m}(\cB)$.
  
		For a fixed $T$, the mappings $V\mapsto |T|_{V}|$, $V\mapsto m(T|_{V})$ are continuous on $\cG_{m}(\cB)$, and $V\mapsto \det(T|V)$ is also continuous on $\mathcal{G}{m}(\mathcal{B})$ by Proposition \ref{Prop:CofDet}.
		For a fixed $V$, $T\mapsto |T|_{V}|$, $T\mapsto m(T|_{V})$ are continuous on $B(\cB)$ under the strong operator topology ($SOT(\cB)$), and $T \mapsto \det(T|V)$ is also continuous on $B_{inj}(\cB)$ under the $SOT(\cB)$ as stated in Remark \ref{Rem:SOTC}. 
		Note that $\cG_{m}(\cB)$ is a separable metric space (see Lemma \ref{Lem:SPGk}). 
		Then, by Lemma \ref{Lem:MPolish}, it follows that the three mappings mentioned above are measurable with respect to $(T,V)$.
		Therefore, it is sufficient to demonstrate that $(\omega,x)\mapsto (Df_{\omega}(x),E(\omega,x))$ is measurable.

        Recall that $(\omega,x)\mapsto Df_{\omega}(x)$ is strongly measurable, which also means it is measurable with respect to the strong $\sigma$-algebra on $B(\cB)$. Furthermore, by Theorem \ref{MET}, we know that $(\omega,x)\mapsto E(\omega,x)$ is measurable.
        Consequently, the mapping $(\omega,x)\mapsto (Df_{\omega}(x),E(\omega,x))$ is measurable.
		
		Let $E=F$, and let $\{v_n\}_{n=1}^{\infty}\subset \cB$ be a countable dense subset of $\cB$. 
		Then, $\{\pi_{F}(\ox) v_n\}_{n=1}^{\infty}$ forms a countable dense subset of the subspace $F(\ox)$. 
		It is worth noting that for every $c\geq 0$,
		\begin{align*}
			&\{(\ox):|Df_{\omega}(x)|_{F(\omega,x)}|\leq c\}=\bigcap_{n>0}\{(\ox):|Df_{\omega}(x)\pi_{F}(\ox) v_n|\leq c|\pi_{F}(\ox) v_n |\}\\
			&\{(\ox):m(Df_{\omega}(x)|_{F(\omega,x)})\geq c\}=\bigcap_{n>0}\{(\ox):|Df_{\omega}(x) \pi_{F}(\ox) v_n|\geq c|\pi_{F}(\ox) v_n|\}.
		\end{align*}
		By Theorem \ref{MET} (d), the mapping $(\ox)\mapsto \pi_{F}(\ox)$ is strongly measurable. 
		Then, the function
        $(\ox)\mapsto |\pi_{F}(\ox)v_n|$ is measurable for each $n$. 
		Additionally, the mapping $(A_1,A_2)\mapsto A_1A_2$ is continuous under the strong operator topology. 
		By the strongly measurable properties of $(\ox)\mapsto D\fo(\ox)$ and $(\ox)\mapsto \pi_{F}(\ox)$, it follows that  
		$(\omega,x) \mapsto Df_{\omega}(x) \pi_{F}(\ox)$ is also strongly measurable. 
		As a result, both of the sets mentioned above are measurable.
		This completes the proof.
	\end{proof}
	
	\begin{proposition}\label{Prop:DetL}
		For every $k\leq r$ and any 
		$1\leq i_1\leq \cdots \leq i_k\leq r$ 
		the mapping 
		$(\omega,x)\mapsto
		\det(Df_{\omega}(x)|\oplus_{j=1}^{k}E_{i_j}(\omega,x))$ is measurable.
		Moreover, for $\mu$-amolst every $(\omega,x)$ we have
		$$\lim_{n\rightarrow \infty} \frac{1}{n}\log
		\det(Df^{n}_{\omega}(x)|\bigoplus_{j=1}^{k}E_{i_j}(\omega,x))=\sum_{j=1}^{k}m_{i_j}\lambda_{i_j}.$$
	\end{proposition}
	\begin{proof}
		We follow the proof of Proposition 4.7 in \cite{Young17}.
		
		First, we will prove the case for $k=1$, using $\lambda$, $E$, and $m$ as shorthand for $\lambda_{i_1}$, $E_{i_1}$, and $m_{i_1}$, respectively, for simplicity.
		Let $\psi(\ox):=\log \det(D\fo(x)|E(\ox))$ for $(\ox)
		\in \Gamma$. 
		It follows immediately from Lemma \ref{Lem:M3} that $\psi$ is measurable. 
		By Proposition \ref{Prop:Det} (3), we have
		$$\psi(\ox)\leq \log m^{\frac{m}{2}}+m\log |\fo|_{B(A_{\omega},r_0)}|_{C^2}.$$
		Hence, by assumption (H2) (iii), we have 
		$\psi^{+}(\ox)\in L^1({u})$. 
		Using the Birkhoff ergodic theorem, there is a constant $\gamma\in [-\infty,\infty)$ such that for $\mu$-almost every $(\ox)$,
		$$ \gamma=\lim_{n\rightarrow \infty}\frac{1}{n}\sum_{i=0}^{n-1} \psi(\Phi^{i}(\ox))=\lim_{n\rightarrow \infty}\frac{1}{n}\sum_{i=0}^{n-1} \psi(\Phi^{-i}(\ox)).$$
		It suffices to show that $\gamma=m\lambda$. 
		On the one hand, by Proposition \ref{Prop:Det} (3) we can choose a unit basis $\{v_i\}_{i=1}^{m}$ for $E(\ox)$ such that 
		$$\det(Df^{n}_{\omega}(x)|E(\omega,x))\leq m^{m/2} \prod_{i=1}^{m}|Df^{n}_{\omega}(x)v_i|.$$
		By Theorem \ref{MET}, it follows that $\gamma\leq m\lambda$. 
		On the other hand, observe that
		\begin{align*}
			\sum_{i=0}^{n-1} \psi(\Phi^{-i}(\ox))&=\det(Df^{n}_{\te^{-n}\omega}(\fo^{-n}x)|E(\Phi^{-n}(\ox)))\\
			&=(\det(Df^{-n}_{\omega}(x)|E(\ox)))^{-1}.
		\end{align*}
		Similarly, we choose a unit basis $\{u_i\}_{i=1}^{m}$ for $E(\ox)$ such that 
		$$\det(Df^{-n}_{\omega}(x)|E(\omega,x))\leq m^{m/2} \prod_{i=1}^{m}|Df^{-n}_{\omega}(x)u_i|.$$
		Again, by Theorem \ref{MET}, we have  $\gamma\geq m\lambda$. 
		This established the case when $k=1$.
		
		To establish the general case, it suffices to provide a proof for $k=2$. 
		Applying Proposition \ref{Prop:Det}, we have the following inequality:
		\begin{align*}
			\dfrac{1}{C_{m_{i_1}+m_{i_2}}\cdot|\pi_{i_1}(\Phi^{n}(\ox))|^{m_{i_1}}}&\leq \dfrac{\det(Df^{n}_{\omega}(x)|E_{i_1}(\omega,x)\oplus E_{i_2}(\omega,x))}{\det(Df^{n}_{\omega}(x)|E_{i_1}(\omega,x)
				)\det(Df^{n}_{\omega}(x)|E_{i_2}(\omega,x))}\\
			&\leq C_{m_{i_1}+m_{i_2}}|\pi_{i_1}(\ox)|^{m_{i_1}}.
		\end{align*}
		Recall that Theorem \ref{MET} states that $|\pi_{i_1}|$ is $\Phi$-temperate. 
		Thus, combined with the case $k=1$, we can conclude that the case for $k=2$ is also valid.
	\end{proof}
 
    \begin{corollary} \label{Cor:InE}
    	Let $1\leq i_1\leq \cdots \leq i_k\leq r$ and let
        $E(\ox):=\bigoplus_{j=1}^{k}E_{i_j}(\omega,x)$.
        Then, we have 
    	$$\int \log \det(Df_{\omega}(x)|E(\omega,x)) d\mu<\infty \ \text{and} \ \int\log^{+}|(Df_{\omega}(x)|_{E(\omega,x)})^{-1}|d\mu<\infty.$$
    \end{corollary}
    \begin{proof}
        Let $\psi(\ox):=\log \det(Df_{\omega}(x)|E(\omega,x))$. 
        It follows immediately from Proposition \ref{Prop:DetL} that $\psi \in L^1(\mu)$.
        Let $m= \dim E(\ox)$ for $\mu$-almost every $(\ox)$.
        Then, for any $v \in E(\ox)$ with $|v|=1$, we have (see Proposition \ref{Prop:Det})
        $$\det(Df_{\omega}(x)|E(\omega,x)) \leq m^m |Df_{\omega}(x)|^{m-1} \cdot |Df_{\omega}(x)v|.$$
        This implies that 
        $$\log |(Df_{\omega}(x)|_{E(\omega,x)})^{-1}|\leq m\log m+(m-1)\log |(f_{\om}|_{B(A_{\om},r_0)})|_{C^2}-\psi(\ox).$$
        Therefore, by $\psi\in L^1(\mu)$ and the assumption (H2) (iii), we have 
        $$\int\log^{+}|(Df_{\omega}(x)|_{E(\omega,x)})^{-1}|d\mu<\infty.$$
        This completes the proof.
    \end{proof}

	\subsection{Lyapunov norms}\label{SEC:Lyapunov norms}
	Recall that $\lambda_{\alpha}<0$, it is sufficient to make a distinction between unstable, center, and stable subspaces, defined by
	\begin{align*}
		&E^{u}(\omega,x)=\bigoplus\{E_{i}(\omega,x):\lambda_i>0\}, \ 
		E^{c}(\omega,x)=\bigoplus\{E_{i}(\omega,x):\lambda_i=0\}\\
		&\text{and} \ E^{s}(\omega,x)=\bigoplus\{E_{i}(\omega,x):\lambda_i<0\}
		\bigoplus F(\omega,x).
	\end{align*} 
	Let $\pi^{\tau}_{(\omega,x)}$ be the projection onto  $E^{\tau}(\omega,x)$ according to the splitting
	$\cB_x=E^{u}(\omega,x)\oplus E^{c}(\omega,x)\oplus E^{s}(\omega,x)$ 
	for $\tau=u,c,s$ respectively, and let $\pi^{cs}_{(\ox)}=\pi^{c}_{(\ox)}+\pi^{s}_{(\ox)}$.
	
	We now introduce a new norm $\Lnorm{\cdot}{(\ox)}$ on the tangent space $\cB_x$, in which the expansions and contractions of vectors are reflected in a single time step and called \textit{Lyapunov norms}. 
	Let $\lambda^{+}=\min\{\lambda_{i}:\lambda_{i}>0\}$ and $\lambda^{-}=\max\{\lambda_{\alpha},\lambda_{i}:\lambda_{i}<0\}$. 
	Denote by $\lambda_0=\min\{\lambda^{+},-\lambda^{-}\}$, and fix $\varepsilon_0\ll \lambda_0$, let $\lambda=\lambda_0-\varepsilon_0$. 
	For each $(\omega,x)\in \Gamma$, we let 
	\begin{equation*}
		\begin{alignedat}{2}
		|u|'_{(\omega,x)}:=&\sum_{n\geq 0}\dfrac{|Df^{-n}_{\omega}(x) u|}{e^{-n\lambda}} 
			\ &\text{for}& \  u\in E^{u}(\omega,x), \\
		|v|'_{(\omega,x)}:=&\sum_{n \in \Z}\dfrac{|Df^{n}_{\omega}(x) v|}{e^{|n|\varepsilon_0}} 
			\ &\text{for}& \  v\in E^{c}(\omega,x),\\
		|w|'_{(\omega,x)}:=&\sum_{n\geq 0}\dfrac{|Df^{n}_{\omega}(x) w|}{e^{-n\lambda}} 
			\ &\text{for}& \  w\in E^{s}(\omega,x)
		\end{alignedat}
		\end{equation*}
	For each $p\in \cB$, let
	$$|p|'_{(\omega,x)}:=
	\max\{\Lnorm{\pi^{u}_{(\omega,x)}p}{(\ox)},\Lnorm{\pi^{c}_{(\omega,x)}p}{(\ox)},\Lnorm{\pi^{s}_{(\omega,x)}p}{(\ox)}\}.$$
	To estimate the difference between these new norms and the original ones, we let
	\begin{align*}
		&C_{u}(\omega,x)=\sup_{n\geq 0}\dfrac{\sup_{v\in E^{u}(\omega,x),|v|=1}|Df_{\omega}^{-n}(x) v|}{e^{-n(\lambda_0-\frac{\varepsilon_0}{2})}} \\
		&C_{c}(\omega,x)=\sup_{n\in \Z}\dfrac{\sup_{v\in E^{c}(\omega,x),|v|=1}|Df_{\omega}^{n}(x) v|}{e^{\frac{|n|\varepsilon_0}{2}}}\\
		&C_{s}(\omega,x)=\sup_{n\geq 0}\dfrac{\sup_{v\in E^{s}(\omega,x),|v|=1}|Df_{\omega}^{n}(x) v|}{e^{-n(\lambda_0-\frac{\varepsilon_0}{2})}}.
	\end{align*}
	and let 
	$$C(\ox)=\max\{C_{u}(\omega,x),C_{c}(\omega,x),C_{s}(\omega,x),|\pi^{u}_{(\omega,x)}|,|\pi^{c}_{(\omega,x)}|,|\pi^{s}_{(\omega,x)}|\}.$$
	By Theorem \ref{MET} and Lemma \ref{Lem:M3}, these functions are sll measurable and finite-valued on $\Gamma$. 
        Moreover, one can easily show that 
        $C_{\tau}$ is $\Phi$-temperate for $\tau=u,c,s$ by using Lemma \ref{Lem:Tem} and Corollary \ref{Cor:InE}.
	This together with Theorem \ref{MET} (d), we have $C$ is also $\Phi$-temperate.
	\begin{lemma}[Linear picture]\label{Lem:LinP}
		The following hold for all $(\omega,x)\in \Gamma$:
		\begin{align}
			e^{\lambda}|u|'_{(\omega,x)} \leq &|Df_{\omega}(x) u|'_{\Phi(\omega,x)} \label{eq:41};  \\
			e^{-\varepsilon_0}|v|'_{(\omega,x)}\leq &|Df_{\omega}(x) v|'_{\Phi(\omega,x)} \leq e^{\varepsilon_0} |v|'_{(\omega,x)}  \label{eq:42};\\
			&|Df_{\omega}(x) w|'_{\Phi(\omega,x)} \leq e^{-\lambda}|w|'_{(\omega,x)} \label{eq:43}.
		\end{align}
		Where $u\in E^{u}(x,h)$, $v\in E^{c}(x,h)$ and $w\in E^{s}(x,h)$. Moreover, the relations between the norms are
		\begin{equation}\label{eq:norms}
			\frac{1}{3}|\cdot|\leq |\cdot|'_{(\omega,x)}\leq 
			\frac{3}{1-e^{-\frac{\varepsilon_0}{2}}}C(\omega,x)^{2} |\cdot|
		\end{equation}
	\end{lemma}  
	The proof is a simple computation and we omitted it here.
	
	For each $(\omega,x)\in \Gamma$.
	Let $\widetilde{B}_{(\omega,x)}(r):=\{v\in \cB_x:\Lnorm{v}{(\ox)} \leq r\}$. Let the connecting map
	$\widetilde{f}_{(\omega,x)}:=\exp^{-1}_{f_{\omega}(x)} \circ f_{\omega}\circ \exp_{x}$,
	and for each $n\in \Z$ denote by 
	$\widetilde{f}^{n}_{(\omega,x)}:=\exp^{-1}_{f^{n}_{\omega}(x)} \circ f^{n}_{\omega}\circ \exp_{x}$.
	Recall the constant $r_0$ as in assumption (H2) (iii).
	\begin{lemma}[Nonlinear picture]\label{Lem:NlinP}
		Define a measurable function 
		${\ell}':\Gamma\rightarrow [1,\infty)$ 
		by
		$${\ell}'(\ox):=\frac{27}{1-e^{-\frac{\varepsilon_0}{2}}} C(\Phi(\ox))^{2}\cdot \max\{1,|(\fo|_{B(A_{\omega},r_0)})|_{C^2}\}.$$
		Then, for every $0<\delta<r_0$ and every $(\omega,x)\in \Gamma$ the following holds for nonlinear map
		$$\widetilde{f}_{(\omega,x)}: (\widetilde{B}_{(\omega,x)}(\delta {\ell}'(\ox)^{-1}),|\cdot|'_{(\omega,x)}) \rightarrow (\cB_{f_{\omega}(x)},|\cdot|'_{\Phi(\omega,x)}).$$
		\begin{enumerate}
			\item[(1)] $\Lip(\widetilde{f}_{(\omega,x)}-Df_{\omega}(x))\leq \delta$;
			\item[(2)] 
			the mapping $z\mapsto D\widetilde{f}_{(\omega,x)}(z)$ satisfies $\Lip (D\widetilde{f}_{(\omega,x)})\leq {\ell}'(\ox)$.
		\end{enumerate}
	\end{lemma}
	\begin{proof}
		Let $\Lnorm{\cdot}{}$ denote the operator norm of bounded linear operators with respect to the Lyapunov norms. 
		
		For any $y,z\in \wB_{(\ox)}(\delta {\ell}'(\ox)^{-1})$, we have $|y-x|\leq r_0$ and $|z-x|\leq r_0$. 
            Hence, we can estimate as follows
		\begin{align*}
			|D\wf_{(\ox)}(y)-D\wf_{(\ox)}(z)|'
			&\leq 
			\frac{9}{1-e^{-\frac{\varepsilon_0}{2}}} C(\Phi(\ox))^{2} \cdot |(\fo|_{B(A_{\omega},r_0)})|_{C^2} \cdot |z-y|\\
			&\leq \frac{27}{1-e^{-\frac{\varepsilon_0}{2}}} C(\Phi(\ox))^{2}\cdot |(\fo|_{B(A_{\omega},r_0)})|_{C^2} \cdot \Lnorm{z-y}{(\ox)}\\
			&\leq {\ell}'(\ox)\Lnorm{z-y}{(\ox)},
		\end{align*}
		Thus, we have shown that (2) holds. To prove (1), we provide the following estimate
		\begin{align*}
			\Lip(\wf_{(\ox)}-D\wf_{(\ox)}(0))
			&\leq \sup_{y\in \wB_{(\ox)}(\delta {\ell}'(\ox)^{-1})}|D(\wf_{(\ox)}-D\wf_{(\ox)}(0))(y)|'\\
			&\leq \sup_{y\in \wB_{(\ox)}(\delta {\ell}'(\ox)^{-1})}|D\wf_{(\ox)}(y)-D\wf_{(\ox)}(0)|'\\
			&\leq \Lip (D\widetilde{f}_{(\omega,x)})\cdot \Lnorm{y}{(\ox)} \leq \delta.
		\end{align*}
	    Here, we utilize item (2) in the third inequality..
	\end{proof}
        By assumption (H2) (iii) and Lemma \ref{Lem:Tem}, we have $\omega\mapsto |(\fo|_{B(A_{\omega},r_0)})|_{C^2}$ is $\theta$-temperate.
	Since $C$ is $\Phi$-temperate, the function ${\ell}'$ is also $\Phi$-temperate. 
	Let $\varepsilon_1\ll \varepsilon_0$. 
        By applying Lemma \ref{Lem:slowly}, there is a measurable function 
        $\ell:\Gamma\rightarrow [1,\infty)$  such that $\ell \geq {\ell}'$ and
	\begin{equation}\label{eq:functionl}
		\ell(\Phi^{\pm}(\ox))\leq e^{\varepsilon_1}\ell(\ox).
	\end{equation}
	By the definition of the function $\ell$, for each $(\ox)\in \Gamma$, we obtain
	\begin{equation}\label{eq:lox}
    \begin{aligned}
       &\frac{1}{3}\Lnorm{\cdot}{(\ox)}\leq |\cdot|\leq \ell(\ox)\Lnorm{\cdot}{(\ox)}, \ \text{and}\\  
       &\max\{|\pi^{u}_{(\ox)}|, \ |\pi^{cs}_{(\ox)}|,\  |(f_{\om}|_{B(A_{\om},r_0)})|_{C^2}\}\leq \ell(\ox).  
    \end{aligned}	
    \end{equation}

	Choose $0<\delta_1<r_0$ to be sufficiently small such that
	\begin{equation}\label{eq:delta1}
		e^{\varepsilon_0}+\delta_1<e^{\frac{1}{2}\lambda}<e^{\lambda}-\delta_1.
	\end{equation}
	\begin{corollary}\label{Cor:NlinP}
		Let $(\omega,x)\in \Gamma$. 
		Denote $\wf^{k}:=\wf^{k}_{(\ox)}$, $\Lnorm{\cdot}{k}:=\Lnorm{\cdot}{\Phi^k(\ox)}$, $\ell(k)=\ell(\Phi^{k}(\ox))$ and $\pi^{\tau}_{k}=\pi^{\tau}_{\Phi^k(\ox)}$ for $k\in \Z$ and $\tau=u,cs$. 
		For every $0<\delta\leq\delta_1$, if $u,v\in \cB_x$ and $n\in \N$ satisfy  
		$\Lnorm{\wf^{k}(u)}{k},\Lnorm{\wf^{k}(v)}{k} \leq \delta \ell(k)^{-1}$ for $k=0,1,\cdots,n-1$, 
		then the following statements holds:
		\begin{enumerate}
			\item[(1)] If $\Lnorm{u-v}{0}=\Lnorm{\pi^{u}(u-v)}{0}$, then for $k=1,\cdots,n$
			\begin{align*}
				\Lnorm{\wf^{k}(u)-\wf^{k}(v)}{k}&=\Lnorm{\pi^{u}_{k}(\wf^{k}(u)-\wf^{k}(v))}{k}\\
				&\geq (e^{\lambda}-\delta)^{k}\Lnorm{u-v}{0}.
			\end{align*}
			\item[(2)] 
			If $\Lnorm{\wf^{n}(u)-\wf^{n}(v)}{n}=\Lnorm{\pi^{cs}_{n}(\wf^{n}(u)-\wf^{n}(v))}{n}$, then for $k=0,1,\cdots,n-1$
			\begin{align*}
				\Lnorm{\wf^{k}(u)-\wf^{k}(v)}{k}&=  \Lnorm{\pi^{cs}_{k}(\wf^{k}(u)-\wf^{k}(v))}{k}\\
				&\leq (e^{\varepsilon_0}+\delta)^{k}\Lnorm{u-v}{0}.
			\end{align*}
		\end{enumerate}
	    Furthermore, if $E^{c}(\omega,x)=\{0\}$, then the right-hand side in item (2) is given by
	    $$\leq (e^{-\lambda}+\delta)^{k}\Lnorm{u-v}{0}.$$
	\end{corollary}
    \begin{proof}
         For $u,v\in \cB_x$ with $\Lnorm{u}{0},\Lnorm{v}{0}\leq \delta \ell(0)^{-1}$. 
         We first assume that $\Lnorm{u-v}{0}=\Lnorm{\pi^{u}(u-v)}{0}$. 
         By Lemma \ref{Lem:NlinP} we have
         \begin{equation}\label{eq:21}
          	\begin{aligned}
            \Lnorm{\pi^{u}_{1}(\wf^{1} u-\wf^{1} v)}{1}\geq& \Lnorm{\pi^{u}_{1}(D\wf^{1}(0)(u-v))}{1}\\
                                                           &-\Lnorm{\pi^{u}_{1}  ((\wf^{1}-D\wf^{1}(0))u-(\wf^{1}-D\wf^{1}(0))v)}{1}\\
                                                       \geq& \Lnorm{D\wf^{1}(0)(\pi^{u}(u-v))}{1}-\delta\Lnorm{u-v}{0}\\
                                                       \geq& (e^{\lambda}-\delta)\Lnorm{u-v}{0}.
          	\end{aligned}
         \end{equation}
         Similarly, we can show
         \begin{equation}\label{eq:22}
         	\Lnorm{\pi^{cs}_{1}(\wf^{1} u-\wf^{1} v)}{1}\leq (e^{\varepsilon_0}+\delta)\Lnorm{u-v}{0}.
         \end{equation}
         By the choice of $\delta_1$, we have 
         $\Lnorm{\pi^{u}_{1}(\wf^{1} u-\wf^{1} v)}{1}\geq \Lnorm{\pi^{cs}_{1}(\wf^{1} u-\wf^{1} v)}{1}$. 
         This implies that
         $\Lnorm{\wf^{1} u-\wf^{1} v}{1}=\Lnorm{\pi^{u}_{1}(\wf^{1} u-\wf^{1} v)}{1}$ when $\Lnorm{u-v}{0}=\Lnorm{\pi^{u}(u-v)}{0}$. 
         
         Now let's assume that $u,v$ and $n$ satisfy the conditions stated in the corollary. 
         For item (1), if $\Lnorm{u-v}{0}=\Lnorm{\pi^{u}(u-v)}{0}$. 
         By induction, one can show 
         $\Lnorm{\wf^{k}(u)-\wf^{k}(v)}{k}=\Lnorm{\pi^{u}_{k}(\wf^{k}(u)-\wf^{k}(v))}{k}$ for every $k=1,\cdots n$. 
         Then, using (\ref{eq:21}) we have        
         $$\Lnorm{\wf^{k}(u)-\wf^{k}(v)}{k}\geq (e^{\lambda}-\delta)\Lnorm{\wf^{k-1}(u)-\wf^{k-1}(v)}{k-1}$$
         for every $k=1,\cdots n$. This proves item (1). 
         
         For item (2), if $\Lnorm{\wf^{n}(u)-\wf^{n}(v)}{n}=\Lnorm{\pi^{cs}_{n}(\wf^{n}(u)-\wf^{n}(v))}{n}$, 
         then we must have
         $\Lnorm{\wf^{k}(u)-\wf^{k}(v)}{k}=\Lnorm{\pi^{cs}_{k}(\wf^{k}(u)-\wf^{k}(v))}{k}$ 
         for every $k=0,\cdots n-1$. 
         Together with (\ref{eq:22}), we have
         $$\Lnorm{\wf^{k}(u)-\wf^{k}(v)}{k}\leq (e^{\varepsilon_0}+\delta)\Lnorm{\wf^{k-1}(u)-\wf^{k-1}(v)}{k-1}$$
         for every $k=1,\cdots n$. This proves item (2). 
         
         Moreover, if $E^{c}(\ox)=\{0\}$, then (\ref{eq:22}) is given by
         $$\Lnorm{\pi^{s}_{1}(\wf^{1} u-\wf^{1} v)}{1}\leq (e^{-\lambda}+\delta)\Lnorm{u-v}{0}.$$ 
         Using the same argument above one can show that
         $$\Lnorm{\wf^{k}(u)-\wf^{k}(v)}{k}=  \Lnorm{\pi^{s}_{k}(\wf^{k}(u)-\wf^{k}(v))}{k}\leq (e^{-\lambda}+\delta)^{k}\Lnorm{u-v}{0}$$
         when $\Lnorm{\wf^{n}(u)-\wf^{n}(v)}{n}=\Lnorm{\pi^{s}_{n}(\wf^{n}(u)-\wf^{n}(v))}{n}$. 
         This completes the proof.
    \end{proof}

	\subsection{Continuity of certain spaces on uniformly sets}\label{SEC:CCS}
	For $\ell_0\geq 1$, we let 
	$$\Gamma_{\ell_0}:=\{(\ox)\in \Gamma: \ell(\ox)\leq \ell_0\}$$
	and called the sets of the form $\Gamma_{\ell_0}$ as \textit{uniformly sets}. Since $\Omega \times \cB$ is a Polish space, one can assume that $\Gamma_{\ell_0}$ is compact for each $\ell_0$.
	
	In finite-dimensional case or Hilbert case, there are more general results about the continuity of Oseledets subspaces, which could reference \cite{Simion16} and \cite{Froyland18}.
	\begin{proposition}\label{Prop:CCS}
		For each $\om\in \Gamma$. 
		The mappings $x \mapsto E^{cs}(\ox)$ and $x \mapsto E^{u}(\ox)$ are continuous on $(\Gamma_{\ell_0})_{\omega}$ with respect to the $d_H$-metric on $\cG(B)$. 
		More specifically, for any $\epsilon>0$, there exists $\delta(\ell_0,\varepsilon)>0$ and ${\delta}'(\ell_0,\varepsilon,\omega)>0$
		such that for every $(\ox),(\oy)\in \Gamma_{\ell_0}$, 
		if $|x-y|\leq \delta$ then  $d_H(E^{cs}(\ox),E^{cs}(\oy))<\epsilon$, 
		if $|x-y|\leq {\delta}'$ then $d_H(E^{u}(\ox),E^{u}(\oy))<\epsilon.$
	\end{proposition}
	\begin{proof}
		For every $n\in \N$ and every $(\ox), (\omega,y)\in \Gamma_{\ell}$.
		We first show that
		$$|Df^n_{\omega}(x)-Df^n_{\omega}(y)|\leq C(n,\ell)\cdot|x-y|.$$
		Indeed, we have
		\begin{align*}
			|Df^n_{\omega}(x)-Df^n_{\omega}(y)|
			&\leq (\prod^{n-1}_{m=0}|(f_{\te^m \omega}|_{B(A_{\te^m \om,r_0})})|_{C^2})\cdot \sum_{j=0}^{n-1}|f^{j}_{\om}x-f^{j}_{\om}y|\\
			&\leq n\ell^{2n} e^{(n^2+n)\varepsilon_1 }\cdot |x-y|=C(n,\ell)\cdot|x-y|.
		\end{align*}
		
		Assume that $(\ox),(\oy)\in \Gamma_{\ell_0}$.
		Fix a unit vector $v\in E^{cs}(\oy)$. 
		Let $v=v^s+v^u$ be the decomposition with respect to the splitting $E^{cs}(\ox)\oplus E^{u}(\ox)$. 
		For each $n\in \N$. 
		On the one hand, by Lemma \ref{Lem:LinP} and (\ref{eq:lox}) we have
		\begin{align*}
			|D\fo^{n}(x) v|&\geq |D\fo^{n}(x) v^u|-|D\fo^{n}(x) v^{cs}| \\             
			&\geq (3\ell_0)^{-1}e^{n(\lambda-\varepsilon_1)}|v^u|-3\ell_0 e^{n\varepsilon_0}\cdot|\pi^{cs}_{(\ox)}|
		\end{align*}
		On the other hand, 
		\begin{align*}
			|D\fo^{n}(x) v|&\leq 
			|D\fo^{n}(y)v|+|(D\fo^{n}(x)-D\fo^{n}(y))v|\\         &\leq 3\ell_0e^{n\varepsilon_0}+C(n,\ell_0)|x-y|.
		\end{align*}
		Recall that $\max\{|\pi^{u}_{(\omega,z)}|,|\pi^{cs}_{(\omega,z)}|\}\leq \ell(\omega,z)$ for every $(\omega,z)\in \Gamma$.
		Therefore, we have that
		$$|v^u|\leq (3\ell_0)^2(1+\ell_0)e^{-n(\lambda-\varepsilon_0-\varepsilon_1)}+3\ell_0 e^{-n\lambda}C(n,\ell_0)|x-y|.$$
		For each $\epsilon>0$, take ${n}':={n}'(\ell_0,\epsilon)$ large enough so that $(3\ell_0)^2(1+\ell_0)e^{-{n}'(\lambda-\varepsilon_0-\varepsilon_1)}<\frac{\epsilon}{4}$,
		and take $\delta:=\delta(\ell_0,\epsilon)$ small enough such that $3\ell_0 e^{-{n}'\lambda}C({n}',\ell_0)\delta< \frac{\epsilon}{4}$.
		Then, by the definition of the gap $G(\cdot,\cdot)$ we have 
		$G(E^{cs}(\oy),E^{cs}(\ox))\leq \frac{\epsilon}{2}$ when $|x-y|\leq \delta$.
		Symmetrically, replacing $(x,y)$ by $(y,x)$ one has $G(E^{cs}(\ox),E^{cs}(\omega,y))<\frac{\epsilon}{2}$. 
		So, we get that $d_{H}(E^{cs}(\ox),E^{cs}(\omega,y))<\epsilon$ when $|x-y|\leq \delta$.
		
		Note that we cannot iterate backwards to estimate $d_H(E^{u}(\ox),E^{u}(\oy))$.
		Fix a unit vector $v\in E^{u}(\omega,y)$, and decompose it as $v=v^u+v^s$  with respect to the splitting $E^{u}(\ox)\oplus E^{cs}(\ox)$. 
		Denote by $x_{-n}$ and $y_{-n}$ the points $f^{-n}_{\omega}(x)$ and $f^{-n}_{\omega}(y)$ respectively, and let $v_{-n}$ be the unique vector in $E^{u}(\te^{-n}\omega,y_{-n})$ such that $Df^n_{\te^{-n}\omega}(y_{-n}) v_{-n}=v$. 
		We decompose $v_{-n}$ by
		$$v_{-n}=v^u_{-n}+v^s_{-n}\in E^{u}(\te^{-n} \omega,x_{-n})\oplus E^{cs}(\te^{-n}\omega,x_{-n}).$$
		By using (\ref{eq:lox}) and Lemma \ref{Lem:LinP}, we have $|v_{-n}|\leq 3\ell_0e^{-n\lambda}$. 
		We can bound $v^s$ as follows:
		$$|v^s|\leq |v^s-Df^n_{\te^{-n}\omega}(x_{-n}) v^s_{-n}|+|Df^n_{\te^{-n}\omega}(x_{-n}) v^s_{-n}|.$$
		The first term can be upper-bounded as
		$$\leq |\pi^{cs}_{(\omega,x)}|\cdot |v-Df^n_{\te^{-n}\omega}(x_{-n}) v_{-n}|\leq 3\ell^{2}_0 e^{-n\lambda} C(n,\ell_0e^{n\varepsilon_1})|x_{-n}-y_{-n}|.$$
		The second term has the following estimation:
		\begin{align*}
			|Df^n_{\te^{-n}\omega}(x_{-n}) v^s_{-n}|&\leq 3\ell_0e^{n(\varepsilon_0+\varepsilon_1)} |v^s_{-n}|\\
			&\leq3\ell_0 e^{n(\varepsilon_0+\varepsilon_1)}|\pi^{cs}_{\Phi^{-n}(\ox)}|\cdot|v_{-n}|\\
			&\leq9\ell_0^{3}e^{-n(\lambda-\varepsilon_0-2\varepsilon_1)}. 
		\end{align*}
		For each $\epsilon>0$,  choose $k:=k(\epsilon,\ell_0)$ large enough such that the second term is upper-bounded by $\epsilon/4$. 
		Therefore, by the definition of $G(\cdot,\cdot)$ we have
		$$G(E^u(\oy),E^u(\ox))\leq \frac{\epsilon}{4}+ 3\ell^{2}_0 e^{-k\lambda}C(k,\ell_0e^{k\varepsilon_1})\cdot |x_{-k}-y_{-k}|.$$
		Choose ${\delta}'':={\delta}''(\epsilon,\ell_0)$ to be sufficiently small such that $G(E^u(\oy),E^u(\ox))< \frac{\epsilon}{2}$ when $|x_{-k}-y_{-k}|<{\delta}''$.  
		Symmetrically, replacing $(x,y)$ by $(y,x)$ we have $G(E^u(\ox),E^u(\omega,y))<\frac{\epsilon}{2}$. 
            Therefore, we obtain $d_{H}(E^u(\ox),E^u(\omega,y))< \epsilon$ when $|x_{-k}-y_{-k}|<{\delta}''$.
		
		To ensure that  $|x_{-k}-y_{-k}|<{\delta}''$, we use the fact that $f^{-k}_{\omega}$ continuous on the compact subset $A_{\omega}$. 
		Therefore, we choose a real number ${\delta}'(\epsilon,\ell_0,\omega)$ to be sufficiently small such that for every $x,y\in A_{\om}$, 
            if $|x-y|\leq {\delta}'$ then $ |x_{-k}-y_{-k}|<{\delta}''$. 
		This completes the proof.
	\end{proof}
		\section{Nonuniform hyperbolic theory}\label{SEC:4}
	\subsection{Unstable manifolds}\label{SEC:4.1}
	For each $(\omega,x)\in \Gamma$.
	Let $\widetilde{B}^{\tau}_{(\omega,x)}(r):=\{v\in E^{\tau}(\omega,x):{|v|}'_{(\omega,x)}\leq r\}$ for $\tau=u,cs$ and $r>0$. 
	The definition of the Lyapunov norms implies that $\widetilde{B}_{(\ox)}(r)=\widetilde{B}^{u}_{(\ox)}(r)+\widetilde{B}^{cs}_{(\ox)}(r)$.
	
	\begin{theorem}[Local unstable manifolds] \label{Thm:Unstable}
		There exists ${\delta}'_1>0$ such that, ${\delta}'_1<\delta_1$ (where $\delta_1$ as in Corollary \ref{Cor:NlinP}) 
		and for every $\delta<{\delta}'_1$ there exists a unique family of continuous maps 
		$$\{g_{(\omega,x)}:\widetilde{B}^{u}_{(\omega,x)}(\delta \ell(\omega,x)^{-1})\rightarrow \widetilde{B}^{cs}_{(\omega,x)}(\delta \ell(\omega,x)^{-1})  \}_{(\omega,x)\in\Gamma}$$ 
		satisfied that
		$$g_{(\omega,x)}(0)=0 \ \text{and} \ \widetilde{f}_{(\omega,x)}\graph(g_{(\omega,x)})\supset \graph(g_{\Phi(\omega,x)}).$$
		With respect to the $\Lnorm{\cdot}{(\ox)}$ norms on $\widetilde{B}_{(\ox)}(\delta \ell(\ox)^{-1})$, the family $\{g_{(\omega,x)}\}_{(\omega,x)\in \Gamma}$ has the following additional properties:
		\begin{enumerate}
			\item[(1)] $g_{(\omega,x)}$ is $C^{1+\Lip}$ Fr\'{e}chet differentiable, with $Dg_{(\omega,x)}(0)=0$;
			\item[(2)] $\Lip g_{(\omega,x)}\leq 1/10$ and $\Lip(Dg_{(\omega,x)})\leq C_0 \ell(\omega,x)$ where $C_0> 0$ is independent of $(\omega,x)$; 
			\item[(3)] if $u_i\in \graph(g_{(\omega,x)})$ and $\widetilde{f}_{(\omega,x)}u_i\in \graph(g_{\Phi(\omega,x)})$ for $i=1,2$, then
			\begin{align*}
				\Lnorm{ \wf_{(\omega,x)}(u_1)-\wf_{(\omega,x)}(u_2) }{\Phi(\omega,x)}
				\geq (e^\lambda-\delta) \Lnorm{u_1-u_2}{(\omega,x)}
			\end{align*}
		\end{enumerate}
	\end{theorem}
	\begin{remark}\label{Rem:mofg}
		The measurability of the mapping $(\ox)\mapsto E^{u}(\ox)$ allows us to can find a measurable basis $\{e_i{(\ox)}\}_{i=1}^{\dim E^{u}}$ that spans $E^{u}(\ox)$ (see \cite[Corollary 7.3]{Lian10}). 
		Moreover, for any fiexd $v=(v_i)_{i=1}^{\dim E^{u}}\in \R^{\dim E^{u}}$, define
		$$v_{\ox}=:\sum_{i=1}^{\dim E^{u}} v_i e_i{(\ox)}.$$
		Then, by \cite[Remark 9.15]{Lian10}, the mapping $(\ox)\mapsto g_{(\omega,x)}(v_{\ox})$ is measurable on the measurable subset 
		$\{(\ox):\Lnorm{v_{\ox}}{(\ox)}\leq \delta \ell(\ox)^{-1}\}$.
	\end{remark}
    These results are well-known for finite-dimensional systems. 
    For infinite-dimensional random dynamical systems, the stable and unstable manifolds theorem has been proven in \cite{Varzaneh21} for fields of Banach spaces, in \cite[Chapter 9]{Lian10} for Banach spaces, and in \cite{Li13} for Hilbert spaces. They are also mentioned in \cite{Young17} and \cite{Young16} for certain maps on a Banach space.
    
	For each $(\omega,x)\in \Gamma$. Let $\wW_{\delta}^{u}(\omega,x)=\graph(g_{(\omega,x)})$ and define the \textit{local unstable manifold} at $(\ox)$ by 
	$$W_{\delta}^{u}(\ox)=\exp_x \wW_{\delta}^{u}(\ox).$$ 
	The \textit{global unstable manifold} at $(\ox)$ defined to be
	$$W^{u}(\ox)=\bigcup_{n\geq 0}f^{n}_{\theta^{-n}(\omega)} W_{\delta}^{u}(\Phi^{-n}(\ox)).$$
	
	\begin{remark}\label{Rem:ivinm}
		Let $(\ox)\in \Gamma$ and $0<\delta<\delta_1'$, it is possible that $W_{\delta}^{u}(\ox)\subset A_{\om}$. 
        However, by Theorem \ref{Thm:Unstable}, we have $W_{\delta}^{u}(\ox)\subset f^n_{\te^{-n}\om} W_{\delta}^{u}(\Phi^{-n}(\ox))$ for each $n>0$. 
        Hence, for every $y\in W_{\delta}^{u}(\ox)$, we denote by $f^{-n}_{\om}(y) $
		the unique point in $W_{\delta}^{u}(\Phi^{-n}(\ox))$ such that $f^{n}_{\te^{-n}(\om)} (f^{-n}_{\om}(y))=y$.
	\end{remark}

	The key point in proving the unstable manifolds theorem is the result of graph transforms. 
    Let $(\ox)\in \Gamma$ and  $0<\delta \leq {\delta}'_1$. 
    In the following statements, the term "$\Lip$" refers to the Lipschitz constant with respect to the $\Lnorm{\cdot}{(\ox)}$ norm.
	Let
	$$\cW(\ox):=\{g:\widetilde{B}^{u}_{(\ox)}(\delta \ell(\ox)^{-1})\rightarrow E^{cs}(\ox) \ \text{with} \ g(0)=0, \ \Lip(g) \leq \frac{1}{10}\}.$$
	For any $g \in \cW(\ox)$, the graph transform of $g$ is denote by $\cT_{(\ox)}(g)$, which is a map $\cT_{(\ox)}(g): \widetilde{B}^{u}_{\Phi(\ox)}(\delta \ell(\Phi(\ox))^{-1}) \rightarrow E^{cs}(\Phi(\ox))$ satisfying
	$$\wf_{(\omega,x)}(\graph(g)) \supset \graph(\cT_{(\ox)}(g)).$$
	\begin{lemma}[\cite{Young17} Lemma 6.2]\label{Lem:Gtran}
		Let $(\omega,x)\in \Gamma$,
		\begin{enumerate}
			\item[(1)] for every $g \in \cW(\ox)$, $\cT_{(\ox)}(g)$ is well defined and belong to $\cW(\Phi(\ox))$.
			\item[(2)] there exists a constant $c \in (0,1)$ independent of $(\ox)$ such that for all $g_1, g_2 \in \cW(\ox)$
			$$|||\cT_{(\ox)}(g_1)-\cT_{(\ox)}(g_2) |||_{\Phi(\ox)}\leq c|||g_1-g_2|||_{(\ox)} $$
			where 
			$$|||g|||_{(\ox)}=\sup_{v\in \widetilde{B}^{u}_{(\ox)}(\delta \ell(\ox)^{-1})\setminus \{0\}} \dfrac{\Lnorm{g(v)}{(\ox)}}{\Lnorm{v}{(\ox)}}.$$
		\end{enumerate}
	\end{lemma}
 
	The notion of "backward graph transformation" is defined only when $E^{c}(\ox)={0}$. 
    We provide the formal definition below.
	
	Take $\delta >0$ be sufficiently small. 
    Let $(\omega,x)\in \Gamma$ and assume that $E^{c}(\ox)=\{0\}$, define
	\begin{align*}
		\cW^{s}(\ox):=\{h:\widetilde{B}^{s}_{(\ox)}(\delta \ell(\ox)^{-1}) &\rightarrow E^{u}(\ox) \ \text{with}\\ \Lnorm{h(0)}{(\ox)}
		&\leq \frac{\delta}{2} \ell(\ox)^{-1}, \ \Lip(h) \leq \frac{1}{10}\}.
	\end{align*}    
    For any $h\in \cW^{s}(\ox)$. If there exists ${h}'\in\cW^{s}(\Phi^{-1}(\ox))$ such that
	$$\wf_{\Phi^{-1}(\omega,x)}(\graph {h}')\subset \graph h,$$
	then we say that ${h}'$ is a backwards graph transformation of $h$ by $ \wf_{\Phi^{-1}(\omega,x)}$, written as $\cT^{s}_{(\ox)}(h):={h}'$. 
	\begin{lemma}[\cite{Young17} Lemma 6.4]\label{Lem:GtranS}
		Let $(\omega,x)\in \Gamma$. 
        Assume that $E^{c}(\ox)=0$ and $h\in \cW^{s}(\ox)$. Then
		\begin{enumerate}
			\item[(1)] $\cT^{s}_{(\ox)}(h)$ exists, and unique.
			\item[(2)] for $z_1,z_2\in \graph (\cT^{s}_{(\ox)}(h))$ ,
			$$\Lnorm{\wf_{\Phi^{-1}(\omega,x)}z_1-\wf_{\Phi^{-1}(\omega,x)}z_2}{(\ox)}\leq (e^{-\lambda}+\delta)\Lnorm{z_1-z_2}{\Phi^{-1}(\ox)}.$$
		\end{enumerate}
		(1) and (2) continue to hold for graph transforms by maps which are $C^{1}$ sufficiently close to $\wf_{\Phi^{-1}(\omega,x)}$ with respect to the Lyapunov norms.
	\end{lemma}
	For more details on graph transformations, please refer to \cite[Section 4]{Lian20}. 
    The proof of Lemma \ref{Lem:GtranS} can be found in the proof of Proposition 9 in \cite{Lian20}.
	
	Note that in the following lemma, we do not assume $E^c=\{0\}$, which is different from \cite[Lemma 6.3]{Young17}. 
	\begin{lemma} \label{Lem:CUnstable}
		For each $(\ox)\in \Gamma$ and for every $0<\delta\leq {\delta}'_1$. 
		\begin{enumerate}
			\item[(a)]	The set $W_{\delta}^{u}(\ox)$ can be equivalently described as follows:
			\begin{equation}\label{eq:CofW}
				\begin{aligned}
					\exp_{x}\{v\in \wB_{(\ox)}(\delta \ell(\ox)^{-1}): \forall n \in &\N, \exists v_n \in \wB_{\Phi^{-n}(\ox)}(\delta_1 \ell(\Phi^{-n}(\ox))^{-1}) \\ \text{such that} \  
					\widetilde{f}^{n}_{\Phi^{-n}(\ox)}v_n=v \ &\text{and} \ \limsup_{n\rightarrow \infty} \frac{1}{n}\log |v_n|\leq -\frac{\lambda}{2}\}.
				\end{aligned}
			\end{equation}
			Furthermore, if $E^{c}(\ox)=\{0\}$ then the description that
			$$\limsup_{n\rightarrow \infty} \frac{1}{n}\log |v_n|\leq -\frac{\lambda}{2}$$ in (\ref{eq:CofW}) can be omitted.
			\item[(b)] For every $y\in W_{\delta}^{u}(\ox)$, the tangent space $T_{y} W_{\delta}^{u}(\ox)$ characterized by
			\begin{equation}\label{Eq:CofT}
				\{v\in \cB_{y}: \ \forall n \in \N, Df^{-n}_{\omega}(y)v \ \text{exists and} \ 
				\limsup_{n\rightarrow \infty} \frac{1}{n}\log |Df^{-n}_{\omega}(y)v|\leq -\lambda\}.	
			\end{equation}
		\end{enumerate}
	\end{lemma}
	\begin{proof}
		To simplify the notation, write $\wf_{-n}:=\wf_{\Phi^{-n}(\ox)}$, $g_{-n}:=g_{\Phi^{-n}(\ox)}$, $\Lnorm{\cdot}{-n}:=\Lnorm{\cdot}{\Phi^{-n}(\ox)}$ and $\pi^{\tau}_{-n}:=\pi^{\tau}_{\Phi^{-n}(\ox)}$ for $n\geq 0$ and $\tau=u,cs$. 
		
		For item (a), the fact that $W^{u}_{\delta}(\ox)$ is a subset of the set \eqref{eq:CofW} follows from Theorem \ref{Thm:Unstable}. 
		It suffices to prove the reverse relation holds.
		
		If there exists $v$ belong to the set \eqref{eq:CofW} but $v\notin \wW^{u}_{\delta}(\ox)$,
		then there exists $u\in \wW^{u}_{\delta}(\ox)$ such that $\pi^{u}_{0}u=\pi^{u}_{0}v$ and $\pi^{cs}_{0}u\neq \pi^{cs}_{0}v$. 
		Therefore, for every $n\in \N$ by Corollary \ref{Cor:NlinP} we have
		\begin{equation}\label{eq:L42}
			\Lnorm{u_n-v_n}{-n}=\Lnorm{\pi^{cs}_{-n}(u_n-v_n)}{-n} \geq (e^{\varepsilon_0}+\delta)^{-n}\Lnorm{u-v}{0},
		\end{equation}
		where $u_n\in \wW^{u}_{\delta}(\Phi^{-n}(\ox))$ is the unique point satisfying $\wf^{n}_{-n}(u_n)=u$. 
		Together with $(\ref{eq:norms})$ we have
		$$\liminf_{n\rightarrow\infty}\frac{1}{n} \log |u_n-v_n|\geq -\log(e^{\varepsilon_0}+\delta),$$
		which is a contradiction with 
		$$\limsup_{n\rightarrow\infty}\frac{1}{n}\log|u_n|\leq -\frac{\lambda}{2} \ \text{and} \ \limsup_{n\rightarrow\infty}\frac{1}{n}\log|v_n|\leq -\frac{\lambda}{2},$$
		So, we must have $v\in \wW^{u}_{\delta}(\ox)$. 
		
		Moreover, if $E^{c}(\ox)=0$, then by Corollary \ref{Cor:NlinP}, for each $n\in \N$ the inequality (\ref{eq:L42}) is 
		$$\Lnorm{u_n-v_n}{\Phi^{-n}(\ox)}=\Lnorm{\pi^{s}_{-n}(u_n-v_n)}{-n}\geq (e^{-\lambda}+\delta)^{-n}\Lnorm{u-v}{(\ox)},$$ 
		which contradicts 
		$$\Lnorm{u_n-v_n}{\Phi^{-n}(\ox)} \leq  2\delta_1 \ell(\Phi^{-n}(\ox))^{-1}\leq 2\delta_1e^{n\varepsilon_1} \ell(\ox)^{-1}\quad \forall n\geq 0.$$
		This proves item (a). 
		
		For item (b), consider a fixed $y\in W_{\delta}^{u}(\ox)$.
		Let $y^{x}:=y-x$ and denote by $y^{x}_{-n}\in \wW^{u}_{\delta}(\Phi^{-n}(\ox))$ the unique point satisfying  $\wf^{n}_{-n}y^{x}_{-n}=y^{x}$ for each $n\geq 0$.
		We will use the $|\cdot|'$ to denote the operator norm of bounded linear operators with respect to the Lyapunov norms.
		 
		By Theorem \ref{Thm:Unstable}, for every $n\geq 0$, we have
		$$
		T_{y^{x}_{-n}} \wW_{\delta}^{u}(\Phi^{-n}(\ox))=(Id+Dg_{\Phi^{-n}(\ox)}(y^{x}_{-n}))E^{u}(\Phi^{-n}(\ox)).
		$$
		By item (a), we can choose $n_0\in \N$ such that $|y^{x}_{-n}|\leq e^{-\frac{n\lambda}{3}}$ for every $n\geq n_0$.		
		Then, for every $n\geq n_0$, by (\ref{eq:lox}) and Theorem \ref{Thm:Unstable}, one has 
		\begin{align*}
			|Dg_{-n}(y^{x}_{-n})|'&\leq  C_0 e^{2n\varepsilon_0} \ell(\ox)^{2} |y^{x}_{-n}|\\
			                     &\leq C_0e^{-n(\frac{\lambda}{3}-2\varepsilon_0)}\ell(\ox)^{2},                            
		\end{align*}
		and by (\ref{eq:lox}) and Lemma \ref{Lem:NlinP}, one has 
		\begin{equation}\label{eq:L43}
			\begin{aligned}
			 |D\wf_{-n}(y^{x}_{-n})-D\wf_{-n}(0)|'&\leq  e^{2n\varepsilon_0} \ell(\ox)^{2} |y^{x}_{-n}|\\
			 &\leq  e^{-n(\frac{\lambda}{3}-2\varepsilon_0)} \ell(\ox)^{2}.
			\end{aligned}
		\end{equation}
        By (\ref{eq:lox}) and the choice of the number ${\delta}'_1$,
        \begin{align*}
        	|D\wf_{-n}(y^{x}_{-n})|'&\leq3\ell(\Phi^{-n+1}(\ox)) \cdot |(f_{\te^{-n} \om}|_{B(A_{\te^{-n} \om},r_0)})|_{C^2}\\
        	                       &\leq 3e^{(n-1)\varepsilon_0}\ell(\ox)\ell(\Phi^{-n}(\ox)) \\
        	                       &\leq 3e^{2n\varepsilon_0}\ell(\ox)^{2}.
        \end{align*}
        Take $c>0$ arbitrarily small. 
        Enlarging $n_0$ if necessary such that for every $n\geq n_0$,
        $$\max\{3C_0\ell(\ox)^{2}, C_0, 1\}\cdot e^{-n(\frac{\lambda}{3}-4\varepsilon_0)} \ell(\ox)^{2}\leq \max\{1,\dfrac{c}{2}e^{\lambda}\}.$$
		For each $n\geq n_0$ and each $v=u+Dg_{-n}(y^{x}_{-n})u \in T_{y^{x}_{-n}} \wW_{\delta}^{u}(\Phi^{-n}(\ox))$, where $u\in E^{u}(\Phi^{-n}(\ox))$. 
		Since $\Lnorm{Dg_{-n}(y^{x}_{-n})}{}<1$, by the definition of the Lyapunov norms we have $\Lnorm{v}{-n}=\Lnorm{u}{-n}$. 
		Then,
		\begin{align*}
			\Lnorm{D\wf_{-n}(y^{x}_{-n})v}{-n+1}\geq&  \Lnorm{D\wf_{-n}(0)u}{-n+1}-\Lnorm{(D\wf_{-n}(y^{x}_{-n})-D\wf_{-n}(0))u}{-n+1}\\-&\Lnorm{D\wf_{-n}(y^x_{-n})Dg_{-n}(y^x_{-n})u}{-n+1}\\
			\geq&e^{\lambda}\Lnorm{u}{-n}-ce^{\lambda}\Lnorm{u}{-n}=(1-c)e^{\lambda}\Lnorm{v}{-n}.
		\end{align*}
	    Therefore, for every $v_0\in T_yW^{u}_{\delta}(\ox)$ and every $n\geq n_0$ we have
	    $$\Lnorm{Df^{-n-1}_{\om}(y)v_0}{-n-1}\leq (1-c)^{-1}e^{-\lambda} \Lnorm{Df^{-n}_{\om}(y)v_0}{-n}.$$
		Together with (\ref{eq:norms}), one has 
		$$\limsup_{k\rightarrow \infty}\frac{1}{k}\log|Df^{-k}_{\omega}(y)v_0|\leq -\lambda-\log(1-c).$$
		Since $c$ is arbitrariness, this prove that $T_yW^{u}_{\delta}(\ox)$ is a subset of (\ref{Eq:CofT}). 
		
		To prove the reverse holds. We denote $E(\omega,y)$ by 
		$$\{v\in \cB_{y}: \ \forall n \in \N, Df^{-n}_{\omega}(y)v \ \text{exists and} \ 
		\limsup_{n\rightarrow \infty} \frac{1}{n}\log |Df^{-n}_{\omega}(y)v|\leq -\lambda\}.$$
		It is easy to show that $E(\omega,y)$ is a subspace of $\cB_y$ and $Df_{\om}(y)E(\omega,y)=E(\Phi(\omega,y))$.
		Assume that there exists $v \in E(\oy)$ but $v\notin E^{u}(\oy)$. Denote by $v_{-n}:=Df^{-n}_{\omega}(y) v$. 
	    By \eqref{eq:L43}, for every $n\geq n_0$ we have
		$$|D\wf_{-n}(y^{x}_{-n})-D\wf_{-n}(0)|'\leq e^{-n(\frac{\lambda}{3}-2\varepsilon_0)} \ell(\ox)^{2}=C(n).$$ 
        For every $n\geq n_0$, every $u,w \in \cB$ with $\Lnorm{u-w}{-n}=\Lnorm{\pi^{u}_{-n} (u-w)}{-n}$, we have
		\begin{align*}
			&\Lnorm{\pi^{u}_{-n+1}(D\wf_{-n}(y^{x}_{-n})(u-w))}{-n+1}\\
			\geq &\Lnorm{\pi^{u}_{-n+1}(D\wf_{-n}(0)(u-w))}{-n+1}
			-\Lnorm{(D\wf_{-n}(y^{x}_{-n})-D\wf_{-n}(0))(u-w)}{-n+1}\\
			\geq &e^{\lambda}\Lnorm{\pi^{u}_{-n} (u-w)}{-n}-C(n)\Lnorm{(u-w)}{-n}\\
			\geq &(e^{\lambda}-C(n))\Lnorm{(u-w)}{-n},
		\end{align*}    	
		and similarly, we also have
		\begin{align}\label{eq:LSS}
			\Lnorm{\pi^{cs}_{-n+1}(D\wf_{-n}(y^{x}_{-n})(u-w))}{-n+1}\leq (e^{\varepsilon_0}+C(n)) \Lnorm{(u-w)}{-n}.
		\end{align}
		Take $n_1>n_0$ large enough, such that for every $n\geq n_1$ one has
		$$e^{\lambda}-C(n)> e^{2\varepsilon_0}> e^{\varepsilon_0}+C(n).$$
		For each $n\geq n_1$, if $u,w \in \cB$ with $\Lnorm{u-w}{-n}=\Lnorm{\pi^{u}_{-n} (u-w)}{-n}$ then
		$$\Lnorm{D\wf_{-n}(y^{x}_{-n})(u-w)}{-n+1} =\Lnorm{\pi^{u}_{-n+1}(D\wf_{-n}(y^{x}_{-n})(u-w))}{-n+1}.$$
		Let $u_{-n_1}\in T_{y^{x}_{-n_0}}\wW^{u}(\Phi^{-n_1}(\ox))$ with $\pi^{u}_{-n_1}u_{-n_1}=\pi^{u}_{-n_1}v_{-n_1}$. Note that $v\notin T_{y}W^{u}(\ox)$ implies that $v_{-n_1}\neq u_{-n_1}$.
		Denote by $u_{-n}:=D\wf^{n_1-n}_{-n_1}(y^x_{-n_1})u_{-n_1}$,
		we have $$\Lnorm{u_{-n_1}-v_{-n_1}}{-n_1}=\Lnorm{\pi^{cs}_{-n_1}(u_{-n_1}-v_{-n_1})}{-n_1}>0,$$
		then for every $n>n_1$ one has
		$$\Lnorm{u_{-n}-v_{-n}}{-n}=\Lnorm{\pi^{cs}_{-n}(u_{-n}-v_{-n})}{-n}.$$
		Using inequality (\ref{eq:LSS}) we have that for every $n>n_1$
		$$\Lnorm{u_{-n}-v_{-n}}{-n}\geq e^{-2(n-n_1)\varepsilon_0}\Lnorm{u_{-n_1}-v_{-n_1}}{-n_1}.$$
		This together with (\ref{eq:norms}) one has
		$$\liminf_{n\rightarrow \infty} \frac{1}{n}\log|u_{-n}-v_{-n}|\geq -2\varepsilon_0,$$
		which is a contradiction with 
		$$\limsup_{n\rightarrow\infty}\frac{1}{n}\log|u_{-n}|\leq -\lambda \ \text{and} \ \limsup_{n\rightarrow\infty}\frac{1}{n}\log|v_{-n}|\leq -\lambda.$$
		Hence, one has $v\in E^{u}(\oy)$. This completes the proof.
	\end{proof}
	
	The following corollary is direct.
	\begin{corollary}\label{Cor:CofT}
		For each $(\ox)\in\Gamma$ and every $(\oy)\in\Gamma$ with $y\in W^{u}(\ox)$. Then, we have $T_y W^{u}(\ox)=E^{u}(\oy)$.
	\end{corollary}

	\subsection{Distortion properties}\label{SEC:Distortion}
	Note that the local unstable manifold $W_{\delta}^{u}(\ox)$ is an embedded finite dimensional submanifold and, the global one is injectively immersed finite dimensional submanifold for $(\ox)\in \Gamma$. 
	Using the notion of induced volumes introduced in Section \ref{SEC:Induced volumes},  we denote by $\nu_{(\ox)}$ the induced volume $\nu_{W^{u}(\ox)}$ for each $(\ox)\in \Gamma$. 
	
	We also use $E^{u}(\omega,y)$ to denote the tangent space to $W^{u}(\ox)$ at $y\in W^{u}(\ox)$. 
	This notation is clear. 
	If the Lyapunov exponents for $(\omega,y)$ exist, then by Corollary \ref{Cor:CofT}, the unstable subspace of $(\oy)$ equivalent to the tangent space $T_y W^{u}(\ox)$.
	Note that $f_{\omega}$ maps $W^{u}(\ox)$ diffeomorphically onto $W^{u}(\Phi(\ox))$, so we have the relations between the measure $\nu_{\Phi(\ox)}$ and the measure $(f_{\omega})_{\ast}\nu_{(\ox)}$ on $W^{u}(\Phi(\ox))$ as follows (recall $f^{-1}_{\te\omega}(y)$ is the unique point such that $f_{\om}(f^{-1}_{\te\omega}(y))=y$)
	\begin{equation}\label{eq:Var}
		\dfrac{d((f_{\omega})_{\ast}\nu_{(\ox)})}{d\nu_{\Phi(\ox)}}(z)=\dfrac{1}{\det(Df_{\omega}(f^{-1}_{\te\omega}(z))|_{E^u(\omega,f^{-1}_{\te\omega}(z))})} \ \text{for} \ z\in W^{u}(\Phi(\ox)).
	\end{equation}
	The distortion property of the determinant which we have written below is important for us. 
	Let $W^{u}_{\loc}(\ox):=W^{u}_{{\delta}'_1}(\ox)$, where ${\delta}'_1$ be chosen in Theorem \ref{Thm:Unstable}.
	Recall the uniformly set is defined by $\Gamma_{\ell}:=\{(\ox): \ell(\ox)\leq \ell\}$.
	\begin{proposition}\label{Prop:dis}
		For every $\ell\geq 1$, there is a constant $D_{\ell}$ such that for every $(\ox)\in \Gamma_{\ell}$ the following holds:
		\begin{enumerate}
			\item[(a)] For all $y_1,y_2\in W^{u}_{\loc}(\ox)$ and all $n\geq 1$, we have
			\begin{equation}
				 \bigg| \log \dfrac{\det(Df^{n}_{\te^{-n}\om}(f^{-n}_{\om}y_1)|E^u(\te^{-n}\omega,f^{-n}_{\om}y_1))}{\det(Df^{n}_{\te^{-n}\om}(f^{-n}_{\omega}y_2)|E^u(\te^{-n}\om,f^{-n}_{\om}y_2))} \bigg| \leq D_{\ell} |y_1-y_2|.
			\end{equation}
			\item[(b)] For fixed ${x}'\in  W^{u}_{\loc}(\ox)$, the functions $y \mapsto \log \Delta_{N}({x}',y)$, where
			\begin{equation}\label{eq:Delta n}
				\Delta_{N}({x}',y):= \prod_{k=1}^{N} \dfrac{\det(Df_{\te^{-k}\om}(f^{-k}_{\om}{x}')|E^u(\te^{-k}\om,f^{-k}_{\omega}{x}'))}{\det(Df_{\te^{-k}\om}(f^{-k}_{\om}y)|E^u(\te^{-k}\om,f^{-k}_{\om}y))} \ N=1,2,\cdots
			\end{equation}
			defined for $y\in W^{u}_{\loc}(\ox)$ converges uniformly as $N\rightarrow \infty$ to a Lipschitz function with constant $\leq D_{\ell}$ in the $|\cdot|$ norm.
		\end{enumerate}
	\end{proposition}
	For certain maps on Banach spaces, the distortion property is given by \cite[Proposition 6.8]{Young17}, and the same method can be carried out in the random case. 
	We give the detailed proofs as follows.
	
	Fixed $(\ox)\in \Gamma$, for ease notations, we denote by $\wf_{-k}:=\wf_{\Phi^{-k}(\ox)}$, $\Lnorm{\cdot}{-k}:=\Lnorm{\cdot}{\Phi^{-k}(\ox)}$,
	$g_{-k}:=g_{\Phi^{-k}(\ox)}$, $\ell(-k):=\ell(\Phi^{-k}(\ox))$ for $k\in \N$, and denote by $|\cdot|'$ the operator norm of bounded linear operators with respect to the Lyapunov norms.
	For $k\in \Z$ and a finite-dimensional subspace $E\subset \cB_{f^{-k}_{\om}(x)}$, we denote by $m'_{E}$ the induced volume on $E$ from the Lyapunov norm $\Lnorm{\cdot}{-k}$, 
	and we use the notion ``$\det'$" to denote the determinant with respect to these volumes.
	
	\begin{lemma}[\cite{Young17}, Proposition 6.9]\label{Lem:Distortion}
		For any $\ell>0$, there is a constant $D'_{\ell}$ with the following property. 
		Fixed $(\ox)\in \Gamma$. 
		Then for any $z^1,z^2\in \graph g_{0}$ with $\Lnorm{z^1-z^2}{0}\leq \delta'_1 (D'_{\ell(0)})^{-1}$ and any $n\in \N$ we have that
		\begin{align*}
			\bigg| \log  \dfrac{\det'(D\wf^n_{-n}(z^1_{-n})|E^1_{-n})}{\det'(D\wf^n_{-n}(z^2_{-n})|E^2_{-n})} \bigg|\leq D'_{\ell(0)}\Lnorm{z^1-z^2}{0},
		\end{align*}
		where $z^i_{-n}$ is the unique point in $\graph g_{-n}$ with $D\wf^n_{-n}(z^i_{-n})=z^i$, and $E^i_{-n}$ is the tangent space to $\graph g_{-n}$ at $z^i_{-n}$.
	\end{lemma}
	\begin{proof}
		Fixed $n\in \N$. 
		By Proposition \ref{Prop:Det}, for arbitrary $z^1,z^2\in \graph g_{0}$ we have
		\begin{equation}\label{eq:4.51}
			 \dfrac{\det'(D\wf^n_{-n}(z^1_{-n})|E^1_{-n})}{\det'(D\wf^n_{-n}(z^2_{-n})|E^2_{-n})}
			=\prod_{k=1}^{n} \dfrac{\det'(D\wf_{-k}(z^1_{-k})|E^1_{-k})}{\det'(D\wf_{-k}(z^2_{-k})|E^2_{-k})}.
		\end{equation}

		We will use the following refinement version of Proposition \ref{Prop:CofDet}, which emphasizes the dependence of the constants and can be referred to \cite[Remark 3.16]{Young17}.
		\begin{lemma}[\cite{Young17} Lemma 6.10]\label{Lem:CofDet}
			For any $m\in \N$, there is a constant $C_m>1$ with the following property.
			Let $(\cB_1,|\cdot|_1)$, $(\cB_2,|\cdot|_2)$ be Banach spaces and fixed $M>1$. 
			Assume $A_1,A_2:\cB_1 \rightarrow \cB_2$ are two injective bounded linear operators and $E_1,E_2$ are $m$-dimensional subspaces of $\cB_1$ satisfied that 
			\begin{align*}
				|A_j|&, |(A_j|_{E_j})^{-1}|\leq M \quad \text{for} \ j=1,2, \\
				|A_1-A_2|&, d_H(E_1,E_2)\leq \dfrac{1}{C_m M^{10m}},
			\end{align*}
			then, we have 
			$$|\log \dfrac{\det(A_1|_{E_1})}{\det(A_2|_{E_2})}|\leq C_m M^{10m}(|A_1-A_2|+d_H(E_1,E_2)).$$
		\end{lemma}
		
		Let $m_u:=\dim E^{u}(\ox)$. 
		For each fixed $1\leq k\leq n$, we choose $M_k$ so that $D\wf_{-k}(z^i_{-k})$, $E^i_{-k}$, $i=1,2$ and $C_{m_u}$ satisfy the hypotheses of Lemma \ref{Lem:CofDet}. We need the following:
		\begin{align}
			&\Lnorm{D\wf_{-k}(z^i_{-k})}{},
			 \Lnorm{(D\wf_{-k}(z^i_{-k})|_{E^i_{-k}})^{-1}}{}\leq M_k \ \text{for} \ i=1,2, \label{eq:111} \\
			&\Lnorm{D\wf_{-k}(z^1_{-k})-D\wf_{-k}(z^2_{-k})}{}
			, d'_H(E^1_{-k},E^2_{-k})\leq C_{m_u} M_{k}^{10m_u},\label{eq:222}
		\end{align}
		where $d'_H$ refers to the Hausdorff distance in the adapted norm $\Lnorm{\cdot}{-k}$.
		
		First we choose $M_k$ so that \eqref{eq:111} holds. 
		Note that $\Lnorm{(D\wf_{-k}(z^i_{-k})|_{E^i_{-k}})^{-1}}{}\leq 1$ for $i=1,2$.
		By the choice of the function $\ell(\ox)$, for $i=1,2$ we have
		$$\Lnorm{D\wf_{-k}(z^i_{-k})}{}\leq 3\ell(-k)\cdot |(f_{\te^{-k}\om}|_{B(A_{\te^{-k}\om},r_0)})|_{C^2}\leq 3e^{2k\varepsilon_1}\ell(0)^{2}.$$
		So, on setting $M_k:=3e^{2k\varepsilon_1}\ell(0)^{2}$, $\eqref{eq:111}$ satisfied.
		
		Next, we estimate the two terms on \eqref{eq:222}.
		By Lemma \ref{Lem:NlinP} and Theorem \ref{Thm:Unstable} (3) we have
		\begin{align*}
			\Lnorm{D\wf_{-k}(z^1_{-k})-D\wf_{-k}(z^2_{-k})}{}&\leq \ell(k) \cdot\Lnorm{z^1_{-k}-z^2_{-k}}{-k}\\
			&\leq  (\dfrac{e^{\varepsilon_1}}{e^{\lambda}-\delta'_1})^k\cdot \ell(0) \cdot \Lnorm{z^1-z^2}{0}.
		\end{align*}
		By a simple computation and Theorem \ref{Thm:Unstable} (2),
		\begin{align*}
			d'_H(E^1_{-k},E^2_{-k})&\leq 2|Dg_{-k}(u^1_{-k})-Dg_{-k}(u^2_{-k})|'\\
			&\leq  2C_0\ell(-k)\cdot\Lnorm{u^1_{-k}-u^2_{-k}}{-k}\\
			&=     2C_0\ell(-k)\cdot\Lnorm{z^1_{-k}-z^2_{-k}}{-k}\\
			&\leq  (\dfrac{e^{\varepsilon_1}}{e^{\lambda}-\delta_1})^k\cdot 2C_0\ell(0) \cdot \Lnorm{z^1-z^2}{0},
		\end{align*}
		where $z^i_{-k}=u^i_{-k}+g_{-k}(u^i_{-k})$ for some $u^i_{-k}\in E^u(\Phi^{-k}(\ox))$ and $i=1,2$.
		To fulfill the inequality \eqref{eq:222}, we take $\Lnorm{z^1-z^2}{0}$ small enough so that
		$$\max\{2C_0,1\}(\dfrac{e^{\varepsilon_1}}{e^{\lambda}-\delta_1})^k\cdot \ell(0) \cdot \Lnorm{z^1-z^2}{0}\leq( C_{m_u} M^{10m_u}_{k})^{-1}.$$
		That is,
		$$\Lnorm{z^1-z^2}{0}\leq \Big(3^{10m_u}C_{m_u}\ell(0)^{20m_u+1} \max\{2C_0,1\}(\dfrac{e^{(20m_u+1)\varepsilon_1}}{e^{\lambda}-\delta'_1})^k \Big)^{-1}.$$
		By the choice of $\varepsilon_1\ll \lambda$ one may assume $e^{(20m_u+1)\varepsilon_1}<e^{\lambda}-\delta'_1$. 
		So, we take $D'_{\ell(0)}$ large enough for which
		$$D'_{\ell(0)}\geq \delta'_1 3^{10m_u}C_{m_u}\ell(0)^{20m_u+1} \max\{2C_0,1\},$$
		one has \eqref{eq:222} holds for every $z^1,z^2\in \graph g_{0}$ with $\Lnorm{z^1-z^2}{0}\leq \delta'_1 (D'_{\ell(0)}) ^{-1}$. 
		
		Therefore, by Lemma \ref{Lem:CofDet} for every $z^1,z^2\in \graph g_{0}$ with $\Lnorm{z^1-z^2}{0}\leq \delta'_1 (D'_{\ell(0)})^{-1}$ and every $1\leq k\leq n$ we have
		\begin{equation}\label{eq:45.2}
			\begin{aligned}
			\bigg| \log	\dfrac{\det'(D\wf_{-k}(z^1_{-k})|E^1_{-k})}{\det'(D\wf_{-k}(z^2_{-k})|E^2_{-k})}\bigg|\leq 
			K' \ell(0)^{20m_u+1} \Lnorm{z^1-z^2}{0}(\frac{e^{(20m_u+1)\varepsilon_1}}{e^{\lambda}-\delta'_1})^k,
			\end{aligned}
		\end{equation}
	    where $K':=C_{m_u}3^{10m_u}(2C_0+1)$.
	    Together with \eqref{eq:4.51}, one has 
	    \begin{align*}
	    	\bigg| \log  \dfrac{\det'(D\wf^n_{-n}(z^1_{-n})|E^1_{-n})}{\det'(D\wf^n_{-n}(z^2_{-n})|E^2_{-n})} \bigg| &\leq K' \ell(0)^{20m_u+1} \Lnorm{z^1-z^2}{0} \sum_{k=1}^{n}(\frac{e^{(20m_u+1)\varepsilon_1}}{e^{\lambda}-\delta'_1})^k\\
	    	&\leq K''\ell(0)^{20m_u+1}\Lnorm{z^1-z^2}{0},
	    \end{align*}
	    where $K''$ is independent of $n$ and $(\ox)$. 
	    By increasing $D'_{\ell(0)}$ once more so that $D'_{\ell(0)}\geq K''\ell(0)^{20m_u+1}$, the statement of Lemma \ref{Lem:Distortion} follows.
	\end{proof}
    We now give the proof of Proposition \ref{Prop:dis}.
    \begin{proof}[proof of Proposition \ref{Prop:dis}]
    	Fixed $(\ox)\in \Gamma_{\ell}$. For $y^1,y^2\in W^{u}_{\loc}(\ox)$. 
    	We denote by $z^i:=\exp_{x}^{-1}y_i$ and $y^{i}_{-k}:=\fo^{-k}y_i$ for $k\in \N$ and $i=1,2$, and we will use the same notation as in \ref{Lem:Distortion}, such as $z^i_{-n}$ and $E^{i}_{-n}$.
    	Since $\ell(\ox)\leq \ell$, by the choice of $D'_{\ell(0)}$ one may assume $D'_{\ell}\geq D'_{\ell(0)}$ so that Lemma \ref{Lem:CofDet} also holds for $D'_{\ell}$.
    	
    	For item (a), we first assume that $y^1,y^2\in W^{u}_{\loc}(\ox)$ with $\Lnorm{z^1-z^2}{0}\leq \delta'_1 (D'_{\ell})^{-1}$. 
    	For fixed $n\geq 1$, we have
    	$$ \dfrac{\det(D\wf^n_{-n}(z^1_{-n})|E^1_{-n})}{\det(D\wf^n_{-n}(z^2_{-n})|E^2_{-n})}=
    	\underbrace{ \dfrac{dm_{E^{1}_{0}}/dm'_{E^{1}_{0}} }{dm_{E^{2}_{0}}/dm'_{E^{2}_{0}}}}_{I}
    	\underbrace{ \dfrac{dm_{E^{1}_{-n}}/dm'_{E^{1}_{-n}} }{dm_{E^{1}_{-n}}/dm'_{E^{1}_{-n}}}}_{II}
    	\underbrace{\dfrac{\det'(D\wf^n_{-n}(z^1_{-n})|E^1_{-n})}{\det'(D\wf^n_{-n}(z^2_{-n})|E^2_{-n})}}_{III}.$$

    	For $i=1,2$, let $z^i=u^i_{0}+g_{-k}(u^i_{0})$ for some $u^i_{0}\in E^u(\ox)$. 
    	For term $ I $, observe that if $L:E^u(\ox)\rightarrow E^{cs}(\ox)$ is a linear map with $\Lnorm{L}{0}\leq 1$, then the definition of $\Lnorm{\cdot}{0}$ implies that $Id+L$ is an isometry on $E^u(\ox)$ with respect to the norm $\Lnorm{\cdot}{0}$. 
    	By the definition of induced volumes, one has
    	$$\dfrac{dm_{(Id+L)E^{u}(\ox)}}{dm'_{(Id+L)E^{u}(\ox)}}=
    	  \dfrac{m_{(Id+L)E^{u}(\ox)}((Id+L)B)}{m'_{(Id+L)E^{u}(\ox)}((Id+L)B)}=
    	  \dfrac{m_{(Id+L)E^{u}(\ox)}((Id+L)B)}{m'_{E^{u}(\ox)}(B)},$$
    	where $B$ is the unit closed ball of $E^{u}(\ox)$. 
    	Then, since $\Lip g_0\leq 1$, the term $ I $ is
    	$$I=\dfrac{\det((Id+Dg_0(u^1_{0}))|E^{u}(\ox))}{\det((Id+Dg_0(u^2_{0}))|E^{u}(\ox))},$$
    	where the determinants involved are in the natural norm $|\cdot|$.
    	
    	We now estimate $I$ by applying Lemma \ref{Lem:CofDet}. 
    	On the one hand
    	\begin{align*}
    		|Id+Dg_0(u^1_{0})|&\leq 3\ell \Lnorm{Id+Dg_0(u^1_{0})}{}=3\ell,\\
    		|(Id+Dg_0(u^1_{0})|_{E^{u}(\ox)})^{-1}|&\leq 3\ell
    		\Lnorm{(Id+Dg_0(u^1_{0})|_{E^{u}(\ox)})^{-1}}{}=3\ell.
    	\end{align*}
        So, we take $M=3\ell$. 
        On the other hand, by Theorem \ref{Thm:Unstable} (2)
       \begin{align*}
       	|Dg_0(u^1_{0})-Dg_0(u^2_{0})|&\leq 3\ell\Lnorm{Dg_0(u^1_{0})-Dg_0(u^2_{0})}{}\\
       	&\leq 3C_0\ell^2\Lnorm{u^1_{0}-u^2_{0}}{0}\\
       	&=3C_0\ell^2\Lnorm{z^1-z^2}{0}.
       \end{align*}
       It suffices to enlarge $D'_{\ell}$ if necessary for which $D'_{\ell}\geq (3\ell)^{10m_u+1}\ell\delta'_1C_0C_{m_u}$. 
       Then, by Lemma \ref{Lem:CofDet} we have
       $$|\log I|\leq C_{m_u}C_03^{10m_u+1}\ell^{10m_u+2}\Lnorm{z^1-z^2}{0}=K'''\ell^{10m_u+2}\Lnorm{z^1-z^2}{0}.$$
       
       Using the same argument, term $II$ has the following estimate
       \begin{equation}\label{eq:45.3}
       	\begin{aligned}
       	|\log II|&\leq K'''\ell(-n)^{10m_u+2}\Lnorm{z^1_{-n}-z^2_{-n}}{-n}\\
       	&\leq K'''\ell(0)^{10m_u+2} (\dfrac{e^{(10m_u+2)\varepsilon_1}}{e^{\lambda}-\delta'_1})^n \Lnorm{z^1-z^2}{0}\\
       	&\leq K'''\ell^{10m_u+2}\Lnorm{z^1-z^2}{0}.
       	\end{aligned}
       \end{equation}

       By Lemma \ref{Lem:Distortion}, we have
       $$|\log III|\leq D'_{\ell}\Lnorm{z^1-z^2}{0}.$$
       Therefore, take $D_{\ell}\geq \ell(2K'''\ell^{10m_u+2}+D'_{\ell})$. 
       This completes the proof of item (a) for $y^1,y^2\in W^{u}_{\loc}(\ox)$ with $\Lnorm{z^1-z^2}{0}\leq \delta'_1 (D'_{\ell})^{-1}$.
       
       For $y^1,y^2\in W^{u}_{\loc}(\ox)$ but $\Lnorm{z^1-z^2}{0}>\delta'_1 (D'_{\ell})^{-1}$, we only need to choose a finite sequence of points $\{u'_i\}^{N}_{i=0}$ on the line segment with the endpoints $u^1_{0}$ and $u^2_{0}$, so that $u'_0=u^1_{0}$, 
       $u'_N=u^2_{0}$ and $\Lnorm{u'_{i}-u'_{i-1}}{0}\leq \delta'_1 (D'_{\ell})^{-1}$ for any $1\leq i\leq N$.
       Take $\hat{z}^{i}=u'_{i}+g_0u'_{i}$, then 
       $\Lnorm{\hat{z}^{i}-\hat{z}^{i-1}}{0}=\Lnorm{u'_{i}-u'_{i-1}}{0}\leq \delta'_1 (D'_{\ell})^{-1}$ for any $1\leq i\leq N$. Using the argument above one has
       \begin{align*}
       	\bigg| \log \dfrac{\det(D\wf^n_{-n}(z^1_{-n})|E^1_{-n})}{\det(D\wf^n_{-n}(z^2_{-n})|E^2_{-n})}\bigg|
       	&=\bigg|\log\prod_{i=1}^{N} \dfrac{\det(D\wf^n_{-n}(\hat{z}^{i}_{-n})|\hat{E}^{i}_{-n})}{\det(D\wf^n_{-n}(\hat{z}^{i-1}_{-n})|\hat{E}^{i-1}_{-n})}\bigg|\\
       	&\leq D_{\ell}(\sum_{i=1}^{N}|\hat{z}^{i}-\hat{z}^{i-1}|)
       	=D_{\ell}|z^1-z^2|,
       \end{align*}
     where $\hat{z}^i_{-n}$ is the unique point in $\graph g_{-n}$ with
     $D\wf^n_{-n}(\hat{z}^i_{-n})=\hat{z}^i$, and $\hat{E}^i_{-n}$ is the tangent space to $\graph g_{-n}$ at $\hat{z}^i_{-n}$ for $0\leq i\leq N$.
     This completes the proof of item (a).
     
     For item (b), it suffices to show that $y \mapsto \log \Delta_{k}({x}',y), k\in \N$ is a Cauchy sequence. Note that for $M<N$,
     $$\bigg| \log \dfrac{\Delta_{N}({x}',y)}{\Delta_{M}({x}',y)}\bigg|= \bigg| \log \prod_{k=M+1}^{N} \dfrac{\det(Df_{\te^{-k}\om}(f^{-k}_{\om}{x}')|E^u(\te^{-k}\om,f^{-k}_{\omega}{x}'))}{\det(Df_{\te^{-k}\om}(f^{-k}_{\om}y)|E^u(\te^{-k}\om,f^{-k}_{\om}y))}\bigg|.$$
     Using the estimations in \eqref{eq:45.2} and \eqref{eq:45.3} one can show the right term above is upper-bounded by
     \begin{align*}
     	&\leq (2K'''\ell(\ox)^{10m_u+2}+K'\ell(\ox)^{20m_u+1})\cdot 3\sum_{i=M}^{N}(\dfrac{e^{(20m_u+1)\varepsilon_1}}{e^{\lambda}-\delta'_1})^i \Lnorm{x'-y}{0}\\& \leq(2K'''\ell(\ox)^{10m_u+1}+K'\ell(\ox)^{20m_u})\cdot 6\sum_{i=M}^{N}(\dfrac{e^{(20m_u+1)\varepsilon_1}}{e^{\lambda}-\delta'_1})^i \delta'_1
     \end{align*}
     for every $y\in W^{u}_{\loc}(\ox)$. This completes the proof of item (b).
    \end{proof}

	\subsection{Unstable stacks of local unstable manifolds}\label{SEC:4.2}
	Now, we discuss the regularity of local unstable manifolds. 
	Notice that the functions $g_{(\ox)}$ and $g_{(\oy)}$ belong to the different coordinate system if $x\neq y$.
    To address this, we need to construct $g^{y}_{(\ox)}$, which represents the coordinate transformation of $g_{(\oy)}$ under the coordinate system $E^{u}(\ox)\oplus E^{cs}(\ox)$.
	
	To avoid using multiple norms that depend on the point, we let $B^{\tau}_{(\ox)}(r)=\{v\in E^{\tau}(\ox):|v|\leq r\}$ for $\tau=u,cs$.
	
	Recall the uniformly set $\Gamma_{\ell}:=\{(\ox): \ell(\ox)\leq \ell\}$ and the $\om$-section $(\Gamma_{\ell})_{\om}:=\{x\in\cB:(\ox)\in \Gamma_{\ell}\}$.

    \begin{figure}[H]
	\includegraphics[width=1.0\linewidth]{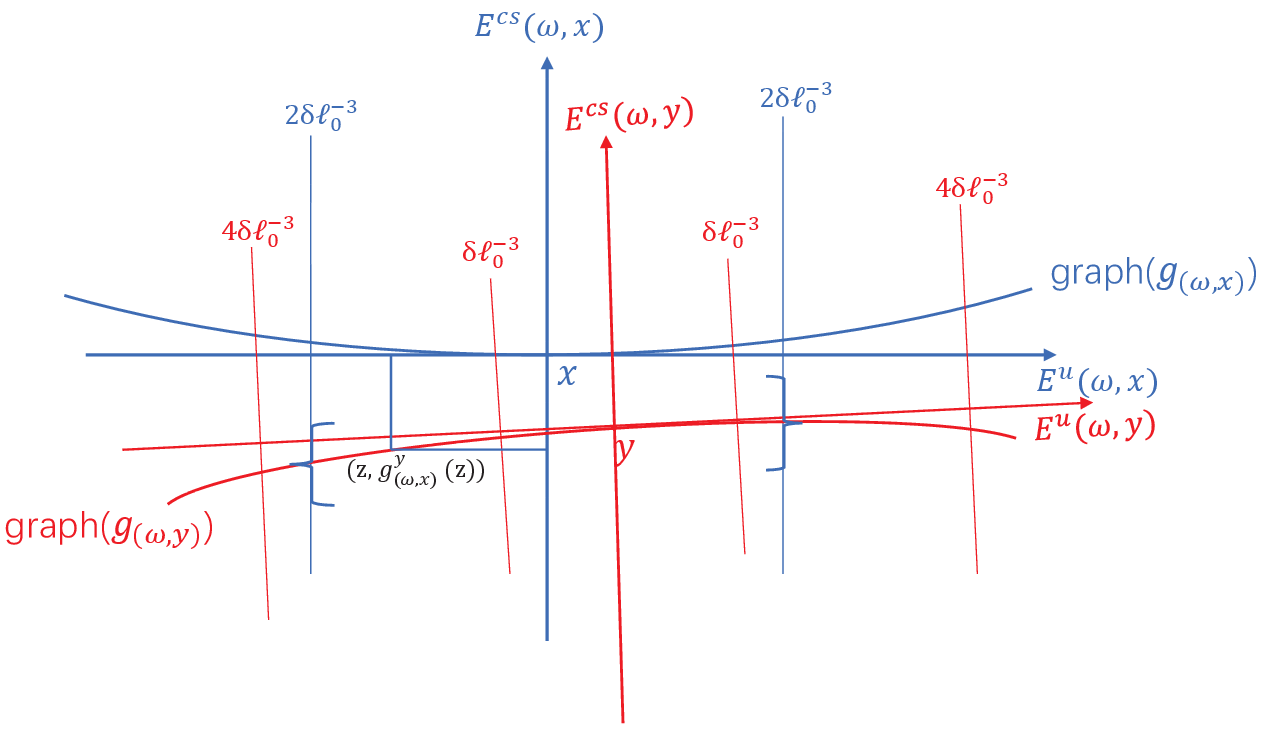}
	\caption{the coordinate transformation of $g_{(\oy)}$ under $E^{u}(\ox)\oplus E^{cs}(\ox)$} 
	\label{Fig: unstablestack}
    \end{figure}

	\begin{lemma}\label{Unstacks}
		Let $\delta>0$ be sufficiently small.
        Let $\ell_0$ be fixed with $\mu(\Gamma_{\ell_0})>0$, and fixed $(\om, x_{\om})\in \Gamma_{\ell_0}$. 
		For every $r>0$ and every $(\ox) \in \Gamma_{\ell_0}$ we let
		$$U_{\omega}(x,r):=(\Gamma_{\ell_0})_{\omega} \cap \{z : |z-x| \leq r\},$$
		and let $g_{(\ox)}$ be as in Theorem \ref{Thm:Unstable}.
		Then, there exists $r_1(\ell_0,\delta,\omega)>0$ such that 
  \begin{enumerate}
		\item[(1)]for $y\in U_{\omega}(x_{\om},r_1)$, there exists a continuous map 
		$g_{(\om,x_{\om})}^{y}: B^{u}_{(\om,x_{\om})}(2\delta \ell_0^{-3}) \rightarrow E^{cs}(\om,x_{\om})$ 
		such that 
		\begin{align*}
		\exp_{y}\graph(g_{(\oy)}|_{B^{u}_{(\oy)}(\delta \ell_{0}^{-3})})
		&\subset
		\exp_{x_{\om}}\graph(g_{(\om,x_{\om})}^{y})\\                          
		&\subset 
		\exp_{y}\graph(g_{(\oy)}|_{B^{u}_{(\oy)}(4\delta \ell_{0}^{-3})}).
		\end{align*}
		\item[(2)] 
		the mapping  $\Theta_{\om}=\Theta_{\om}^{x_{\om}}:U_{\omega}(x_{\om},r_1)\rightarrow C(B_{(\omega, x_{\om})}^{u}(2\delta \ell_0^{-3}),E^{cs}(\om,x_{\om}))$ defined by $\Theta_{\omega}(y)=g_{(\om,x_{\om})}^{y}$ is continuous with respect to the $C^0$ norm $||\cdot||$.
    \end{enumerate}
	\end{lemma}
   
	\begin{proof}
		With a minor modification of the proof of \cite[Lemma 6.5]{Young17}, we give the proof of this lemma in the following.

		Let $(\ox),(\oy)\in \Gamma_{\ell_0}$. 
		For ease notations, we denote by $g_y:=g_{(\oy)}$, $E^{\tau}_y:=E^{\tau}(\oy)$, $\pi^{\tau}_{y}:=\pi^{\tau}_{(\oy)}$ and 
		$W^{u}_{x,\delta,y}:=\exp^{-1}_x \exp_y \graph(g_y|_{\wB^{u}_{(\oy)}(\delta \ell^{-2}_{0})})$ for $\tau=u,cs$ and $\delta>0$.
		
		We first show that, $\pi^{u}_{x}$ is invertible on $W^{u}_{x,4\delta,y}$ when $|x-y|$ and $\delta$ are sufficiently small.  
		This is guaranteed if $|\pi^{cs}_{x}(w_1-w_2)|<|w_1-w_2|$ for all $w_1, w_2\in W^{u}_{x,4\delta,y}$.
		By (\ref{eq:lox}) and Theorem \ref{Thm:Unstable}, for every $u_1, u_2\in \graph(g_y|_{\wB^{u}_{(\oy)}(4\delta \ell^{-2}_{0})})$ we have
		\begin{align*}
			|\pi^{cs}_{x}(u_1-u_2)|&=|\pi^{cs}_{x}(\pi^{u}_{y}(u_1-u_2))+\pi^{cs}_{x}(g_y(\pi^{u}_{y} u_1)-g_y(\pi^{u}_{y} u_2))|\\
			                    &\leq |\pi^{cs}_x|_{E^{u}_y}|  \cdot |\pi^{u}_{y}(u_1-u_2)|+
			                    |\pi^{cs}_x|_{E^{cs}_y}| \cdot |g_y(\pi^{u}_{y} u_1)-g_y(\pi^{u}_{y} u_2)|\\
			                 &\leq \ell_{0}|\pi^{cs}_x|_{E^{u}_y}|\cdot |u_1-u_2|+|\pi^{cs}_x|_{E^{cs}_y}| \cdot |g_y(\pi^{u}_{y} u_1)-g_y(\pi^{u}_{y} u_2)|.
		\end{align*}
		By (\ref{eq:lox}) and the definition of the Lyapunov norms,
		\begin{align*}
			|g_y(\pi^{u}_{y} u_1)-g_y(\pi^{u}_{y} u_2)|&\leq 3 \Lnorm{g_y(\pi^{u}_{y} u_1)-g_y(\pi^{u}_{y} u_2)}{(\oy)}\\
			               &\leq 3{\Lip}'_{y}(g_y) \cdot \Lnorm{\pi^{u}_{y}(u_1-u_2)}{(\oy)} \\
			               &\leq 3(\sup_{w\in \wB^{u}_{(\oy)}(4\delta\ell_0^{-2})} \Lnorm{Dg_y (w)}{}) \cdot\Lnorm{u_1-u_2}{(\oy)}\\
			               &\leq 3\ell_{0} (\sup_{w\in \wB^{u}_{(\oy)}(4\delta\ell_0^{-2})} \Lnorm{Dg_y (w)}{}) \cdot |u_1-u_2|\\
			               &\leq 12\delta \ell_{0}^{-1} {\Lip}'_{y}(Dg_y) \cdot |u_1-u_2|,
		\end{align*}
		where ``${\Lip}'_{y}$" is the Lipschitz constant with respect to the norm $\Lnorm{\cdot}{(\oy)}$.
		By Theorem \ref{Thm:Unstable}, we have ${\Lip}'_{y}(Dg_y)\leq C_0\ell_0$.
		Therefore, we have 
		$$|g_y(\pi^{u}_{y} u_1)-g_y(\pi^{u}_{y} u_2)|\leq 12C_0\delta|u_1-u_2|,$$
		and we may assume that the number $12C_0\delta\ll 1$.
		
		By Lemma \ref{Lem:ProDh}, we have $|\pi^{cs}_x|_{E^{u}_y}|\rightarrow 0, |\pi^{cs}_x|_{E^{cs}_y}| \rightarrow 1$ whenever 
		$d_H(E^u_y,E^u_x)\rightarrow 0$ and $d_H(E^{cs}_y,E^{cs}_x)\rightarrow 0$.
		Thus, with Proposition \ref{Prop:CCS}, take ${r}'_1(\ell_0, \om)$ small enough such that for every $x, y\in U_{\om}(x_{\om},{r}'_1)$, one has $d_H(E^u_y,E^u_x)$ and $d_H(E^{cs}_y,E^{cs}_x)$ are sufficiently small so that 
		$$(\ell_0|\pi^{cs}_x|_{E^{u}_y}|+|\pi^{cs}_x|_{E^{cs}_y}|\cdot 12C_0\delta)<1.$$
		
		We define the function $g^{y}_{(\ox)}$ by $g^{y}_{(\ox)}(u):=\pi^{cs}_x((\pi^{u}_x)^{-1}(u))$ for $u\in \pi^{u}_x(W^{u}_{x,4\delta,y})$, where $(\pi^{u}_x)^{-1}(u)$ is the unique point in $W^{u}_{x,4\delta,y}$ such that $\pi^{u}_x((\pi^{u}_x)^{-1}(u))=u$.
		Take $r_1(\ell_0,\delta,\om)<{r}'_1$ small enough, so that for every $x, y\in U_{\om}(x_{\om},r_1)$ one has
		\begin{align*}
			 \pi^u_x (\graph(g_y|_{B^{u}_{(\oy)}(\delta \ell_{0}^{-3})})+y-x)
			&\subset B^{u}_{(\ox)}(3\delta \ell_0^{-3})\\
			&\subset \pi^u_x(\graph(g_y|_{B^{u}_{(\oy)}(4\delta \ell_{0}^{-3})})+y-x).
		\end{align*}
	    Hence, the function $g^{y}_{(\ox)}$	is well-defined and continuous on $B^{u}_{(\ox)}(3\delta \ell_0^{-3})$.
	    	
		In particularly, for each $y\in U_{\om}(x_{\om},r_1)$ we consider the function $g^{y}_{(\om,x_{\om})}$ on $B^{u}_{(\om,x_{\om})}(2\delta \ell_0^{-3})$. 
		By the construction $g^{y}_{(\om,x_{\om})}$ satisfying item (1).
    
		For item (2), fix $x\in U_{\omega}(x_{\om},r_1)$, let $\{y^n\}_{n\in \N} \subset U_{\omega}(x_{\om},r_1)$ be a sequence with $y^n \rightarrow x$. 
		Shirking $r_1(\ell_0,\delta,\om)$ if necessary, such that for every $x_1, x_2\in U_{\om}(x_{\om},r_1)$ one has
		$$\exp_{x_{\omega}} \graph(g^{x_1}_{(\om,x_{\om})}|_{B^{u}_{(\om,x_{\om})}(2\delta \ell_0^{-3})})\subset 
		  \exp_{x} \graph( g^{x_1}_{(\om,x_2)}|_{B^{u}_{(\om,x_2)}(3\delta \ell_0^{-3})}).$$
		Therefore, to prove $||\Theta_{\om}(y^n)-\Theta_{\om}(x)||\rightarrow 0$. 
		It suffices to prove that $||g^{y^n}_{(\ox)}-g_{(\ox)}||\rightarrow 0$ on $B^{u}_{(\ox)}(3\delta \ell_0^{-3})$.
		Let $\gamma>0$ be fixed, we prove that $||g^{y^n}_{(\ox)}-g_{(\ox)}||<\gamma$ for $n$ large enough. 
		
		For $k,n\in \N$, let $x_{-k}=f^{-k}_{\omega}x$, $y^n_{-k}=f^{-k}_{\omega}y^n$,
	    let
		$$\textbf{0}_{y^n_{-k}}: \wB^{u}_{\Phi^{-k}(\omega,y^n)}(4\delta \ell(\Phi^{-k}(\omega,y^n))^{-1})\rightarrow E^{cs}(\Phi^{-k}(\omega,y^n))$$ 
		be the function that is identically equal to $0$, and let 
		$$\phi_{y^n_{-k}}:=\cT_{\Phi^{-1}(\omega,y^n)}\circ\cdots\circ \cT_{\Phi^{-k}(\omega,y^n)} \textbf{0}_{y^n_{-k}},$$ 
		where $\cT$ is the graph transformation as in Lemma \ref{Lem:Gtran}. 
		Likewise, define $\phi_{x_{-k}}$.
		 	
		For fixed $k\in \N$, we have $y_{-k}^n, x_{-k}\in (\Gamma_{\ell_0e^{k\varepsilon_1}})_{\te^{-k}\omega}$ for every $n\in \N$.
	    By Proposition \ref{Prop:CCS}, one has $E^u(\te^{-k}\omega,y_{-k}^n)\rightarrow E^u(\te^{-k}\omega,x_{-k})$ as $n\rightarrow \infty$. 
	    This implies that
	    \begin{equation}\label{eq:444}
	    \exp_{y^n_{-k}} B^{u}_{\Phi^{-k}(\om,y^n)}(10\delta \ell_0^{-3}) \rightarrow \exp_{x_{-k}} B^{u}_{\Phi^{-k}(\ox)}(10\delta \ell_0^{-3}) \ \text{as} \ n\rightarrow \infty,
	    \end{equation}
	    and so, their $f_{\te^{-k} \om}^{k}$-images as well.
	    Then, analogous to  item (1), for sufficiently large $n$ the coordinate transformation $\phi^{y^n_{-k}}_{(\ox)}$ of $\phi_{y^n_{-k}}$ is defined on $B^{u}_{(\ox)}(3\delta \ell_0^{-3})$, which satisfied that
	    $$\exp_{x}\graph(\phi^{y^n_{-k}}_{(\ox)})=\exp_{y}\graph(\phi_{y^n_{-k}}|_{Dom(y^n)})$$
	    for some $Dom(y^n)\subset B^{u}_{(\om,y^n)}(10\delta \ell_0^{-3})$.
	    Moreover, also by (\ref{eq:444}) we have
	    $||\phi_{x_{-k}}-\phi^{y^n_{-k}}_{(\ox)}||\rightarrow 0$ on $B^{u}_{(\ox)}(3\delta \ell_0^{-3})$ as $n\rightarrow \infty$. 
	      
		By Lemma \ref{Lem:Gtran}, and the uniform equivalence of the $|\cdot|$ and $\Lnorm{\cdot}{(\oy)}$ norms on $\Gamma_{\ell_0}$. There exists $k_0:=k_0(\ell_{0},\gamma)\in \N$ such that for all sufficiently large $n$ we have
		$$||\phi^{y^n_{-k_0}}_{(\ox)}-g_{(\ox)}^{y^n}||<\gamma/3 \ \text{and}\
		  ||\phi_{x_{-k_0}}-g_{(\ox)}||<\gamma/3$$
		on $B^{u}_{(\ox)}(3\delta \ell_0^{-3})$.		
		To finish, we estimate $||g_{(\ox)}^{y^n}-g_{(\ox)}||$ by
		\begin{align*}
			||g_{(\ox)}^{y^n}-g_{(\ox)}^{x}||\leq ||g_{(\ox)}^{y^n}-\phi^{y^n_{-k_0}}_{(\ox)}||+||\phi^{y^n_{-k_0}}_{(\ox)}-\phi_{x_{-k_0}}||+||g_{(\ox)}-\phi_{x_{-k_0}}||.
		\end{align*}
		The first and third terms above are $<\gamma/3$, and for $n$ large enough the middle term is $< \gamma/3$. 
		This completes the proof.
	\end{proof}
	We referred to sets of the form
	\begin{equation}\label{eq:stack}
		\cS(\om,x_{\om})=\bigcup_{y\in \bar{U}_{\omega}} \exp_{x_{\om}}(\graph \Theta_{\omega}(y)),
	\end{equation}
	where $x_{\om},U_{\omega}(x_{\om},r_1)$ and $\Theta_{\om}$ as in above and $\bar{U}_{\omega}\subset U_{\omega}(x_{\om},r_1)$ is a compact subset, 
	as a \textit{unstable stack of local unstable manifolds} through the point $(\om,x_{\om})$. 
	By Lemma \ref{Unstacks} (2), the set $\cS(\om,x_{\om})$ is a compact subset of $\cB$.
	\begin{lemma}\label{UStaL}
		For fixed $(\om,x_{\om})\in \Gamma_{\ell_0}$. Let 
		$\Theta_{\omega}$ be defined in Lemma \ref{Unstacks}. 
		For $y_i \in U_{\omega}(x_{\om},r_1), i=1,2$ with 
		$$\exp_{x_{\om}}(\graph \Theta_{\omega}(y_1))\neq \exp_{x_{\om}}(\graph \Theta_{\omega}(y_2)).$$
		Then, we have 
		$$\exp_{x_{\om}}(\graph \Theta_{\omega}(y_1)) \cap \exp_{x_{\om}}(\graph \Theta_{\omega}(y_2))=\varnothing.$$
	\end{lemma}
	\begin{proof}
		Denote by $G_1:=\exp_{x_{\om}}(\graph \Theta_{\omega}(y_1))$ and $G_2:=\exp_{x_{\om}}(\graph \Theta_{\omega}(y_2))$.
		If there exists $z \in G_1\cap G_2$, then $z$ contained in some local unstable manifolds of $(\omega,y_1)$ and $(\omega,y_2)$.
		For each ${z}'\in G_1$, by Lemma \ref{Lem:CUnstable} one has 
		\begin{align*}
			\limsup_{n\rightarrow\infty}\frac{1}{n}
			\log|\fo^{-n}({z}')- \fo^{-n}(y_2)|
			&\leq \limsup_{n\rightarrow\infty} \frac{1}{n}\log(|\fo^{-n}({z}')- \fo^{-n}(y_1)|\\
		   +&|\fo^{-n}(z)- \fo^{-n}(y_1)|+|\fo^{-n}(z)- \fo^{-n}(y_2)|) \leq -\frac{\lambda}{2}.
		\end{align*}
		Therefore, also by Lemma \ref{Lem:CUnstable}, one can easily to show that ${z}'$ contained in some local unstable manifold of $(\omega,y_2)$. 
		This implies that ${z}'\in G_2$ and so, $G_1 \subset G_2$. 
		Symmetrically, we also have $G_1 \supset G_2$. 
	\end{proof}
	\begin{lemma}\label{GUStacks}
		Let $\cS(\omega,x_{\om})$ be defined above, and let $\Sigma_{(\omega,x_{\om})}:=\exp_{x_{\om}}^{-1}\cS(\omega,x_{\om}) \cap E^{cs}(\omega,x_{\om})$. 
		Then $\cS(\omega,x_{\om})$ is homeomorphic to  
		$\Sigma_{(\omega,x_{\om})}\times B^{u}_{(\omega,x_{\om})}(2\delta \ell^{-3}_0)$
		under the mapping $\Psi(\sigma,u):=\exp_{x_{\omega}}(u+g^{\omega}_{\sigma}(u))$,
		where $g^{\omega}_{\sigma}=\Theta_{\omega}(y)$ is the unique leaf of $\cS_{\omega}$ such that $\Theta_{\omega}(y)(0)=\sigma$. 
	\end{lemma}
	\begin{proof}				
		The  proof as same as the proof of \cite[Lemma 6.7]{Young17}.
		By Lemma \ref{UStaL}, the mapping $g^{\omega}_{\sigma}$ is well defined. 
		Since $\Psi$ is a bijection from two compact sets, we only need to show that $\Psi$ is continuous. 
		This is guaranteed by the fact that $\sigma\mapsto  g^{\omega}_{\sigma}$ is continuous, which in turn is a direct consequence of Lemma \ref{Unstacks}. 
	\end{proof}
 
	\begin{figure}[H]
	\includegraphics[width=1.0\linewidth]{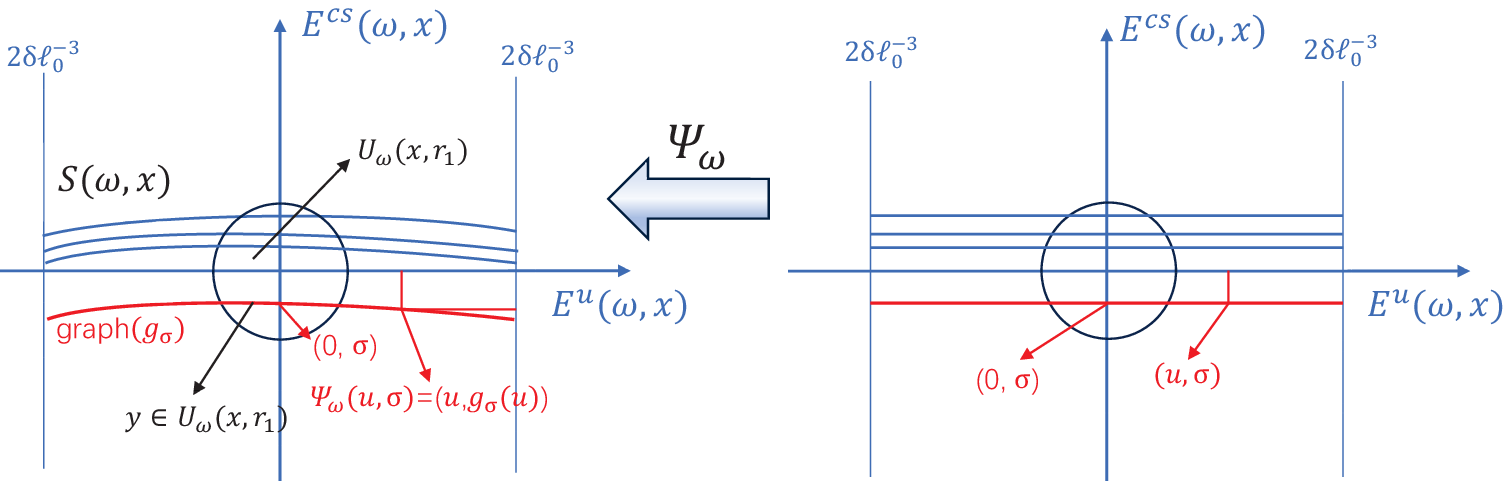}
	\caption{geometric picture of $\cS(\om,x)$}
	\label{Fig: picture of stacks}
    \end{figure}

	\subsection{Random unstable stacks}\label{SEC:Random US}
	We now construct the \textit{random unstable stacks}, which are random compact subset of $\Omega\times \cB$ with positive $\mu$-measure, and for each fixed $w$, the $\omega$-section $\cS_{\omega}$ has a measurable partition consisting of local unstable manifolds.
	
	Choose $\ell_0$ large enough such that $\mu(\Gamma_{\ell_0})>0$ and fix $\delta>0$ sufficiently small. 
	By Proposition $\ref{Prop:Trandom compact}$, for each $\epsilon>0$ there exists a compact subset $A_{\epsilon}\subset \cB$ such that
	$\bP\{\omega:A_{\omega}\subset A_{\epsilon}\}>1-\epsilon$. 
	This implies that $\mu(\Omega\times A_{\epsilon})>1-\epsilon$ and we can assume that 
	$$\mu(\Gamma_{\ell_0}\cap (\Omega\times A_{\epsilon}))>0.$$
	Let $r_2\ll \delta$, by compactness we choose a reference point $x_{\ast}$ such that 
	$$\mu(\Gamma_{\ell_0}\cap (\Omega\times (B(x_{\ast},r_2)\cap A_{\epsilon}))>0,$$
	where $B(x_{\ast},r_2):=\{y:|y-x_{\ast}|\leq r_2\}$.
	Denote by $A_{\ast}:= B(x_{\ast},r_2)\cap A_{\epsilon}$ which is also compact, and denote by $W^{u}_{\delta,(\oy)}:=\exp_{y}\graph(g_{(\omega,y)}|_{B^{u}_{(\oy)}(\delta \ell_{0}^{-3})})$. 
	Let 
	$$S_{\omega}:=\bigcup \{W^{u}_{\delta,(\oy)}\cap B(x_{\ast},2r_2): y\in A_{\ast} \ \text{such that} \ (\omega,y)\in \Gamma_{\ell_0} \},$$
	By chosen $r_2$ small enough (depend only on $\delta, \ell_{0}$), 
	we can ensure that: for each $y_i \in A_{\ast}$ with $(\omega,y_i)\in \Gamma_{\ell_0}$, $i=1,2$,
	$W^{u}_{\delta,(\omega,y_i)}\cap B(x_{\ast},2r_2)$ is connected, and either 
	$$W^{u}_{\delta,(\omega,y_1)}\cap B(x_{\ast},2r_2)=W^{u}_{\delta,(\omega,y_2)}\cap B(x_{\ast},2r_2),$$ 
	or the two terms in the above equation are disjoint.
	
	By Lemma \ref{Unstacks}, for fixed $\omega$, the subset $\cS_{\omega}$ is compact, and the mapping $y \mapsto W^{u}_{\delta,(\omega,x_y)}\cap B(x_{\ast},2r_2)$ is continuous on $\cS_{\om}$ with respect to the Hausdorff distance between compact subsets of $\cB$. Here $x_y\in A_{\ast}\cap (\Gamma_{\ell_{0}})_{\om}$ is the point such that $y\in W^{u}_{\delta,(\omega,x_y)}\cap B(x_{\ast},2r_2)$.
 
	To see this, we can choose a small number $r_{1}(\omega)$ such that for each $x\in (\Gamma_{\ell_0})_{\om}$ the unstable stack 
	$$\cS(\ox):=\bigcup_{y\in \bar{U}^{x}_{\om}} \exp_{x}(\graph \Theta_{\omega}^{x}(y))$$ 
    is defined, where $\bar{U}^{x}_{\om}=A_{\ast}\cap U_{\omega}(x,r_1)$ is compact and $\Theta_{\om}^{x}$ is as defined  in Lemma \ref{Unstacks}. 
	By choosing $r_2\ll \delta$, we can assume that for each $y\in \bar{U}^{x}_{\om}$
	\begin{equation}\label{eq:cWW}
	     W^{u}_{\delta,(\oy)}\cap B(x_{\ast},2r_2)=(\exp_{x}(\graph \Theta_{\omega}^{x}(y)))\cap B(x_{\ast},2r_2).
	\end{equation}
	Then, by compactness, there is a finite subset $\{x_{\omega,i}\}^{k(\om)}_{i=1}\subset (\Gamma_{\ell_0})_{\om} \cap A_{\ast}$
	such that 
	\begin{equation}\label{eq:cSom}
		\cS_{\omega}= \bigcup^{k(\om)}_{i=1} \cS(\om, x_{\omega,i})\cap B(x_{\ast},2r_2).
	\end{equation}
    Since each $\cS(\om, x_{\omega,i})$ is compact, it follows that $\cS_{\omega}$ is also compact.
    
	Fix $i\in \{1,\cdots,k(\om)\}$ and let $\Theta_{\omega}^{i}=\Theta_{\omega}^{x_{\om,i}}$. By Lemma \ref{Unstacks} and \ref{GUStacks}, it is easy to deduce that
	$y \mapsto \exp_{x_{\om,i}}\graph(\Theta_{\omega}^{i}(x_y))$ is continuous on $\cS(\om, x_{\omega,i})$, where $x_y\in \bar{U}^{x_{\om,i}}_{\om}$ is the point such that $y\in \exp_{x_{\omega,i}}\graph\Theta_{\omega}^{i}(x_y)$.
	Note that $\graph\Theta_{\omega}^{i}(z)$ is connected for any $z\in \bar{U}^{x_{\om,i}}_{\om}$. 
	Then, we have the mapping $y \mapsto (\exp_{x_{\om,i}}\graph\Theta_{\omega}^{i}(x_y))\cap B(x_{\ast},2r_2)$ is also continuous on $\cS(\om, x_{\omega,i})$.
	Together with \eqref{eq:cWW} and \eqref{eq:cSom}, we deduce that the mapping $y \mapsto W^{u}_{\delta,(\omega,x_y)}\cap B(x_{\ast},2r_2)$ is continuous on $\cS_{\om}$. 
	
	Define the subset $\cS$ of $\Om \times \cB$ by
	$$\cS:=\bigcup \{\{\omega\}\times S_{\omega}: \om\in \Om \ \text{satisfying} \ \exists y\in A_{\ast} \ \text{with} \ (\oy)\in\Gamma_{\ell_0} \}.$$
	\begin{lemma}\label{Lem:MMM}
		The set $\cS$ is random compact, and so it is measurable with $\mu(\cS)>0$.
	\end{lemma}
	\begin{proof}
		By Proposition \ref{Prop:random compact}, it suffices to show for every open subset $U$ of $\cB$ the set $\{\omega:S_{\om}\cap U\neq \varnothing\}$ is measurable. 
		By the construction of $\cS$, we have
		$$
		\{\omega:S_{\om}\cap U\neq \varnothing\}=\pi_{\Omega}
		\{(\ox)\in \Gamma_{\ell_0}\cap (\Omega\times A_{\ast}):W^{u}_{\delta,(\ox)}\cap B(x_{\ast},2r_2)\cap U\neq \varnothing \},
		$$ 
		where $\pi_{\Omega}$ is the canonical projection from $\Omega\times \cB$ onto $\Omega$. 
		Then, by Proposition \ref{Prop:projec P}, we only need to show that 
		$$\{(\ox)\in \Gamma_{\ell_0}\cap (\Omega\times A_{\ast}): W^{u}_{\delta,(\ox)}\cap B(x_{\ast},2r_2)\cap U\neq \varnothing  
		\}$$
		is a measurable subset of $\Omega\times \cB$.
		
		Let $m:= \dim E^{u}$, and let $\{v_k\}_{k>0}$ be a countable dense subset of $\R^{m}$. 
		Choose a measurable basis $\{e_i{(\ox)}\}_{i=1}^{m}$ spanning $E^{u}(\ox)$, for each $v_k=(v_k^i)\in \R^{m}$, let $v_k(\ox):= \sum_{i}v_k^i e_i{(\ox)}$. 
		Then, the mapping $(\ox)\mapsto v_k(\ox)$ is measurable for every $k$.
		By density, restricted to the subset $\Gamma_{\ell_0}\cap (\Omega\times A_{\ast})$ we have 
		\begin{align*}
			\{(\omega,x): W^{u}_{\delta,(\omega,x)}\cap B(x_{\ast},2r_2)\cap U\neq \varnothing\}
			=\bigcup_{k>0}(\{(\ox): |v_k(\ox)|\leq \delta \ell^{-3}_0\}\bigcap&\\
			\{(\ox): v_k(\ox)+g_{(\ox)}v_k(\ox)+x\in B(x_{\ast},2r_2)\cap U\}&).
		\end{align*}
		By Remark \ref{Rem:mofg}, the map $(\ox) \mapsto g_{(\omega,x)}v_k(\om,x)$ is measurable on the measurable subset 
		$\{(\ox): |v_k(\ox)|\leq \delta \ell^{-3}_0\}$ for every $k$. 
		Therefore, one has   
		$$\{(\omega,x)\in \Gamma_{\ell_0}\cap (\Omega\times A_{\ast}): W^{u}_{\delta,(\omega,x)}\cap B(x_{\ast},2r_2)\cap U\neq \varnothing
		\}$$
		is measurable.
		
		The construction of $\cS$ implies that $\Gamma_{\ell_0}\cap (\Omega\times A_{\ast})$ is a subset of $\cS$, so one has $\mu(\cS)>0$.
	\end{proof}

	For each $\omega$, let $\xi_{\omega}$ be a partition of $\cS_{\omega}$ defined as follows
	$$\xi_{\omega}:=\{W^{u}_{\delta,(\oy)}\cap B(x_{\ast},2r_2): y\in A_{\ast} \ \text{such that} \ (\omega,y)\in \Gamma_{\ell_0} \}.$$
	Each element $W \in \xi_{\omega}$ is called an \textit{unstable leave} of $\cS_{\omega}$.
	Since $S_{\omega}$ is contained in the finite union of unstable stacks of local unstable manifolds, it follows that $\eta_{\omega}$ is a measurable partition of $\cS_{\omega}$ (by Lemma \ref{GUStacks}).
	We refer to the set $\cS$ as a random unstable stack with the center $x_{\ast}$.

	\subsection{Properties of maps on unstable stacks}\label{SEC:4.4}
	Let the random unstable stack $\cS$ be defined above.
	
	\begin{lemma}\label{Lem:inverse}
		For all $n\in \N$, the map $f^{-n}_{\omega}$ is defined and continuous on $S_{\omega}$.
	\end{lemma}
	\begin{proof}
		For each $n\in \N$, the map $f^{-n}_{\omega}$ is defined on $S_{\omega}$ follows from Theorem \ref{Thm:Unstable}. 
		We now prove it is continuous. 
		Let $r_1(\om)$ be a small number satisfying Lemma \ref{Unstacks}, so that for each $x\in (\Gamma_{\ell_0})_{\om}$ the unstable stacks 
		$$\cS(\ox):=\bigcup_{y\in \bar{U}} \exp_{x}(\graph \Theta_{\omega}(y))$$
		is defined, where $\bar{U}:=A_{\ast}\cap U_{\omega}(x,r_1)$. 
		Since there is a finitely subset $\{x_{\omega,i}\}^{k(\om)}_{i=1}\subset (\Gamma_{\ell_0})_{\om} \cap A_{\ast}$ such that 
		$$\cS_{\omega}\subset \bigcup^{k(\om)}_{i=1} (\cS(\om, x_{\omega,i})\cap B(x_{\ast},2r_2)).$$
		It suffices to show $f^{-n}_{\omega}$ is continuous on $\cS(\ox)$ for each fixed $x\in (\Gamma_{\ell_0})_{\om} \cap A_{\ast}$.
		
		We claim that, by shrinking $r_1>0$, there exists $n_0\in \N$, for each $n\geq n_0$ there exists a compact subset $V_{n,x}\subset (\Gamma_{\ell_0e^{n\varepsilon_1}})_{\te^{-n}\omega} $ such that one can define a compact unstable stack $\cS_{n,x}$ of local unstable manifolds with $f^{n}_{\te^{-n}\omega}(\cS_{n,x})\supset \cS(\ox)$. 
		
		Assuming the claim holds. For $n\geq n_0$, by $f^{n}_{\te^{-n}\omega}$ is injective, one has $f^{-n}_{\omega}$ is continuous on $\cS(\ox)$. For $n<n_0$, only need to write $f^{n}_{\omega}=f^{n_0-n}_{\te^{-{n_0}}\omega} f^{-n_0}_{\omega}$.
		
		It remains to show that the claim holds. 
		Note that $f^{-n}_{\omega} \bar{U}\subset (\Gamma_{\ell_0e^{n\varepsilon_1}})_{\te^{-n}\omega}$,
		so $z\mapsto E^{u}(\te^{-n}\omega,z)$ and $z\mapsto E^{cs}(\te^{-n}\omega,z)$ are continuous on $f^{-n}_{\omega} \bar{U}$. 
		Shrinking $r_1>0$, such that for every $z\in f^{-n}_{\omega} \bar{U}$, one has $d_H(E^{cs}(\te^{-n}\omega,z), E^{cs}(\te^{-n}\omega,\fo^{-n}x))$ and
		$d_H(E^{u}(\te^{-n}\omega,z), E^{u}(\te^{-n}\omega,\fo^{-n}x))$ small enough,
		so that the map $g^{\fo^{-n}z}_{\Phi^{-n}(\ox)}$ as in Lemma \ref{Unstacks} is defined on $D_{n,x}$, 
		where $D_{n,x}:=B^{u}_{\Phi^{-n}(\ox)}(2\delta(e^{n\varepsilon_1}\ell_0)^{-3})$.
		Let $V_{n,x}:=f^{-n}_{\omega} \bar{U}$, and let 
		$$\cS_{n,x}=\bigcup_{z\in V_{n,x}} \exp_{f^{-n}_{\omega}x} \graph(g^{\fo^{-n}z}_{\Phi^{-n}(\ox)}|_{D_{n,x}}).$$
		Shrinking $r_1>0$ if necessary, one may assume that for every $z\in V_{n,x}$
		$$\exp_{z} \graph (g_{(\te^{-n}\om,z)}|_{\wB^{u}_{(\te^{-n}\om,z)}(\frac{1}{6}\delta(e^{n\varepsilon_1}\ell_0)^{-3})})
		  \subset \exp_{f^{-n}_{\omega}x} \graph(g^{\fo^{-n}z}_{\Phi^{-n}(\ox)}|_{D_{n,x}}).$$
		By Theorem \ref{Thm:Unstable} (iii) and Lemma \ref{Unstacks} (1), we can choose $n_0$ large enough (depends only on $\ell_0$), such that for every $n\geq n_0$
		\begin{align*}
			& \wf^{n}_{(\te^{-n}\omega,z)}\graph (g_{(\te^{-n}\om,z)}|_{\wB^{u}_{(\te^{-n}\om,z)}(\frac{1}{6}\delta(e^{n\varepsilon_1}\ell_0)^{-3})})\\
	\supset & \graph (g_{(\om,f^{n}_{\te^{-n}\om}(z))}|_{\wB^{u}_{(\om,f^{n}_{\te^{-n}\om}(z))}(4\delta \ell_{0}^{-2})})\\
	\supset & \graph \Te_{\om}(f^{n}_{\te^{-n}\om}(z)).
		\end{align*}
		Therefore, we have $f^{n}_{\te^{-n}\omega}\cS_{n,x}\subset \cS(\ox)$ for each $n\geq n_0$.
	\end{proof}
	
	Let $\xi_{\omega}$ be the measurable partition of $\cS_{\omega}$ into unstable leaves. 
	Let $\xi:=\{\{\omega\}\times W:W\in \xi_{\omega}\}$ be the partition of $\cS$, since $\Omega$ is Polish space one has $\xi$ is also measurable. 
	For each $(\ox)\in \cS$, denote by $\xi(\ox)$ the element of $\xi$ which contained $(\ox)$ and denote by $\xi(\ox)_{\om}$ the $\om$-section of the subset $\xi(\ox)$ of $\Om\times \cB$. 
	Then by construction one has $\xi_{\omega}(x)=\xi(\ox)_{\om}:=\{y:(\oy)\in \xi(\ox)\}$.
	
	Recall that $\nu_{(\ox)}$ is the induced volume on $W^{u}(\ox)$ for $(\ox)\in \Gamma$ which we defined in Section \ref{SEC:Distortion}.
	\begin{lemma}\label{Lem:MofV}
		For every measurable set $E \subset \Om\times \cB$, the map $ (\ox) \mapsto \nu_{(\ox)}((E \cap \xi(\ox))_{\om})$ is measurable on $\Gamma_{\ell_0}\cap (\Om\times A_{\ast})$.
	\end{lemma}
	\begin{proof}
		Let $m=\dim E^{u}$, choose a measurable basis $\{e_i{(\ox)}\}_{i=1}^{m}$ spanning $E^{u}(\ox)$. 
		For each $(\ox)\in \Gamma_{\ell_0}$ and each $v=(v_1,\cdots,v_m)\in \R^{m}$, let $v(\ox)=\sum_{i=1}^{m}v_ie_i{(\ox)}$, let 
		\begin{equation*}
			G_{(\ox)}(v)=
			\begin{cases}
				v(\ox)+g_{(\ox)}(v(\ox)) & \text{if} \  |v(\ox)|\leq\delta \ell_{0}^{-3};\\
				0 & \text{otherwise},
			\end{cases}
		\end{equation*}
		and let $DG_{(\ox)}(v):\R^m \rightarrow \cB$
		\begin{equation*}
			DG_{(\ox)}(v)u=
			\begin{cases}
				u(\ox)+Dg_{(\ox)}(v(\ox))u(\ox) & \text{if} \  |v(\ox)|\leq\delta \ell_{0}^{-3};\\
				0 & \text{otherwise}. 
			\end{cases}
		\end{equation*}
		
		Then, for each $(\ox)\in \Gamma_{\ell_0}\cap (\Omega\times A_{\ast})$ and for every measurable set $E \subset \Om\times \cB$, we have
		$$
		\nu_{(\ox)}((E \cap \xi(\ox))_{\om}):=\int_{E(\ox)}\det(DG_{(\ox)}(v))dv.
		$$ 
		Here $E(\ox):=\{v\in \R^{m}: G_{(\ox)}(v)+x\in E_{\om}\cap B(x_{\ast},2r_2)\}
		      \cap \{v\in \R^{m}:|v(\ox)|\leq \delta \ell^{-3}_0\}$ 
		and the ``$\det$" with respect to the Euclidean volume on $\R^{m}$ and the induced volume $m_{T_{G_{(\ox)}(v)} \wW^{u}_{\delta}(\ox)}$.	
		 
		By Remark \ref{Rem:mofg}, for fixed $v\in \R^{m}$ we have $(\ox) \mapsto G_{(\ox)}(v)$ is  measurable on $\Gamma_{\ell_0}$, and so $(\ox)\mapsto \chi_{E(\ox)}(v)$ is also a measurable function defined for $(\ox)\in \Gamma_{\ell_0}$.  
		
		Next, we prove that for each $v\in \R^{m}$ the function $(\ox)\mapsto \det(DG_{(\ox)}(v))$ is measurable. 
        By a similar discussion as in Lemma \ref{Lem:M3}, it suffices to show that for fixed $v\in \R^{m}$, the map $(\ox)\mapsto DG_{(\ox)}(v)$ is strongly measurable on $\Gamma_{\ell_0}$.
        
        If $|v(\ox)|\leq \delta \ell^{-3}_0$, for
        fixed $u\in\R^m$, let $\{t_n\}_{n\in \N}$ be a decreasing sequence of positive numbers with $t_n\rightarrow 0$. Then, we have $\Lnorm{(v+t_nu)(\ox))}{(\ox)}\leq \delta \ell^{-2}_0$ for $n$ large enough. Hence, we can write
        $$DG_{(\ox)}(v)u=u(\ox)+\lim_{n \rightarrow \infty}\dfrac{g_{(\ox)}((v+t_nu)(\ox))-g_{(\ox)}(v(\ox))}{t_n}.$$
        By Remark \ref{Rem:mofg}, we have that $(\ox)\mapsto DG_{(\ox)}(v)u$ is measurable 
        on $\{(\ox)\in \Gamma_{\ell_0}:|v(\ox)|\leq \delta \ell^{-3}_0\}$. 
        Thus, by the construction, we conclude that $(\ox)\mapsto DG_{(\ox)}(v)$ is strongly measurable on $\Gamma_{\ell_0}$.
        
        Consequently, for every $v\in \R^m$ we get that $(\ox) \mapsto \det(DG_{(\ox)}(v))\cdot \chi_{E(\ox)}(v)$ is measurable on $\Gamma_{\ell_0}\cap (\Om\times A_{\ast})$. This suffices to show that$ (\ox) \mapsto \nu_{(\ox)}((E \cap \xi(\ox))_{\om})$ is measurable on $\Gamma_{\ell_0}\cap (\Om\times A_{\ast})$.
	\end{proof}

	In the remainder of this section, we assume that $E^{c}(\ox)=0$ for every $(\ox)\in \Gamma$.
	Recall that $E^{u}(\omega,y)$ also denote the tangent space $T_y W^{u}(\ox)$ if $y\in W^{u}(\ox)$ for some $(\ox)\in \Gamma$.
	
	\begin{lemma}\label{Lem:CofU2}
		Assume $E^{c}=\{0\}$. For each $q\in \N$, the map $x\mapsto E^{u}(\omega,x)$ is continuous on $f^{-q}_{\omega} S_{\omega}$
	\end{lemma}
	\begin{proof}
	    We give the proof for $q=0$, and the case for $q\geq 1$ can be proven similarly.
		
		Let $(\ox)\in \cS$ be fixed, the main difference with the argument in Proposition \ref{Prop:CCS} is that no intrinsically defined the stable subspace $E^{cs}(\ox)$. 
		As in \cite[proof of Lemma 7.6]{Young17}, we use the backward graph transform to overcome it. 
		
		For $(\ox)\in \cS$, there exists a point $(\om,x_0)\in \Gamma_{\ell_0}$ with the property that $x\in W^{u}_{\delta,(\om,x_0)}$.
		We write $\wf_{-n}:=\wf_{\Phi^{-n}(\om,x_0)}$, $\Lnorm{\cdot}{-n}:=\Lnorm{\cdot}{\Phi^{-n}(\om,x_0)}$ and
		${x}'_{-n}:=\fo^{-n}x-\fo^{-n}x_{0}$ for each $n\geq 0$. 
		
		Let $F_0=E^{s}(\om,x_0)$, then $\cB=E^{u}(\ox)\oplus F_0$.
		For $\delta$ small enough, by Lemma \ref{Lem:NlinP}, we may assume that $D\wf_{-n}({x}'_{-n})$ is sufficiently close to $D\wf_{-n}(0)$ so that the backward graph transform (Lemma \ref{Lem:GtranS}, also see \cite[Proposition 9]{Lian20}) can be applied to the linear map $D\wf_{-n}({x}'_{-n})$ for every $n\geq 0$.
		Then, by Lemma \ref{Lem:GtranS}, we get a sequence of subspace $\{F_n\}_{n\in \N}\subset \cG(\cB)$, such that $D\wf_{-n}({x}'_{-n})(F_n)\subset F_{n-1}$ and for each $v_{-n}\in F_n$, one has
		$$\Lnorm{D\wf_{-n}({x}'_{-n})v_{-n}}{-n+1}\leq (e^{-\lambda}+\delta) \Lnorm{v_{-n}}{-n}.$$
		Since $F_n$ is a graph of some linear map from $E^{s}(\Phi^{-n}(\om,x_0))$ to $E^{u}(\Phi^{-n}(\om,x_0))$, one has $\cB= E^{u}(\Phi^{-n}(\ox))\oplus F_{n}$.
		
		For every $y_k\in \cS_{\om}$ with $y_k\rightarrow x$ as $k\rightarrow \infty$. Let $v^k\in E^{u}(\om,y_k)$ and decompose $v^k=v^{k,u}+v^{k,s}$ with respect to the splitting $\cB=E^{u}(\ox)\oplus F_0$. 
		Using the same estimation in Proposition \ref{Prop:CCS}, one has $v^{k,s}\rightarrow 0$ as $k\rightarrow \infty$.
		Hence, we have $G(E^{u}(\om,y_k),E^{u}(\ox))\rightarrow 0$ as $k\rightarrow \infty$. 
		By Lemma \ref{Lem:gap}, this is enough to show that $x\mapsto E^{u}(\omega,x)$ is continuous on $\cS_{\om}$.   
	\end{proof}
	For $(\ox)\in \cS$, let $J^{u}(\ox)=\det(Df_{\omega}(x)|E^u(\ox))$. 
	By Proposition \ref{Prop:dis} (b), the following function is well defined on $\xi_{\omega}(x)$ and Lipschitz-continuous:
	$$z\mapsto \Delta^{\omega}(x,z):= \prod_{n=1}^{\infty} \dfrac{J^{u}(\Phi^{-n}(\ox))}{J^{u}(\Phi^{-n}(\omega,z))}, \quad z\in \xi_{\omega}(x).$$
	
	Define the function $q(\omega,z):\cS \rightarrow \R$ by
	$$q(\omega,z):=\dfrac{\Delta^{\omega}(x,z)}{\int_{\xi_{\omega}(x)} \Delta^{\omega}(x,y) d\nu_{(\ox)}(y)}, \quad z\in \xi_{\omega}(x).$$
	Since for every ${x}'\in \cB$ satisfying that $\xi_{\omega}(x)=\xi_{\omega}({x}')$, we have
	$$\Delta^{\omega}({x}',z)=\Delta^{\omega}(x,z) \cdot \Delta^{\omega}({x}',x).$$
	Hence, $q$ is not dependent on the choice of $x\in \xi_{\omega}(z)$ and so is well defined as a function of $(\om,z)$.
	Moreover, for fixed $z$, the function $\om\mapsto q(\om,z)$ is measurable on $\{\om:(\om,z)\in\cS\}$.
	Indeed, by Lemma \ref{Lem:MMM} for each $n>0$, the function $(\oy)\mapsto \Delta^{\om}_{n}(y,z)$ is measurable on $\Om\times \cB$, where 
	$$\Delta^{\om}_{n}(z,y)=\prod_{k=1}^{n} \dfrac{J^{u}(\Phi^{-k}(\om,z))}{J^{u}(\Phi^{-k}(\oy))}.$$
	By Lemma \ref{Lem:MofV} and standard arguments, the map $\om \mapsto \int_{\xi_{\omega}(z)} \Delta^{\om}_{n}(z,y) d\nu_{(\om,z)}(y)$ is measurable. 
	Recall that $\Delta^{\om}_{n}(z,\cdot)$ converge uniformly to $\Delta^{\om}(z,\cdot)$ on $\xi_{\omega}(z)$ (see Proposition \ref{Prop:dis}).
	Therefore, $\om \mapsto \int_{\xi_{\omega}(z)} \Delta^{\om}(z,y) d\nu_{(\om,z)}(y)$ is also measurable.
	This is enough to show that $\om \mapsto q(\om,z)$ is measurable.
	\begin{lemma}\label{Lem:CofFq}
		Assume $E^{c}=\{0\}$. For fixed $\omega$, the function $z\mapsto q(\omega,z)$ is continuous on $\cS_{\omega}$.
	\end{lemma}
	\begin{proof}
		We follow the proof of Lemma 7.9 in \cite{Young17}.
		
		Let $\omega$ be fixed. 
		For each $n$, by Lemma \ref{Lem:inverse} we have $\fo^{-n}|_{\cS_{\omega}}$ is continuous. 
		By Lemma \ref{Lem:CofU2}, we have 
		$x\mapsto E^{u}(\ox)$ is continuous on $\fo^{-n}{\cS_{\omega}}$. 
		Then,  by Proposition \ref{Prop:CofDet}, we have $x\mapsto \det(Df_{\te^{-n}\omega}(\fo^{-n}x)|E^u(\Phi^{-n}\ox))$ is continuous on $\cS_{\om}$.
		
		Let $r_1(\om)$ be a small number satisfying Lemma \ref{Unstacks}. 
		For each $(\ox)\in \Gamma_{\ell_0}\cap (\Om\times A_{\ast})$,
		define the unstable stack $\cS(\ox)$ as follows
		$$
		\cS(\ox):=\bigcup_{y\in \bar{U}} \exp_{x}(\graph \Theta_{\omega}(y))
		$$
		where $\bar{U}=A_{\ast}\cap U_{\omega}(x,r_1)$. 
		Since $\cS_{\om}$ is contained in the finite union of unstable stacks. 
		It suffices to show that $q(\om,\cdot)$ is continuous on $\cS(\ox)\cap B(x_{\ast},2r_2)$.
		
		By Lemma \ref{GUStacks}, there exists a homeomorphism 
		$\Psi_{\ox}: \Sigma_{(\ox)}\times B^{u}_{(\ox)}(2\delta \ell^{-3}_0)\rightarrow \cS(\ox)$, $\Psi_{\ox}(\sigma,u)=\exp_{x}(u+g_{\sigma}u)$,
		where $\Sigma_{(\ox)}$ and $g_{\sigma}$ are as defined in Lemma \ref{GUStacks}. 
		For every $z\in \cS(\ox)\cap B(x_{\ast},2r_2)$, denote $(\sigma_z,u_z):=\Psi^{-1}_{\ox}(z)$ and let $\sigma(z):=\Psi_{\ox}(\sigma_z,0)$.
		Define
		$$z\mapsto \bar{\Delta}_n(z):=\Delta^{\om}_{n}(\sigma(z),z),$$
		where $\Delta^{\om}_{n}(\sigma(z),z)$ as in $(\ref{eq:Delta n})$.
		Since $\sigma(z)$ depends continuously on $z$, one has $z\mapsto \bar{\Delta}_n(z)$ is continuous. 
		By Proposition \ref{Prop:dis}, 
		we have $\bar{\Delta}_n$ converges uniformly to $\bar{\Delta}$ which is also continuous on $\cS(\ox)$.
		
		It remains to show that $z\mapsto \int \bar{\Delta} d\nu_{(\om,z)}$ is continuous on $\cS(\ox)\cap B(x_{\ast},2r_2)$. 
		Denote by $E_z:=\pi^{u}_{(\ox)}(\exp^{-1}_{x}\xi_{\omega}(z))$.
		By the construction of $\cS_{\om}$, one has $z\mapsto \xi_{\omega}(z)$ is continuous.
		Thus, $z\mapsto E_z$ is also continuous with respect to the Hausdorff topology.
		
		By Lemma \ref{GUStacks} and the change of variables formula, we have
		$$
		\int_{\xi_{\omega}(\sigma(z))} \bar{\Delta}(z) dv_{(\om,\sigma(z))}=
		\int_{B^{u}_{(\ox)}(2\delta \ell^{-3}_0)} 
		\bar{\Delta}(z) \det(Id+Dg_{\sigma_z}(u))\chi_{E_z}(u)dm(u)$$	
		where $m$ is the induced volume on $E^{u}(\ox)$ and the ``$\det$" is with respect to $m$ and $m_{E^u(\om,z)}$.
	    By Lemma \ref{Lem:CofU2} we have that 
	    $(\sigma,u) \mapsto E^u(\om,\Psi_{\ox}(\sigma,u))$ 
	    is continuous on 
	    $\Sigma_{\ox} \times B^{u}_{\ox}(2\delta \ell^{-3}_0)$. 
	    So, by Lemma \ref{Lem:CdhA} the map $(\sigma,u) \mapsto Id+Dg_{\sigma}(u)$ is continuous.
	    From this and the properties of the ``det"-function (see Proposition \ref{Prop:CofDet}), 
	    we deduce that the mapping $z \mapsto \det(Id+g_{\sigma_z}(u))$ is continuous for each $u$.
	    Therefore, we have $z\mapsto \int \bar{\Delta} d\nu_{(\om,z)}$ is continuous. This completes the proof.
	\end{proof}
	\section{Proof of main Theorem}\label{SEC:5}
	In this section, we use the same method as in \cite{Young85} (also in \cite{Young17} and \cite{Li13}) to give a detailed proof of Theorem \ref{Thm:A} and Theorem \ref{Thm:B} when $\mu$ is ergodic.
	
	We recall some notations about partitions at the beginning. 
	Let $T: X\rightarrow X$ be a measure preserving transformation of a probability space $(X,\mu)$.
	For a measurable partition $\eta$ on $X$, we let $\eta(x)$ denote the atom of $\eta$ containing $x$ for each $x\in X$.
	For two measurable partitions $\eta_1$ and $\eta_2$ , we say $\eta_2$ is a refinement of $\eta_1$ if $\eta_1(x)\supset \eta_2(x)$ for $\mu$-almost every $x$, and written by $\eta_1\leq \eta_2$. 
	Let $ \eta_1 \vee \eta_2=\{A\cap B:A\in \eta_1, B\in\eta_2\}$ and let $T^{-1} \eta_1=\{T^{-1}(A):A\in\eta_1\}$. We say $\eta$ is $T$-decreasing if $\eta \leq T^{-1} \eta$.
	
	We fixed $\ell_0$ to be large enough such that $\mu(\Gamma_{\ell_0})>0$, and let the random stack $\cS$ be as defined in Section \ref{SEC:Random US}. 
	
	\subsection{Partitions subordinate to the unstable foliation}
	Let $\xi_{\omega}$ be the measurable partition of $\cS_{\omega}$ into unstable leaves, and let $\xi:=\{\{\omega\}\times W:W \in \xi_{\omega}\}$ be the measurable partition of $\cS$. 
	By definition, for each $(\ox)\in \cS$, one has $\xi(\ox)\subset \{\om\}\times \cB$.
	We denote by $\xi(\ox)_{\om}$ the $\om$-section of the subset $\xi(\ox)$ of $\Om\times \cB$.
	
	Denote by $\{\mu_{\xi(\ox)}\}_{(\ox)\in \cS}$ the canonical disintegration of $\mu|_{\cS}$, 
	which satisfied that
	$$\mu|_{\cS}=\int_{\cS} \mu_{\xi(\ox)} d\mu(\ox).$$
	The family of measures $\{\mu_{\xi(\ox)}\}$ on $W^{u}(\ox)$ (one can identify $\{\omega\}\times W^{u}(\ox)$ with $W^{u}(\ox)$) exists by Rokhlin's disintegration theorem, and any two families like that are equal on a full $\mu|_{\cS}$-measure subset of $\cS$. 
	Recall that $\nu_{(\ox)}$ is the induced volume on $W^{u}(\ox)$ for $(\ox)\in \Gamma$. 
 
	We give a formal definition of SRB measures for random dynamical systems on Banach spaces as follows:
	\begin{definition}\label{DEF:SRBm}
		We say $\mu$ is a random SRB measure if (i) it has a positive exponent for $\mu$-almost every $(\omega,x)$, and (ii) for every random stack $\cS$ with $\mu(\cS)>0$, $\mu_{\xi(\ox)}$ is absolutely continuous with respect to $\nu_{(\ox)}$ for $\mu$-almost every $(\ox)$ in $\cS$.
	\end{definition}
	
	\begin{definition}\label{DEF:PSubord}
		We say that a measurable partition $\eta$ is \textit{subordinate to the unstable foliation} $W^{u}$ if for $\mu$-almost every $(\ox)$ one has
		\begin{enumerate}
			\item[(i)]  $\eta(\ox)\subset \{\omega\}\times W^{u}(\ox)$,
			\item[(ii)] $\eta(\ox)_{\omega}$ contains a neighborhood of $x$ in $W^{u}(\ox)$, and
			\item[(iii)]$\eta(\ox)_{\omega} \subset f^{N}_{\te^{-n}\omega} W^{u}_{\loc}(\Phi^{-n}(\ox))$ for some $N\in \N$ depending on $(\ox)$.
	    \end{enumerate}	 
	\end{definition}
	
	\begin{lemma}
		The following are equivalent:
		\begin{enumerate}
			\item[(1)] $\mu$ is a random SRB measure;
			\item[(2)] there exists a partition $\eta$ subordinate to $ W^{u}$ with $\mu_{\eta(\ox)}\ll \nu_{(\ox)}$ for $\mu$-almost every $(\ox)$;
			\item[(3)] for every partition $\eta$ subordinate to $ W^{u}$, one has $\mu_{\eta(\ox)}\ll \nu_{(\ox)}$ for $\mu$-almost every $(\ox)$.
		\end{enumerate}
	\end{lemma}
	\begin{proof}
		We first prove (2) implies (3): Let $\eta_1$ be the partition satisfies (2), for each $\eta_2$ subordinate to 
		$W^{u}(\ox)$, we denote by $\eta=\eta_1 \vee \eta_2$ and let $\{\mu_{W}\}_{W\in \eta}$ be the canonical disintegration of $\mu$ with respect to $\eta$. Note that the condition (2) implies $\mu_{\eta(\ox)} \ll \nu_{(\ox)}$ for $\mu$-almost every $(\ox)$. Therefore, by the uniqueness of the canonical disintegration, for $\mu$-almost every $(\ox)$ one has 
  $$\mu_{\eta_2(\ox)}=\int_{\eta_2(\ox)} \mu_{\eta(\oy)} d\mu_{\eta_2(\ox)}(y).$$
  Hence, we have $\mu_{\eta_2(\ox)}\ll \nu_{(\ox)}$ for $\mu$-almost every $(\ox)$.
		
		The statement (1) implies (2) and, (3) implies (1) follow immediately from the proposition below.
	\end{proof}
	
	\begin{proposition}\label{Prop:Exist P}
		Assuming that $(\Phi,\mu)$ has a positive Lyapunov exponent, then there is a random stack $\cS$ with positive $\mu$-measure, and a measurable partition $ \eta$ on $\widetilde{\cS}:=\cup_{n\geq0}\Phi^n \cS$ such that 
		\begin{enumerate}
			\item[(a)] $\eta$ subordinate to $ W^{u}$,
			\item[(b)] $\eta$ is $\Phi$-decreasing,
			\item[(c)] for any measurable set $E \subset \Omega\times \cB$, the function $(\ox) \mapsto \nu_{(\ox)}(E_{\omega} \cap \eta(\ox)_{\omega})$ is finite-valued and measurable on a full $\mu$-measure set, and
			\item[(d)] $\bigvee_{n\geq 0}\Phi^{-n}\eta$ is a partition of $\widetilde{\cS}$ into points modulo  $\mu$-measure $0$.
		\end{enumerate}
	\end{proposition}
	\begin{proof}
		Using the notation in Section \ref{SEC:Random US}, we first construct a random stack.
		Let $r\in (1/2,1)$ and let
		$$\cS^{r}_{\omega}:=\bigcup \{W^{u}_{\delta,(\oy)}\cap \bar{B}(x_{\ast},2rr_2): y\in A_{\ast} \ \text{such that} \ (\omega,y)\in \Gamma_{\ell_0} \},$$
		and define $\cS^{r}:=\bigcup \{\{\omega\}\times S^{r}_{\omega}:   \om\in \Om \ \text{satisfying} \ \exists y\in A_{\ast} \ \text{with} \ (\oy)\in\Gamma_{\ell_0}\}$. 
		Let $\widetilde{\cS}:=\widetilde{\cS}^r=\bigcup_{n\geq0}\Phi^n \cS^r$. Since $\mu$ is $\Phi$-ergodic, it follows that $\mu(\widetilde{\cS})=1$.
		
		Let $\xi:=\{\{\omega\}\times W_{\omega}:W_{\omega}\in \xi_{\omega}\}$ be a measurable partition of $\cS^r$, where $\xi_{\omega}$ is the measurable partition of $\cS^r_{\omega}$ into unstable leaves. 
		For each $n\geq 0$, let $\xi^n=\{\Phi^n(W):W\in \xi\} \cup \{\widetilde{\cS}-\Phi^n(\cS^r)\}$ be a measurable partition of $\widetilde{\cS}$, and let $\eta=\bigvee_{n=0}^{\infty}\xi^n$. 
		Note that $\eta$ is a refinement of the partition $\eta_0:=\{\{\omega\}\times \cB\}_{\omega\in \Omega}$. 
		For each $(\omega,x)\in \widetilde{\cS}$, a point $(\omega,y)\in \eta(\ox)$ if and only if
		\begin{equation}\label{eq:eta}
			\left\{
			\begin{alignedat}{2}
				&f^{-n}_{\omega} y \in \xi_{\te^{-n}\omega}(f^{-n}_{\omega} x)
				\quad &\text{if}& \ \Phi^{-n}(\ox) \in \cS^r; \\
				&\Phi^{-n}(\oy) \notin \cS^r   &\text{if}& \ \Phi^{-n}(\ox) \notin \cS^r.
			\end{alignedat}
			\right.	
		\end{equation}
		for any $n\geq 0$. Therefore, one has that $\eta$ is $\Phi$-decreasing. This proves item (b).
		
		For each $(\omega,x)\in \widetilde{\cS}$, there exists $N(\ox)\in \N$ such that $\Phi^{-N}(\ox) \in \cS^r$. 
		Then, by (\ref{eq:eta}) we have $\eta(\ox)_{\omega}\subset f^{N}_{\te^{-n}(\omega)} W^{u}_{\loc}(\Phi^{-N}(\ox))$. 
		To show $\eta$ subordinate to $W^{u}$, we only requirement that $\eta(\ox)_{\omega}$ contains a neighborhood of $x$ in $W^{u}_{\loc}(\ox)$ for $\mu$-almost every $(\ox)$ left.
		 
		We now choose $r$ carefully to guarantee that. 
		Let 
		$$ \partial \cS^{r}_{\omega}:=\bigcup \{W^{u}_{\delta,(\oy)}\cap 
		\{z:|z-x_{\ast}|=2rr_2\}: y\in A_{\ast} \ \text{such that} \ (\omega,y)\in \Gamma_{\ell_0} \},$$
		and let 
		$$\partial \cS^r:=\bigcup \{\{\omega\}\times \partial \cS^{r}_{\omega}: \om\in \Om \ \text{satisfying} \ \exists y\in A_{\ast} \ \text{with} \ (\oy)\in\Gamma_{\ell_0} \}.$$
		For fixed $(\ox)\in \widetilde{\cS}\cap \Gamma$.
		Since $\{\omega\}\times \partial (\eta(\ox)_{\omega})\subset \bigcup_{n\geq0} \Phi^{n} (\partial (\cS^r))$, we have that for $0<\varepsilon<{\delta}'_{1}$ (${\delta}'_{1}$ be as in Theorem \ref{Thm:Unstable}), if $f^{-n}_{\omega}W^u_{\varepsilon}(\ox)\cap (\partial (\cS^r))_{\te^{-n}\omega}=\varnothing$ then $W^u_{\varepsilon}(\ox)\subset \eta(\ox)_{\omega}$. 
		To guarantee that, we needed that for every $y\in W^u_{\varepsilon}(\ox)$ and every $n\in \N$,
		$$|f^{-n}_{\omega}y-f^{-n}_{\omega}x|<
		|(2rr_2-|f^{-n}_{\omega}x-x_{\ast}|)|.$$
		For every $y\in W^u_{\varepsilon}(\ox)$, by Theorem \ref{Thm:Unstable} (iii) we have
		$$3|f^{-n}_{\omega}y-f^{-n}_{\omega}x|\leq \Lnorm{f^{-n}_{\omega}y-f^{-n}_{\omega}x}{\Phi^{-n}(\ox)}
		                                      \leq \Lnorm{x-y}{(\ox)}e^{\frac{-n\lambda}{2}}.$$
		Therefore, we need to take $\varepsilon(\ox)$ small enough such that
		$$\varepsilon< 
		\inf_{n\in \N}\{{\delta}'_1, 3\ell(\ox)e^{\frac{n\lambda}{2}} |(2rr_2-|f^{-n}_{\omega}x-x_{\ast}|)|\}.$$ 
		Take $r\in (1/2,1)$, then the following subset
		\begin{equation}\label{eq:51}
			\{r:\sum_{n=0}^{\infty}\mu(\{(\ox): |(|x-x_{\ast}|-2rr_2)|<e^{\frac{-n\lambda}{2}}\})<\infty \}
		\end{equation}
		have full Lebesgue measure on the interval $(1/2,1)$ by standard estimations (see \cite[Proposition 3.2]{Ledrappier82} for details). Since there are most countable choice of $r$ such that $\mu(\partial \cS^r)>0$, one can choose $r\in (1/2,1)$ such that $r$ belongs to the set (\ref{eq:51}), and $\mu(\partial \cS^r)=0$.
		Moreover, by $\mu$ is $\Phi$ invariant, we also have
		$$\sum_{n=0}^{\infty}\mu(\{(\ox): |(|f^{-n}_{\omega}x-x_{\ast}|-2rr_2)|<e^{\frac{-n\lambda}{2}}\})<\infty.$$
		By a Borel-Cantelli type argument, for $\mu$-almost every $(\ox)$ we have 
		$$ |(|f^{-n}_{\omega}x-x_{\ast}|-2rr_2)|\geq e^{\frac{-n\lambda}{2}}$$
		except for a finite number of integers $n\in \N$.
		So, one can choose $\varepsilon(\ox)>0$ such that $W^u_{\varepsilon}(\ox)\subset \eta(\ox)_{\om}$. 
		Item (a) be proved.

        By Lemma \ref{Lem:MofV}, the function
        $$
		(\ox) \mapsto \nu_{(\ox)}((E \cap  \xi(\ox))_{\omega})$$
        is measurable on $ \Gamma_{\ell_0}\cap(\Om\times A_{\ast})$.
		It follows clearly that 
		$$
		(\ox) \mapsto \nu_{(\ox)}(E_{\omega}\cap ((\bigvee_{i=0}^{n-1}\xi^{i})(\ox))_{\omega}), \ n\in \N$$
		is measurable on $\bigcup^{n-1}_{i=0 }\Phi^{i}(\Gamma_{\ell_0}\cap(\Om\times A_{\ast}))$, and decreasing for $n$. Therefore, we have 
		$$
		(\ox) \mapsto \nu_{(\ox)}(E_{\om} \cap \eta(\ox)_{\omega})=\lim_{n\rightarrow \infty}\nu_{(\ox)}((E\cap (\bigvee_{i=0}^{n-1}\xi^{i})(\ox))_{\omega})
		$$
		is measurable on the full $\mu$-measure set $\bigcup^{\infty}_{n=0}\Phi^{i}(\Gamma_{\ell_0}\cap(\Om\times A_{\ast}))$. Item (c) be proved.

        Note that for $\mu$-almost every $(\ox)$, there is a subsequence $\{k_n(\ox)\}_{n\in \N}$ of $\N$ such that $\Phi^{-k_n}(\ox)\in \cS^r$.
		Item (d) follows from (\ref{eq:eta}) and Theorem \ref{Thm:Unstable} directly.
	\end{proof}
	\subsection{SRB measures implies entropy formula}
	\begin{proof}[proof of Theorem \ref{Thm:A} (ergodic case)]
		Since Ruelle's inequality \cite{Ruelle82} (see also \cite{Li12} for the random case on Banach spaces) tells us that 
		$$h({\mu}'|\bP)\leq \int \sum_{i} m_i\lambda^{+}_{i}d{\mu}'$$
		for any $\Phi$-invariant probability measure ${\mu}'$ projection to $\bP$. 
		Then, it suffices to show that if $\mu$ is a random SRB measure the reverse inequality holds. 
		We assume that $\mu$ is ergodic here.
		Let $\eta$ be constructed in Proposition \ref{Prop:Exist P}. 
		Recall that $\eta_0:=\{\{\omega\}\times \cB\}_{\omega\in \Omega}$ and the definition of $h(\mu|\bP)$ in Section \ref{SEC:2.1}. 
		By $\eta$ is $\Phi$-decreasing and $\eta_0\leq \eta$, we have 
		\begin{align*}
	  h(\mu|\bP)&=h(\mu|\bP,(\Phi|_A)^{-1})\\
                  &\geq h(\mu|\bP,(\Phi|_A)^{-1},\eta)
                  =H_{\mu}(\eta|\bigvee^{\infty}_{i=1}\Phi^{i} \eta\vee \eta_{0})\\
                  &=H_{\mu}(\eta|\Phi\eta)=H_{\mu}(\Phi^{-1}\eta|\eta)=\frac{1}{n}H_{\mu}(\Phi^{-n}\eta|\eta).
		\end{align*}
		It suffices to show that 
		$$H_{\mu}(\Phi^{-1}\eta|\eta)=\int \log J^{u}(\ox)d\mu(\ox)=\sum m_i\lambda_{i}^{+},$$
		where $J^{u}(\ox)=\det(Df_{\omega}(x)|E^{u}(\ox))$.
		Note that $\log J^{u}\in L^1(\mu)$, the second equality follows immediately from the Birkhoff ergodic theorem and Proposition \ref{Prop:DetL}.
		
		By Proposition \ref{Prop:Exist P} (c), we let $\nu$ be the $\sigma$-finite measure defined by
		$$\nu(K)=\int v_{(\ox)}((\eta(\ox) \cap K)_{\omega}) d\mu(\ox), \ K \subset \Om\times \cB \ \text{is measurable}.$$
		By assumption, we have $\mu \ll \nu$. Let $\rho=\frac{d\mu}{d\nu}$, then for $\mu$-almost every $(\ox)$ one has
		$$\rho(\omega,y)=\frac{d\mu_{\eta(\ox)}}{d\nu_{(\ox)}}(y) \quad \text{for} \ \nu_{(\ox)}\text{-almost every}\ y\in \eta(\ox)_{\omega}.$$
		Recall that the information function $I(\Phi^{-1}\eta|\eta)$ is
		$$I(\Phi^{-1}\eta|\eta)(\ox)=-\log\mu_{\eta(\ox)}((\Phi^{-1}\eta)(\ox)_{\omega}).$$ 
		  For every measurable set $E\subset \cB$.
		On the one hand, since $\eta$ is $\Phi$-decreasing, one has
		\begin{align*}
			\mu_{(\Phi^{-1}\eta)(\ox)}(E)&=\dfrac{\mu_{\eta(\ox)}((\Phi^{-1}\eta)(\ox)_{\omega}\cap E) }{\mu_{\eta(\ox)}((\Phi^{-1}\eta)(\ox)_{\om})}\\
			&=e^{I(\Phi^{-1}\eta|\eta)(\ox)}\int_{((\Phi^{-1}\eta)(\ox))_{\omega}\cap E} \rho(\oy) d\nu_{(\ox)}(y).
		\end{align*}
		On the other hand, by $\mu$ is $\Phi$-invariant, we have for $\mu$-almost every $(\ox)$
		\begin{align*}
			\mu_{(\Phi^{-1}\eta)(\ox)}(E)&=\mu_{\eta(\Phi (\ox))}(f_{\omega}E)\\
			&=\int_{f_{\omega}(((\Phi^{-1}\eta)(\ox))_{\omega}\cap E)} \rho(\te(\om),z) d\nu_{\Phi(\ox)}(z)\\
			&=\int_{f_{\omega}(((\Phi^{-1}\eta)(\ox))_{\omega}\cap E)} \rho(\te(\om),z) J^{u}(\om,f^{-1}_{\te\omega}(z)) d((f_{\omega})_{\ast}\nu_{(\ox)})(z)\\
			&=\int_{((\Phi^{-1}\eta)(\ox))_{\omega}\cap E} \rho(\Phi(\oy)) J^{u}(\oy) d\nu_{(\ox)}(y),
		\end{align*}
		where we use the change of variables formula \eqref{eq:Var} in the third equality 
		(recall $f^{-1}_{\te\omega}(y)$ is the unique point such that $f_{\om}(f^{-1}_{\te\omega}(y))=y$).
		Consequently, by $E$ is arbitrary and $ I(\Phi^{-1}\eta|\eta)(\oy)= I(\Phi^{-1}\eta|\eta)(\ox)$ for any 
		$y\in ((\Phi^{-1}\eta)(\ox))_{\omega}$, we get that for $\mu$-almost every $(\ox)$
		$$I(\Phi^{-1}\eta|\eta)(\oy)=-\log L(\oy) \quad \text{for} \ \nu_{(\ox)}\text{-almost every}\ y\in ((\Phi^{-1}\eta)(\ox))_{\omega}.$$
		where $L(\oy)=\rho(\oy)(J^{u}(\oy))^{-1} (\rho(\Phi(\oy)))^{-1}$. Therefore, for $\mu$-almost every $(\ox)$ one has
		$$I(\Phi^{-1}\eta|\eta)(\ox)=\log J^{u}(\ox)+\log \dfrac{\rho(\Phi(\ox))}{\rho(\ox)}.$$
		Since $\log J^{u}\in L^1(\mu)$ and $I(\Phi^{-1}\eta|\eta)(\ox)\geq 0$, one has $\log^{-} \frac{\rho\circ \Phi}{\rho}\in L^1(\mu)$. Therefore, by Lemma \ref{Lem:Tem} one has $\int \log \frac{\rho\circ \Phi}{\rho}d\mu=0$. 
		So, we have
		$$H_{\mu}(\Phi^{-1}\eta|\eta)=
		\int \log J^{u}(\ox)d\mu(\ox).$$
		This completes the proof.
	\end{proof}
	\subsection{Entropy formula implies SRB measures}
	We assume that $(\Phi,\mu)$ satisfies (H1)-(H5) and $\mu$ is ergodic.  
	Let $\eta,\xi, \cS$ and $\widetilde{\cS}$ be as in Proposition \ref{Prop:Exist P}.
	We first prove the following:
	\begin{lemma}\label{Lem:last}
		Let $(\Phi,\mu)$ with a positive Lyapunov exponent for $\mu$-almost everywhere, and let $\eta$ be as in Proposition \ref{Prop:Exist P}. Assume
		$$H_{\mu}(\Phi^{-1}\eta|\eta)=\int \log \det(Df_{\om}(x)|E^{u}(\ox)) d\mu,$$
		then we have $\mu$ is a random SRB measure.
	\end{lemma}
	
	\begin{proof}
		Using the notations in Section \ref{SEC:4.4}.
		For $(\om,x)\in \cS$. 
		Note that we also denote by $E^{u}(\oy)$ the tangent space of $W^{u}(\ox)$ at the point $y\in W^{u}(\ox)$. Let
		$$J^{u}(\ox)=\det (Df_{\om}(x)|E^{u}(\ox)).$$ 
		
		Recall the Lipschitz-continuous function $z \mapsto \Delta^{\omega}(x,z)$ on $\xi(\ox)_{\om}$ which defined by
		$$z\mapsto \Delta^{\omega}(x,z):= \prod_{n=1}^{\infty} \dfrac{J^{u}(\Phi^{-n}(\ox))}{J^{u}(\Phi^{-n}(\omega,z))},$$
		and the function $q(\omega,z)$ on $\cS$ defined by
		$$q(\omega,z):=\dfrac{\Delta^{\omega}(x,z)}{\int_{\xi_{\omega}(x)} \Delta^{\omega}(x,y) d\nu_{(\ox)}(y)}, \quad z\in \xi(\ox)_{\om}.$$
		By Lemma \ref{Lem:MPolish} and Lemma \ref{Lem:CofFq}, one can deduce that $q$ is measurable on $\cS$.
	 
		Since $\eta$ is subordinate to $W^{u}$,
		for each $(\omega,x) \in \widetilde{\cS}$ there exists $n\geq 0$ such that $f^{-n}_{\omega}(\eta(\ox))$ contained in a leaf of $\cS_{\te^{-n}\om}$. 
		By assumption (H1) and Theorem \ref{Thm:Unstable}, $f^{n}_{\te^{-n}\omega}$ is a $C^2$ embedding on each unstable leaf. 
		Then, we have that for $\mu$-almost every $(\ox)$, the map $z\mapsto \Delta^{\omega}(x,z)$ is Lipschitz-continuous and bounded on $\eta(\ox)_{\omega}$. 
		By $\eta$ is subordinate to $W^{u}$, for $\mu$-almost every $(\omega,x)$ we have $\nu_{(\ox)}(\eta(\ox)_{\omega})>0$. 
		Therefore, the following function
		$$p(\omega,z):=\dfrac{\Delta^{\omega}(x,z)}{\int_{\eta(\ox)_{\omega}} \Delta^{\omega}(x,y)  d\nu_{(\ox)}(y)}, \ z\in \eta(\ox)_{\omega},$$
		is well defined for $\mu$-almost every $(\omega,x)$.
		
		Suppose we have that $\mu_{\eta(\ox)}\ll \nu_{(\ox)}$ for $\mu$-almost $(\ox)$. 
		Then there exist some function $\rho$ such that $d\mu_{\eta(\ox)}=\rho(\om,z) \nu_{(\ox)}$ for $\nu_{(\ox)}$-almost every 
		$z \in \eta(\ox)_{\om}$. 
		By the argument in the proof of Theorem A, the function $\rho$ must satisfy $\int_{\eta(\ox)_{\om}}\rho(\om,z)d\nu_{(\ox)}(z)=1$ and $\rho(\Phi(\om,z))J^{u}(\om,z)\rho(\om,z)^{-1}$ must be constant on $(\Phi^{-1}\eta)(\ox)_{\om}$.
		So, one can guess that $\rho=p$.
		
		The function $p$ is measurable on $\widetilde{\cS}$ by the measurability of $q$ with standard arguments. 
		So, we can define a probability measure $\upsilon$ on $\widetilde{\cS}$ by 
		$$\upsilon(K)=\int(\int_{(K\cap \eta(\ox))_{\omega}} p(\oy) dv_{(\ox)}(y))d\mu(\ox) $$
		for each $K$ measurable. 
		By construction, the measure $\upsilon$ and $\mu$ project to the same measure on the quotient space $\widetilde{\cS}/\eta$, and the conditional measure of $\upsilon$ on each element of $\eta$, which we denote by $\upsilon_{\eta(\ox)}$, is absolutely continuous with respect to $\nu_{(\ox)}$ with the densities function $p$.
		It suffices to prove that $\upsilon=\mu$.
		
		For each $n\in \N$, by the construction we have
		$$\upsilon_{\eta(\ox)} ((\Phi^{-n}\eta)(\ox)_{\omega})=
		\dfrac{\int_{(\Phi^{-n}\eta)(\ox)_{\omega}} \Delta^{\om}(x,y)d\nu_{(\ox)}(y)}{\int_{\eta(\ox)_{\omega}} \Delta^{\om}(x,y)d\nu_{(\ox)}(y)}.$$
		By the change of variables formula (recall that $f^{-n}_{\te^n\omega}(y)$ is the unique point such that 
		$f^n_{\om}(f^{-n}_{\te^n\omega}(y))=y$),
		\begin{align*}
	     &\int_{(\Phi^{-n}\eta)(\ox)_{\omega}} \Delta^{\om}(x,z)d\nu_{(\ox)}(z)\\
		=&\int_{f^{-n}_{\te^n\om}(\eta(\Phi^{n}(\ox))_{\te^{n}\omega})} \Delta^{\om}(x,z)d\nu_{(\ox)}(z)\\
		=&\int_{\eta(\Phi^{n}(\ox))_{\te^{n}\omega}} \Delta^{\om}(x,f^{-n}_{\te^n\om}(y))d((f^n_{\om})_{\ast}\nu_{(\ox)})(y)\\
		=&\int_{\eta(\Phi^{n}(\ox))_{\te^{n}\omega}} \dfrac{\Delta^{\om}(x,f^{-n}_{\te^n\om}(y))}{\det(Df^n_{\om}(f^{-n}_{\te^n\om}(y))|E^{u}(\om,f^{-n}_{\te^n\om}(y)))}
		d\nu_{\Phi^{n}(\ox)}(y)\\
		=&\dfrac{1}{\det(Df^n_{\om}(x)|E^{u}(\ox))} \int_{\eta(\Phi^{n}(\ox))_{\te^{n}\omega}} \Delta^{\te^n\om}(f^n_\om x,y)d\nu_{\Phi^{n}(\ox)}(y).
		\end{align*}
	    Let 
	    $L(\ox):=\int_{\eta(\ox)_{\omega}} \Delta^{\om}(x,y)d\nu_{(\ox)}(y)$, then we have
	    $$\upsilon_{\eta(\ox)} ((\Phi^{-n}\eta)(\ox)_{\omega})=
	    \dfrac{L(\Phi^n(\ox))}{L(\ox)}\cdot
	    \dfrac{1}{\det(Df^n_{\om}(x)|E^{u}(\ox))}.$$
		Using the same argument in the proof of Theorem A, for each $n>0$ we have
		\begin{equation}\label{eq:S7}
			\int -\log \upsilon_{\eta(\ox)} ((\Phi^{-n}\eta)(\ox)_{\omega})d\mu(\ox)=\int \log\det(Df^n_{\omega}(\ox)|E^{u}(\ox))d\mu(\ox).
		\end{equation}
		Note that the main assumption of the lemma tells us
		\begin{equation}\label{eq:S8}
			\begin{aligned}
				 \int -\log \mu_{\eta(\ox)} ((\Phi^{-n}\eta)(\ox)_{\omega})&d\mu(\ox)=H_{\mu}(\Phi^{-n}\eta|\eta)\\
			    =&\int \log\det(Df^n_{\omega}(\ox)|E^{u}(\ox))d\mu(\ox).
			\end{aligned}
		\end{equation}

		Consider the $\widetilde{\cS}/(\Phi^{-n}\eta)$-measurable function
		$$\psi_n(\ox)=\dfrac{\upsilon_{\eta(\ox)} ((\Phi^{-n}\eta)(\ox)_{\omega})}{\mu_{\eta(\ox)} ((\Phi^{-n}\eta)(\ox)_{\omega})},$$ 
		which is defined for $\mu$-almost everywhere. 
		Together with (\ref{eq:S7}) and (\ref{eq:S8}) we have $\int \log \psi_n d\mu=0$. 
		Let $\mathbb{B}_n$ be the sub-$\sigma$-algebra of measurable subsets that are unions of elements of $\Phi^{-n}\eta$, and let $\mu_n$ and $\upsilon_n$ be the restriction of $\mu$ and $\upsilon$ on $\mathbb{B}_n$ respectively. 
		Decompose $\upsilon_n=\upsilon_{n}^{\mu}+\upsilon^{\perp}_n$ where $\upsilon_{n}^{\mu}\ll \mu_n$ and $ \upsilon^{\perp}_n $ is mutually singular with $\mu_n$ ($\upsilon^{\perp}_n$ can be strictly positive), then $\psi_n=\frac{d\upsilon_{n}^{\mu}}{d\mu_{n}}$. 
		Thus, we have $\int \psi_n d\mu\leq 1$.
		
		By Jensen's inequality we have
		$$0=\int \log \psi_n d\mu \leq \log\int \psi_n d\mu\leq 0$$
		and the equality holds if and only if  $\psi_n$ is constant for $\mu$-almost everywhere. 
		So, $\psi_n\equiv1$ for $\mu$-almost every $(\ox)$.
		This implies $\mu$ and $\upsilon$ coincide on $\mathbb{B}_n$. Since $\Phi^{-n}\eta$ partition into points modulo $\mu$-measure 0, one has  $\mu=\upsilon$. This completes the proof.
	\end{proof}
	
	\begin{proof}[proof of Theorem \ref{Thm:B} (ergodic case)]
		By Lemma \ref{Lem:last}, it suffices to show that 
		$$H_{\mu}(\Phi^{-1}\eta|\eta)=h(\mu|\bP).$$
        Note that $h(\mu|\bP)\geq H_{\mu}(\Phi^{-1}\eta|\eta)$, we only need to prove the reverse inequality holds.      
		We will use the following lemma which needed that $A$ is visible:
		\begin{lemma}[\cite{Li13} Proposition 4.2]\label{Lem:visivle}
			Assume that $A$ is visible. 
			Let $\tau:A\rightarrow (0,+\infty)$ be a measurable function satisfied $\tau(\Phi^{\pm}(\ox))\leq e^{2\varepsilon_{1}}\Phi(\ox)$ for $\mu$-almost every $(\ox)$.
		    Then, there exists a measurable partition $\cP$ of $A$ such that
			$H_{\mu}(\cP|\eta_{0})<\infty$ and $$((\bigvee_{i=0}^{\infty}\Phi^{i}\cP)(\ox))_{\omega}\subset \bigcap_{i=0}^{\infty} f^{i}_{\te^{-i}\om}B(f^{-i}_{\om}x,\tau(\Phi^{-i}(\ox)))$$ 
            for $\mu$-almost every $(\ox)$.
		\end{lemma}
  
		Let $\tau=\delta\ell(\ox)^{-2}$.
		Using Lemma \ref{Lem:visivle}, there exists a partition $\cP$ of $A$ with finite conditional entropy.
		Let $\cP^{+}:=\bigvee_{n\geq 0}\Phi^{n}\cP$, for $\mu$-almost every $(\ox)$ one has 
		$$\Lnorm{\fo^{-n}y-\fo^{-n}x}{\Phi^{-n}(\ox)}\leq \delta \ell(\Phi^{-n}(\ox))^{-1} $$
        for any $y \in \cP^{+}(\ox)_{\om}$ and any $ n\in \N$.
		Then, by Lemma \ref{Lem:CUnstable} one has
		$\cP^{+}(\ox)_{\om}\subset W^{u}_{\delta}(\ox)$ (note that we assume $E^{c}=\{0\}$). 
		
		For each $r>0$. Choose a countable partition $\cP_0$ of $A$ with finite conditional entropy such that
		$h(\mu|\bP)<h(\mu|\bP,\cP_0)-r$.
		Let $\cQ:= \cP\vee\{\widetilde{\cS}\cap A,A-\widetilde{\cS}\}\vee \cP_0$. Then, we also have $h(\mu|\bP)<h(\mu|\bP,\cQ)-r$. Let $ \cQ^{+}=\bigvee_{n\geq 0}\Phi^{n}\cQ$. 
		Use the fact that $\cP^{+}(\ox)_{\om}\subset W^{u}_{\delta}(\ox)$ and the construction of $\eta$ we can check that $\eta(\ox)_{\om}\supset \cQ^{+}(\ox)_{\om}$ for $\mu$-almost every $(\ox)$. 
		This implies that $\eta \leq \cQ^{+}\vee \eta_0$.
		Then, we have
		\begin{align*}
			H_{\mu}(\Phi^{-1}\eta|\eta)&=\frac{1}{n} H_{\mu}(\Phi^{-n}\eta|\eta)\\
			&\geq \frac{1}{n} H_{\mu}(\bigvee_{j=0}^{n} \Phi^{-j}\cQ|\eta)-\frac{1}{n} H_{\mu}(\bigvee_{j=0}^{n} \Phi^{-j}\cQ|\Phi^{-n}\eta)\\
			&\geq \frac{1}{n} H_{\mu}(\bigvee_{j=0}^{n} \Phi^{-j}\cQ|\cQ^{+}\vee \eta_0)-
			      \frac{1}{n} H_{\mu}(\bigvee_{j=0}^{n} \Phi^{-j}\cQ|\Phi^{-n}\eta)\\
			&\geq h(\mu|\bP,\cQ)-\frac{1}{n} H_{\mu}(\bigvee_{j=0}^{n} \Phi^{-j}\cQ|\Phi^{-n}\eta),
		\end{align*}
		the second term in above can be shown to be $<r$ for large $n$ since $\bigvee_{n\geq0}\Phi^{-n} \eta$ partition into points $\mu$-modulo $0$. Therefore, we have that 
		$$H_{\mu}(\Phi^{-1}\eta|\eta)\geq   h(\mu|\bP,\cQ)-r>h(\mu|\bP)-2r.$$
		By arbitrarily of $r$, we have $ H_{\mu}(\Phi^{-1}\eta|\eta)=h(\mu|\bP)$.
		This completes the proof of Theorem B in the ergodic case.
	\end{proof} 
	\section{The non-ergodic case}\label{SEC:NON-ergodic}
	The case when $\mu$ is not ergodic can be proved in the same way with the following modifications.
	
	In the non-ergodic case of the Multiplicative Ergodic Theorem, we consider a measurable $\Phi$-invariant function $ \lambda_{\alpha}$ with $0>\lambda_{\alpha}>l_{\alpha}$. 
	Then there is a measurable function $r:\Gamma\rightarrow \N$ such that for every $(\ox)\in \Gamma$, there are $r(\ox)$ Lyapunov exponents
	$$\lambda_1(\ox)\geq \cdots \geq \lambda_{r(\ox)}(\ox)(> \lambda_{\alpha}(\ox))$$ 
	and a unique splitting
	$$\cB=E_1(\ox)\oplus \cdots E_{r(\ox)}(\ox) \oplus F(\ox)$$
	satisfying the properties (a)-(d) in Theorem \ref{MET}.
	Moreover, $\lambda_{i}(\ox)$ and $m_i(\ox):=\dim E_{i}(\ox)$ are measurable functions on $\{(\ox):r(\ox)\geq i\}.$ 
		
	The unstable, center, and stable subspace defined in Section \ref{SEC:Lyapunov norms} are still applicable. Let $\lambda^{\pm}(\ox)$ be as before, which are measurable functions with positive (negative) values but not be necessarily bounded away from $0$. 
	For $m,n\in \N$ and $p,q\in \Z^{+}$, let
	\begin{align*}
		\Gamma(m,n,p,q):=\{(\ox):&\dim E^{u}(\ox)=m, \dim E^{c}(\ox)=n; \\
		&\lambda^{+}(\ox)\geq \frac{1}{p}, \lambda^{-}(\ox)\leq -\frac{1}{q}\}.
	\end{align*}
	Each $\Gamma(m,n,p,q)$ is either empty or $\Phi$-invariant, and 
	$$\Gamma=\bigcup_{m,n,p,q}\Gamma(m,n,p,q).$$
	The limit in Proposition \ref{Prop:DetL} also exists for $\mu$-almost every $(\ox)$, but the equality becomes
	\begin{align*}
		\int_{\Gamma(m,n,p,q)} \lim_{n\rightarrow \infty} \frac{1}{n}\log
		\det(Df^{n}_{\omega}(x)|E^{u}(\omega,x)) d\mu(\ox)\\
		=\int_{\Gamma(m,n,p,q)} 
  \sum_{i} m_i(\ox)\lambda^{+}_{i}(\ox)d\mu(\ox)<\infty.
	\end{align*}
	Moreover, other results in Section \ref{SEC:3} and Section \ref{SEC:4} also hold on $\Gamma(m,n,p,q)$.
	
	Note that $\Gamma(m,n,p,q)$ is not disjoint, so we define
	$$\bar{\Gamma}(m,n,p,q)=\Gamma(m,n,p,q)\setminus \Gamma(m,n,p-1,q-1).$$
	Then, $\bar{\Gamma}$ is disjoint, and $\Gamma$ is a countable union of $\bar{\Gamma}(m,n,p,q)$. 
	As in \cite{Young17}, we can see that $\bar{\Gamma}$ as an alternative to ergodic decomposition .
	
	In non-ergodic cases, it may happen that the set $\widetilde{\cS}$ which is defined in Section \ref{SEC:5} does not have full $\mu$-measure set. However, it is no problem that since there is at most a countable union of disjoint sets of the form $\widetilde{\cS}$ will cover a full $\mu$-measure set.
	
	
	\printbibliography

@book {Arnold98,
	AUTHOR = {Arnold, Ludwig},
	TITLE = {Random dynamical systems},
	SERIES = {Springer Monographs in Mathematics},
	PUBLISHER = {Springer-Verlag, Berlin},
	YEAR = {1998},
	PAGES = {xvi+586},
	ISBN = {3-540-63758-3},
	MRCLASS = {37Hxx (37-02 60H10)},
	MRNUMBER = {1723992},
	MRREVIEWER = {Yuri Kifer},
	DOI = {10.1007/978-3-662-12878-7},
	URL = {https://doi.org/10.1007/978-3-662-12878-7},
}

@article {Simion16,
	AUTHOR = {Ara\'{u}jo, V\'{\i}tor and Bufetov, Alexander I. and Filip, Simion},
	TITLE = {On {H}\"{o}lder-continuity of {O}seledets subspaces},
	JOURNAL = {J. Lond. Math. Soc. (2)},
	FJOURNAL = {Journal of the London Mathematical Society. Second Series},
	VOLUME = {93},
	YEAR = {2016},
	NUMBER = {1},
	PAGES = {194--218},
	ISSN = {0024-6107},
	MRCLASS = {37D25 (37A50 37B40 37C40 37H15)},
	MRNUMBER = {3455789},
	MRREVIEWER = {Lin Shu},
	DOI = {10.1112/jlms/jdv057},
	URL = {https://doi.org/10.1112/jlms/jdv057},
}

@book {Liu95,
	AUTHOR = {Liu, Peidong and Qian, Min},
	TITLE = {Smooth ergodic theory of random dynamical systems},
	SERIES = {Lecture Notes in Mathematics},
	VOLUME = {1606},
	PUBLISHER = {Springer-Verlag, Berlin},
	YEAR = {1995},
	PAGES = {xii+221},
	ISBN = {3-540-60004-3},
	MRCLASS = {58F11 (60H10)},
	MRNUMBER = {1369243},
	MRREVIEWER = {Vadim A. Ka\u{\i}manovich},
	DOI = {10.1007/BFb0094308},
	URL = {https://doi.org/10.1007/BFb0094308},
}

@article {Liu98,
	AUTHOR = {Bahnm\"{u}ller, J\"{o}rg and Liu, Peidong},
	TITLE = {Characterization of measures satisfying the {P}esin entropy
	formula for random dynamical systems},
	JOURNAL = {J. Dynam. Differential Equations},
	FJOURNAL = {Journal of Dynamics and Differential Equations},
	VOLUME = {10},
	YEAR = {1998},
	NUMBER = {3},
	PAGES = {425--448},
	ISSN = {1040-7294},
	MRCLASS = {58F11 (28D20 34D08)},
	MRNUMBER = {1646606},
	MRREVIEWER = {Viviane Baladi},
	DOI = {10.1023/A:1022653229891},
	URL = {https://doi.org/10.1023/A:1022653229891},
}

@book {Crauel02,
	AUTHOR = {Crauel, Hans},
	TITLE = {Random probability measures on {P}olish spaces},
	SERIES = {Stochastics Monographs},
	VOLUME = {11},
	PUBLISHER = {Taylor \& Francis, London},
	YEAR = {2002},
	PAGES = {xvi+118},
	ISBN = {0-415-27387-0},
	MRCLASS = {60-02 (60B05 60G57)},
	MRNUMBER = {1993844},
	MRREVIEWER = {Natesan Renganathan},
}

@book {Pesin07,
	AUTHOR = {Barreira, Luis and Pesin, Yakov},
	TITLE = {Nonuniform hyperbolicity},
	SERIES = {Encyclopedia of Mathematics and its Applications},
	VOLUME = {115},
	NOTE = {Dynamics of systems with nonzero Lyapunov exponents},
	PUBLISHER = {Cambridge University Press, Cambridge},
	YEAR = {2007},
	PAGES = {xiv+513},
	ISBN = {978-0-521-83258-8},
	MRCLASS = {37D25 (34D08 37A25 37C40 37D10)},
	MRNUMBER = {2348606},
	MRREVIEWER = {Ian Melbourne},
	DOI = {10.1017/CBO9781107326026},
	URL = {https://doi.org/10.1017/CBO9781107326026},
}

@article {Ledrappier82,
	AUTHOR = {Ledrappier, Fran\c{c}ois and Strelcyn, Jean-Marie},
	TITLE = {A proof of the estimation from below in {P}esin's entropy
	formula},
	JOURNAL = {Ergodic Theory Dynam. Systems},
	FJOURNAL = {Ergodic Theory and Dynamical Systems},
	VOLUME = {2},
	YEAR = {1982},
	NUMBER = {2},
	PAGES = {203--219 (1983)},
	ISSN = {0143-3857},
	MRCLASS = {58F11 (28D20 34C35)},
	MRNUMBER = {693976},
	MRREVIEWER = {M. I. Brin},
	DOI = {10.1017/S0143385700001528},
	URL = {https://doi.org/10.1017/S0143385700001528},
}

@article {Young02,
	AUTHOR = {Young, Lai-Sang},
	TITLE = {What are {SRB} measures, and which dynamical systems have
	them?},
	NOTE = {Dedicated to David Ruelle and Yasha Sinai on the occasion of
	their 65th birthdays},
	JOURNAL = {J. Statist. Phys.},
	FJOURNAL = {Journal of Statistical Physics},
	VOLUME = {108},
	YEAR = {2002},
	NUMBER = {5-6},
	PAGES = {733--754},
	ISSN = {0022-4715},
	MRCLASS = {37Dxx (37A25 37C40)},
	MRNUMBER = {1933431},
	MRREVIEWER = {Henry van den Bedem},
	DOI = {10.1023/A:1019762724717},
	URL = {https://doi.org/10.1023/A:1019762724717},
}

@article {Blumenthal16,
	AUTHOR = {Blumenthal, Alex},
	TITLE = {A volume-based approach to the multiplicative ergodic theorem
	on {B}anach spaces},
	JOURNAL = {Discrete Contin. Dyn. Syst.},
	FJOURNAL = {Discrete and Continuous Dynamical Systems. Series A},
	VOLUME = {36},
	YEAR = {2016},
	NUMBER = {5},
	PAGES = {2377--2403},
	ISSN = {1078-0947},
	MRCLASS = {37H15 (37L30 46B07)},
	MRNUMBER = {3485402},
	MRREVIEWER = {Ricardo F. Vila Freyer},
	DOI = {10.3934/dcds.2016.36.2377},
	URL = {https://doi.org/10.3934/dcds.2016.36.2377},
}

@article {Young85,
	AUTHOR = {Ledrappier, Fran\c{c}ois and Young, Lai-Sang},
	TITLE = {The metric entropy of diffeomorphisms. {I}. {C}haracterization
	of measures satisfying {P}esin's entropy formula},
	JOURNAL = {Ann. of Math. (2)},
	FJOURNAL = {Annals of Mathematics. Second Series},
	VOLUME = {122},
	YEAR = {1985},
	NUMBER = {3},
	PAGES = {509--539},
	ISSN = {0003-486X},
	MRCLASS = {58F11 (58F15)},
	MRNUMBER = {819556},
	MRREVIEWER = {D. Newton},
	DOI = {10.2307/1971328},
	URL = {https://doi.org/10.2307/1971328},
}

@article {Young17,
	AUTHOR = {Blumenthal, Alex and Young, Lai-Sang},
	TITLE = {Entropy, volume growth and {SRB} measures for {B}anach space
	mappings},
	JOURNAL = {Invent. Math.},
	FJOURNAL = {Inventiones Mathematicae},
	VOLUME = {207},
	YEAR = {2017},
	NUMBER = {2},
	PAGES = {833--893},
	ISSN = {0020-9910},
	MRCLASS = {37H15 (37C40)},
	MRNUMBER = {3595937},
	MRREVIEWER = {Zhiming Li},
	DOI = {10.1007/s00222-016-0678-0},
	URL = {https://doi.org/10.1007/s00222-016-0678-0},
}

@article {Young16,
	AUTHOR = {Blumenthal, Alex and Young, Lai-Sang},
	TITLE = {Absolute continuity of stable foliations for mappings of
	{B}anach spaces},
	JOURNAL = {Comm. Math. Phys.},
	FJOURNAL = {Communications in Mathematical Physics},
	VOLUME = {354},
	YEAR = {2017},
	NUMBER = {2},
	PAGES = {591--619},
	ISSN = {0010-3616},
	MRCLASS = {37L40 (37C05 37C40)},
	MRNUMBER = {3663618},
	MRREVIEWER = {Christian P\"{o}tzsche},
	DOI = {10.1007/s00220-017-2912-z},
	URL = {https://doi.org/10.1007/s00220-017-2912-z},
}

@article {Froyland18,
	AUTHOR = {Dragi\v{c}evi\'{c}, Davor and Froyland, Gary},
	TITLE = {H\"{o}lder continuity of {O}seledets splittings for
	semi-invertible operator cocycles},
	JOURNAL = {Ergodic Theory Dynam. Systems},
	FJOURNAL = {Ergodic Theory and Dynamical Systems},
	VOLUME = {38},
	YEAR = {2018},
	NUMBER = {3},
	PAGES = {961--981},
	ISSN = {0143-3857},
	MRCLASS = {37H15 (37A30 37C40 37D25)},
	MRNUMBER = {3784250},
	MRREVIEWER = {Doan Thai Son},
	DOI = {10.1017/etds.2016.55},
	URL = {https://doi.org/10.1017/etds.2016.55},
}

@article {Froyland13,
	AUTHOR = {Froyland, Gary and Lloyd, Simon and Quas, Anthony},
	TITLE = {A semi-invertible {O}seledets theorem with applications to
	transfer operator cocycles},
	JOURNAL = {Discrete Contin. Dyn. Syst.},
	FJOURNAL = {Discrete and Continuous Dynamical Systems. Series A},
	VOLUME = {33},
	YEAR = {2013},
	NUMBER = {9},
	PAGES = {3835--3860},
	ISSN = {1078-0947},
	MRCLASS = {37H15 (37A30 37L55)},
	MRNUMBER = {3038042},
	MRREVIEWER = {Luciana A. Alves},
	DOI = {10.3934/dcds.2013.33.3835},
	URL = {https://doi.org/10.3934/dcds.2013.33.3835},
}

@phdthesis{Doan09,
	author = {Doan, Thai Son}, 
	title = {Lyapunov exponents for random dynamical systems},
	school = {Fakult\"{a}t Mathematik und Naturwissenschaften der Technischen Universit\"{a}t Dresden},
	year = 2009,
}

@article {Varzaneh21,
	AUTHOR = {Varzaneh, Mazyar Ghani and Riedel, Sebastian},
     TITLE = {Oseledets {S}plitting and {I}nvariant {M}anifolds on {F}ields
              of {B}anach {S}paces},
   JOURNAL = {J. Dynam. Differential Equations},
  FJOURNAL = {Journal of Dynamics and Differential Equations},
    VOLUME = {35},
      YEAR = {2023},
    NUMBER = {1},
     PAGES = {103--133},
      ISSN = {1040-7294},
   MRCLASS = {37H15 (37B55 37L55)},
  MRNUMBER = {4549812},
       DOI = {10.1007/s10884-021-09969-1},
       URL = {https://doi.org/10.1007/s10884-021-09969-1},
}

@article {Quas12,
	AUTHOR = {Gonz\'{a}lez-Tokman, Cecilia and Quas, Anthony},
	TITLE = {A semi-invertible operator {O}seledets theorem},
	JOURNAL = {Ergodic Theory Dynam. Systems},
	FJOURNAL = {Ergodic Theory and Dynamical Systems},
	VOLUME = {34},
	YEAR = {2014},
	NUMBER = {4},
	PAGES = {1230--1272},
	ISSN = {0143-3857},
	MRCLASS = {47A35 (37A30 37H15)},
	MRNUMBER = {3227155},
	MRREVIEWER = {Wojciech Bartoszek},
	DOI = {10.1017/etds.2012.189},
	URL = {https://doi.org/10.1017/etds.2012.189},
}

@book {Kato95,
	AUTHOR = {Kato, Tosio},
	TITLE = {Perturbation theory for linear operators},
	SERIES = {Classics in Mathematics},
	NOTE = {Reprint of the 1980 edition},
	PUBLISHER = {Springer-Verlag, Berlin},
	YEAR = {1995},
	PAGES = {xxii+619},
	ISBN = {3-540-58661-X},
	MRCLASS = {47A55 (46-00 47-00)},
	MRNUMBER = {1335452},
}

@article {Li12,
	AUTHOR = {Li, Zhiming and Shu, Lin},
	TITLE = {The metric entropy of random dynamical systems in a {B}anach
	space: {R}uelle inequality},
	JOURNAL = {Ergodic Theory Dynam. Systems},
	FJOURNAL = {Ergodic Theory and Dynamical Systems},
	VOLUME = {34},
	YEAR = {2014},
	NUMBER = {2},
	PAGES = {594--615},
	ISSN = {0143-3857},
	MRCLASS = {37H15 (37A55 37C45)},
	MRNUMBER = {3233706},
	MRREVIEWER = {Ricardo G\'{o}mez},
	DOI = {10.1017/etds.2012.138},
	URL = {https://doi.org/10.1017/etds.2012.138},
}

@article {Li13,
	AUTHOR = {Li, Zhiming and Shu, Lin},
	TITLE = {The metric entropy of random dynamical systems in a {H}ilbert
	space: characterization of invariant measures satisfying
	{P}esin's entropy formula},
	JOURNAL = {Discrete Contin. Dyn. Syst.},
	FJOURNAL = {Discrete and Continuous Dynamical Systems. Series A},
	VOLUME = {33},
	YEAR = {2013},
	NUMBER = {9},
	PAGES = {4123--4155},
	ISSN = {1078-0947},
	MRCLASS = {37C40 (28D20)},
	MRNUMBER = {3038055},
	MRREVIEWER = {Peng Sun},
	DOI = {10.3934/dcds.2013.33.4123},
	URL = {https://doi.org/10.3934/dcds.2013.33.4123},
}

@article {Lian10,
	AUTHOR = {Lian, Zeng and Lu, Kening},
	TITLE = {Lyapunov exponents and invariant manifolds for random
	dynamical systems in a {B}anach space},
	JOURNAL = {Mem. Amer. Math. Soc.},
	FJOURNAL = {Memoirs of the American Mathematical Society},
	VOLUME = {206},
	YEAR = {2010},
	NUMBER = {967},
	PAGES = {vi+106},
	ISSN = {0065-9266},
	ISBN = {978-0-8218-4656-8},
	MRCLASS = {37H15 (34D08 37A55 37C45 37D10 37H10 37L25)},
	MRNUMBER = {2674952},
	MRREVIEWER = {Bj\"{o}rn Schmalfuss},
	DOI = {10.1090/S0065-9266-10-00574-0},
	URL = {https://doi.org/10.1090/S0065-9266-10-00574-0},
}

@article {Lian11,
	AUTHOR = {Lian, Zeng and Young, Lai-Sang},
	TITLE = {Lyapunov exponents, periodic orbits and horseshoes for
	mappings of {H}ilbert spaces},
	JOURNAL = {Ann. Henri Poincar\'{e}},
	FJOURNAL = {Annales Henri Poincar\'{e}. A Journal of Theoretical and
	Mathematical Physics},
	VOLUME = {12},
	YEAR = {2011},
	NUMBER = {6},
	PAGES = {1081--1108},
	ISSN = {1424-0637},
	MRCLASS = {37H15 (37D25 37L55)},
	MRNUMBER = {2823209},
	MRREVIEWER = {Jairo Bochi},
	DOI = {10.1007/s00023-011-0100-9},
	URL = {https://doi.org/10.1007/s00023-011-0100-9},
}

@article {Lian16,
	AUTHOR = {Lian, Zeng and Liu, Peidong and Lu, Kening},
	TITLE = {S{RB} measures for a class of partially hyperbolic attractors
	in {H}ilbert spaces},
	JOURNAL = {J. Differential Equations},
	FJOURNAL = {Journal of Differential Equations},
	VOLUME = {261},
	YEAR = {2016},
	NUMBER = {2},
	PAGES = {1532--1603},
	ISSN = {0022-0396},
	MRCLASS = {37L40 (37C40 37D20 37D30)},
	MRNUMBER = {3494405},
	MRREVIEWER = {Harald Walter Proppe},
	DOI = {10.1016/j.jde.2016.04.006},
	URL = {https://doi.org/10.1016/j.jde.2016.04.006},
}

@article {Lian17,
	AUTHOR = {Lian, Zeng and Liu, Peidong and Lu, Kening},
	TITLE = {Existence of {SRB} measures for a class of partially
	hyperbolic attractors in {B}anach spaces},
	JOURNAL = {Discrete Contin. Dyn. Syst.},
	FJOURNAL = {Discrete and Continuous Dynamical Systems. Series A},
	VOLUME = {37},
	YEAR = {2017},
	NUMBER = {7},
	PAGES = {3905--3920},
	ISSN = {1078-0947},
	MRCLASS = {37D45 (37C40)},
	MRNUMBER = {3639444},
	MRREVIEWER = {Xu Zhang},
	DOI = {10.3934/dcds.2017164},
	URL = {https://doi.org/10.3934/dcds.2017164},
}

@article {Lian20,
	AUTHOR = {Lian, Zeng and Ma, Xiao},
	TITLE = {Existence of periodic orbits and horseshoes for mappings in a
	separable {B}anach space},
	JOURNAL = {J. Differential Equations},
	FJOURNAL = {Journal of Differential Equations},
	VOLUME = {269},
	YEAR = {2020},
	NUMBER = {12},
	PAGES = {11694--11738},
	ISSN = {0022-0396},
	MRCLASS = {37D25},
	MRNUMBER = {4152222},
	MRREVIEWER = {Alex Michael Blumenthal},
	DOI = {10.1016/j.jde.2020.08.004},
	URL = {https://doi.org/10.1016/j.jde.2020.08.004},
}

@article {Lucas17,
	AUTHOR = {Backes, Lucas and Dragi\v{c}evi\'{c}, Davor},
	TITLE = {Periodic approximation of exceptional {L}yapunov exponents for
	semi-invertible operator cocycles},
	JOURNAL = {Ann. Acad. Sci. Fenn. Math.},
	FJOURNAL = {Annales Academi\ae  Scientiarum Fennic\ae . Mathematica},
	VOLUME = {44},
	YEAR = {2019},
	NUMBER = {1},
	PAGES = {183--209},
	ISSN = {1239-629X},
	MRCLASS = {37H15 (37A20 37D25)},
	MRNUMBER = {3919132},
	MRREVIEWER = {Zhenghe Zhang},
	DOI = {10.5186/aasfm.2019.4410},
	URL = {https://doi.org/10.5186/aasfm.2019.4410},
}

@incollection {M83,
	AUTHOR = {Ma\~{n}\'{e}, Ricardo},
	TITLE = {Lyapounov exponents and stable manifolds for compact
	transformations},
	BOOKTITLE = {Geometric dynamics ({R}io de {J}aneiro, 1981)},
	SERIES = {Lecture Notes in Math.},
	VOLUME = {1007},
	PAGES = {522--577},
	PUBLISHER = {Springer, Berlin},
	YEAR = {1983},
	MRCLASS = {58F15},
	MRNUMBER = {730286},
	MRREVIEWER = {Charles C. Conley},
	DOI = {10.1007/BFb0061433},
	URL = {https://doi.org/10.1007/BFb0061433},
}

@article {Oseledec68,
	AUTHOR = {Oseledets, Valery I.},
	TITLE = {A multiplicative ergodic theorem. {C}haracteristic {L}japunov,
	exponents of dynamical systems},
	JOURNAL = {Trudy Moskov. Mat. Ob\v{s}\v{c}.},
	FJOURNAL = {Trudy Moskovskogo Matemati\v{c}eskogo Ob\v{s}\v{c}estva},
	VOLUME = {19},
	YEAR = {1968},
	PAGES = {179--210},
	ISSN = {0134-8663},
	MRCLASS = {28.70 (34.00)},
	MRNUMBER = {0240280},
	MRREVIEWER = {J\'{o}zsef Sz\"{u}cs},
}

@article {Ruelle82,
	AUTHOR = {Ruelle, David},
	TITLE = {Characteristic exponents and invariant manifolds in {H}ilbert
	space},
	JOURNAL = {Ann. of Math. (2)},
	FJOURNAL = {Annals of Mathematics. Second Series},
	VOLUME = {115},
	YEAR = {1982},
	NUMBER = {2},
	PAGES = {243--290},
	ISSN = {0003-486X},
	MRCLASS = {58F15 (58F11)},
	MRNUMBER = {647807},
	MRREVIEWER = {D. Newton},
	DOI = {10.2307/1971392},
	URL = {https://doi.org/10.2307/1971392},
}

@article {Thieullen87,
	AUTHOR = {Thieullen, Philippe},
	TITLE = {Fibr\'{e}s dynamiques asymptotiquement compacts. {E}xposants de
	{L}yapounov. {E}ntropie. {D}imension},
	JOURNAL = {Ann. Inst. H. Poincar\'{e} Anal. Non Lin\'{e}aire},
	FJOURNAL = {Annales de l'Institut Henri Poincar\'{e}. Analyse Non Lin\'{e}aire},
	VOLUME = {4},
	YEAR = {1987},
	NUMBER = {1},
	PAGES = {49--97},
	ISSN = {0294-1449},
	MRCLASS = {58F11 (28D20 34D05 58F15)},
	MRNUMBER = {877991},
	MRREVIEWER = {Massimo Campanino},
	URL = {http://www.numdam.org/item?id=AIHPC_1987__4_1_49_0},
}

@book {Castaing77,
	AUTHOR = {Castaing, Charles and Valadier, Michel},
	TITLE = {Convex analysis and measurable multifunctions},
	SERIES = {Lecture Notes in Mathematics, Vol. 580},
	PUBLISHER = {Springer-Verlag, Berlin-New York},
	YEAR = {1977},
	PAGES = {vii+278},
	MRCLASS = {46G99 (26A51 28A05 49A50 54C60)},
	MRNUMBER = {0467310},
	MRREVIEWER = {Vladimir L. Levin},
}

@article {Walters93,
	AUTHOR = {Walters, Peter},
	TITLE = {A dynamical proof of the multiplicative ergodic theorem},
	JOURNAL = {Trans. Amer. Math. Soc.},
	FJOURNAL = {Transactions of the American Mathematical Society},
	VOLUME = {335},
	YEAR = {1993},
	NUMBER = {1},
	PAGES = {245--257},
	ISSN = {0002-9947},
	MRCLASS = {28D05 (58F11)},
	MRNUMBER = {1073779},
	MRREVIEWER = {U. Krengel},
	DOI = {10.2307/2154267},
	URL = {https://doi.org/10.2307/2154267},
}

@article {LeY88,
AUTHOR = {Ledrappier, Fran\c{c}ois and Young, Lai-Sang},
     TITLE = {Entropy formula for random transformations},
   JOURNAL = {Probab. Theory Related Fields},
  FJOURNAL = {Probability Theory and Related Fields},
    VOLUME = {80},
      YEAR = {1988},
    NUMBER = {2},
     PAGES = {217--240},
      ISSN = {0178-8051},
   MRCLASS = {58F11 (60J05)},
  MRNUMBER = {968818},
       DOI = {10.1007/BF00356103},
       URL = {https://doi.org/10.1007/BF00356103},
       }

@article {liu99,
    AUTHOR = {Liu, Peidong},
     TITLE = {Entropy formula of {P}esin type for noninvertible random
              dynamical systems},
   JOURNAL = {Math. Z.},
  FJOURNAL = {Mathematische Zeitschrift},
    VOLUME = {230},
      YEAR = {1999},
    NUMBER = {2},
     PAGES = {201--239},
      ISSN = {0025-5874},
   MRCLASS = {37C40 (28D20 34A35 37H15 60G10)},
  MRNUMBER = {1676734},
       DOI = {10.1007/PL00004694},
       URL = {https://doi.org/10.1007/PL00004694},
}

@article {Ledrappier84,
    AUTHOR = {Ledrappier, Fran\c{c}ois},
     TITLE = {Propri\'{e}t\'{e}s ergodiques des mesures de {S}inai},
   JOURNAL = {Inst. Hautes \'{E}tudes Sci. Publ. Math.},
  FJOURNAL = {Institut des Hautes \'{E}tudes Scientifiques. Publications
              Math\'{e}matiques},
    VOLUME = {59},
      YEAR = {1984},
     PAGES = {163--188},
      ISSN = {0073-8301},
   MRCLASS = {58F11 (58F15)},
  MRNUMBER = {743818},
MRREVIEWER = {Carmen Chicone},
       URL = {http://www.numdam.org/item?id=PMIHES_1984__59__163_0},
}
 
\end{sloppypar}

\end{document}